\documentclass[11pt]{amsart}
\usepackage[utf8]{inputenc}
\usepackage{amsfonts, mathtools, nccmath, amssymb, graphicx, multicol, commath, bbm, mathrsfs, bm, subfig}
\usepackage{amsthm, mathabx}
\usepackage{esint}
\usepackage{braket}
\usepackage[foot]{amsaddr}
\usepackage{bibentry}
\usepackage{color,dsfont}
\usepackage[top=1in, bottom=1in, left=1in, right=1in]{geometry}
\usepackage{hyperref}
\hypersetup{colorlinks,breaklinks,
             linkcolor=blue,urlcolor=blue,
             anchorcolor=blue,citecolor=blue}

\usepackage{cite}

\newcommand{\restr}{%
  \,\raisebox{-.127ex}{\reflectbox{\rotatebox[origin=br]{-90}{$\lnot$}}}\,%
}
\newcommand{\R}{\mathbb{R}}

\newcommand{\N}{\mathbb{N}}

\renewcommand{\P}{\mathscr{P}}

\newcommand{\V}{\mathcal{V}}

\newcommand{\Leb}[1][]{\mathscr{L}}

\renewcommand{\O}{\mathcal{O}}
\newcommand{\E}{\mathcal{E}}

\newcommand{\D}{\mathcal{D}}

\renewcommand{\H}{\mathcal{H}}

\newcommand{\W}{\mathcal{W}}

\newcommand{\Lfun}{\mathcal{L}}

\newcommand{\U}{\mathcal{U}}

\newcommand{\eps}{\varepsilon}

\newcommand{\br}{\bm{r}}

\newcommand{\id}{\bm{i}}
\newcommand{\bmu}{\bm{\mu}}
\newcommand{\bv}{\bm{u}}
\newcommand{\bu}{\bm{u}}

\DeclareMathOperator{\supp}{supp}

\DeclareMathOperator{\tr}{tr}

\DeclareMathOperator*{\argmin}{argmin}

\DeclareMathOperator{\Lip}{Lip}
\DeclareMathOperator{\grad}{grad}

\newtheorem{theorem}{Theorem}[section]

\newtheorem{lemma}[theorem]{Lemma}
\newtheorem{proposition}[theorem]{Proposition}
\newtheorem{corollary}[theorem]{Corollary}
\theoremstyle{definition}
\newtheorem{definition}[theorem]{Definition}

\definecolor{byzantine}{rgb}{0.74, 0.2, 0.64}
\definecolor{darkgreen}{rgb}{0.1,0.6,0.1}
\definecolor{darkred}{rgb}{0.6,0,0}
\definecolor{lightgray}{rgb}{0.5,0.5,0.5}

 \newenvironment{listi}
  {\begin{list} 
 {(\roman{broj})}
{ \usecounter{broj}}
     \setlength{\labelwidth}{25pt}
  }
{   \end{list} }
\newcounter{broj}

\def\XXint#1#2#3{{\setbox0=\hbox{$#1{#2#3}{\int}$}
     \vcenter{\hbox{$#2#3$}}\kern-.5\wd0}}

\newtheorem{remarkx}[theorem]{Remark}
\newenvironment{remark}{\pushQED{\qed}\remarkx}{\popQED\endremarkx}

\numberwithin{equation}{section}

\title[A Variational perspective on the dissipative Hamiltonian structure of VFP]{A variational perspective on the dissipative Hamiltonian structure of the Vlasov-Fokker-Planck equation}
\author{Sangmin Park} 
\address{Department of Mathematical Sciences, Carnegie Mellon University, 5000 Forbes ave., Pittsburgh, PA 15213}
\email{sangminp@andrew.cmu.edu}

\begin{document}

\begin{abstract}
The Vlasov-Fokker-Planck equation describes the evolution of the probability density of the position and velocity of particles under the influence of external confinement, interaction, friction, and stochastic force. It is well-known that this equation can be formally seen as a dissipative Hamiltonian system in the Wasserstein space of probability measures. In order to better understand this geometric formalism, we introduce a time-discrete variational scheme, solutions of which converge to the solution of the Vlasov-Fokker-Planck equation as time step vanishes; in particular, this provides a new proof of the existence of a weak solution to the equation. The variational scheme combines the symplectic Euler scheme and the (degenerate) implicit steepest descent, and updates the probability density at each iteration first in the velocity variable then in the position variable.

The algorithm leverages the geometric structure of the equation, and has several desirable properties. Energy functionals involved in each variational problem are geodesically-convex, which implies the unique solvability of the problem. Furthermore, the correct dissipation of the Hamiltonian is observed at the discrete level up to higher order errors controlled by the second moments.
\end{abstract}

\maketitle


\medskip


\setcounter{tocdepth}{1}
\tableofcontents

\newpage
\subsection*{Notation}
\begin{itemize} \addtolength{\itemsep}{3pt}
\addtolength{\itemindent}{-35pt}
\item[ ] $\Leb^k$ -- the $k$-dimensional Lebesgue measure. $\Leb=\Leb^{2d}$.
\item[] $\P_2(\Omega)$ with $\Omega=\R^d$ or $\R^{2d}$ -- sets of Borel probability measures with bounded second moments; see \eqref{def:P2}
\item[] $\P_2^r(\Omega)$ with $\Omega=\R^d$ or $\R^{2d}$ -- sets of absolutely continuous Borel probability measures with bounded second moments; see \eqref{def:P2r}
\item[] $\P_2^x(\R^{2d};\sigma),\P_2^v(\R^{2d};\sigma)$ -- subspace of $\P_2(\R^{2d})$ with marginal in the $x$-variable (resp. $v$-variable) equal to $\sigma\in\P_2(\R^d)$; see \eqref{def:P2x;P2p}
\item[] $\id_x,\id_v$ -- projection operators on $\R^{2d}$, namely $\id_x(x,v)=x$ and $\id_v(x,v)=v$
\item[] $\id$ -- identity operator on $\R^{2d}$, namely $\id(x,v)=(\id_x(x,v),\id_v(x,v))=(x,v)$
\item[] $\Pi^x\mu,\Pi^v\mu$ -- $x$-marginal (resp. $v$-marginal) of $\mu\in\P_2(\R^{2d})$; see \eqref{def:PixPip}
\item[] $\mu^x,\mu^v$ -- disintegration of $\mu\in\P_2(\R^{2d})$ with respect to $\Pi^x\mu$ (resp. $\Pi^v\mu$); see \eqref{def:disintegration}
\item[] $W_2$ -- the $2$-Wasserstein distance; see \eqref{def:Wp}.
\item[] $W_{2,v}$,$W_{2,x}$ --  the $2$-Wasserstein distance with fixed $x$-marginals (resp. $v$-marginals); see \eqref{def:partial_Wass}.
\item[] $\Gamma(\mu,\nu)$ -- set of transport plans between probability measures $\mu,\nu$; see \eqref{def:Gamma}
\item[] $\Gamma^x(\mu,\nu),\Gamma^v(\mu,\nu)$ -- set of transport plans with fixed marginals between probability measures $\mu,\nu\in\P(\R^{2d})$; see \eqref{def:Gammapx}
\item[] $\Gamma_o(\mu,\nu)$ --  set of optimal transport plans $\mu,\nu \in \P_2(\R^{2d})$ for the quadratic cost; see \eqref{def:Gamma}
\item[] $\Gamma_o^x(\mu,\nu)$, $\Gamma_o^v(\mu,\nu)$ --  set of optimal transport plans with fix marginals between $\mu,\nu \in \P_2(\R^{2d})$ for the quadratic cost; see \eqref{def:Gammapxo}
\item[] $T_\mu^\nu$ -- optimal transport map for quadratic cost from $\mu\in\P_2(\Omega)$ to $\nu\in\P_2(\Omega)$ with $\Omega=\R^d$ or $\R^{2d}$
\item[] $|\partial\E|(\mu)$ -- metric slope of the functional $\E:\P_2(\R^{2d})\rightarrow(-\infty,+\infty]$ with respect to $W_{2}$ at $\mu\in \P_2(\R^{2d})$; see \eqref{def:met_slope}
\item[] $|\partial_x\E|(\mu),|\partial_v\E|(\mu)$ -- metric slope of the functional $\E:\P_2(\R^{2d})\rightarrow(-\infty,+\infty]$ with respect to $W_{2,x}$ (resp. $W_{2,v}$) at $\mu\in \P_2(\R^{2d})$; see \eqref{def:partial_slope}
\item[] $D(\E)$ -- effective domain of a functional $\E:X\rightarrow(-\infty,+\infty]$
\item[] $AC^p(I;X,m)$ -- space of $p$-absolutely continuous curves $\bmu:I\rightarrow X$ in the metric space $(X,m)$ given an interval $I$ and $p=[1,\infty)$; see Definition~\ref{def:ACcurve}
\item[] $J,S$ -- the $2d\times2d$ antisymmetric and degenerate negative definite matrices respectively; see \eqref{def:JS}

\end{itemize}

\section{Introduction}\label{sec:intro}

The gradient flow perspective allows one to consider various dissipative partial differential equations (PDEs) as ordinary differential equations (ODEs) in suitable infinite dimensional spaces. This interpretation in the setting of the Wasserstein space of probability measures is understood precisely and rigorously, and has led to fruitful applications. In their seminal work, Jordan, Kinderlehrer, and Otto \cite{JKO98} observed that the Fokker-Planck equation can be obtained as a limit of a minimizing movement scheme in the Wasserstein space. Later Otto \cite{Otto01} developed a Riemannian perspective on the Wasserstein space that enables interpreting as gradient flows a broad class of dissipative equations including the porous medium equation. The celebrated work of Ambrosio, Gigli, and Savar\'e \cite{AGS} established rigorous theory unifying these perspectives: \emph{gradient flows}, formulated in terms of the (sub-)differential structure, coincide with \emph{curves of maximal slope} which solve a metric differential inequality, and more generally corresponds to \emph{minimizing movements} obtained as limits from the time-discrete variational schemes. Furthermore, the theory provides many useful results such as asymptotic behavior and stability of solutions in terms of the geodesic-convexity of the energy functional driving the dissipation.

The Vlasov-Fokker-Planck equation (VFP) describes the evolution of the probability density of the position and velocity of particles under the influence of external confinement, interaction, friction, and stochastic force. When the interaction is prescribed by the Coloumb potential, the equation is called the Vlasov-Poisson-Fokker-Planck equation and is of great importance in plasma physics \cite{Chandra43,Vil02_review}. In case the interaction is absent the equation becomes linear and is often referred to as the kinetic Fokker-Planck equation (KFP), study of which dates back to Kolmogorov \cite{Kolmogoroff34}. The difficulties in analysis of these equations partially lie in that the diffusion is only present in the velocity variable. Indeed, the degeneracy of these PDEs played a large part in motiviating the general theory of hypoelliptic operators developed by H\"ormander \cite{Hor67}, and the hypocoercivity methods due to Villani \cite{Vil09hypo}. Instead of attempting to review the vast history of the developments regarding these equations we refer the readers to the works \cite{Hor67,Vil09hypo,HelfferNier05,AAMN21} and references therein.

It is well-known that the Vlasov-Fokker-Planck equation formally has a dissipative Hamiltonian structure (or equivalently a damped Hamiltonian structure) in the Wasserstein space, understood in the sense of Ambrosio and Gangbo \cite{AmbGan08} and Gangbo, Kim, and Pacini \cite{GanKimPac11}; see for instance \cite[Section 8]{Vil03} or bibliographical notes following \cite[Section 23]{Vil09}. It is natural to ask if this geometric perspective in the Wasserstein setting can be understood as precisely as in the case of gradient flows.

Possible connections between the geometric structure and the rates of convergence to equilibrium provide an additional motivation to rigorously understand the geometric formalism. When the potential function satisfies the Polyak-{\L}ojasiewicz inequality with constant $\lambda>0$, the corresponding damped Hamiltonian system in the Euclidean space with optimal damping converges to the equilibrium at the rate $O(e^{-c\sqrt{\lambda}t})$ with some universal constant $c>0$ \cite{AujDosRon22,ApiGinVil22}. This is often referred to as \emph{the accelerated rate of convergence} as it is much faster than the convergence rate $O(e^{-\lambda t})$ of gradient flows in the difficult case when $\lambda\ll 1$. 

In fact, the solution of the kinetic Fokker-Planck equation converges to a stationary distribution that only depends on the external confinement potential, and the exponential convergence rate has also been extensively studied \cite{Vil09,AAMN21,DR-D20,EbGuiZi19,CLW23,BolGuiMal10,BFL23}. Sharp convergence rates are of significant interest also in statistics and data science, as the underdamped Langevin dynamics, the corresponding stochastic equation to KFP, have been used successfully in sampling problems \cite{CCBJ18,ZCLBE23,LeimMat13}. An exciting recent development in this direction is the first \emph{accelerated convergence rate} established by Cao, Lu, and Wang \cite{CLW23}. They showed that, with the optimal choice of the friction parameter, the $\chi^2$-divergence decays at the accelerated rate $O(e^{-c\sqrt{\lambda}t})$ in terms of the Poincar\'e constant $\lambda$ of the invariant probability measure. Eberle and L\"orler recently observed \cite{EberleLorer24} that this rate cannot be improved in terms of $\lambda$.

While this resembles the accelerated convergence rates of damped Hamiltonian systems in the Euclidean space, the typical $\chi^2$-divergence between probability measures grows exponentially with the dimension. Thus it is desirable to establish a parallel result for the relative entropy functional, which behaves more moderately in high dimensions. On the one hand KFP is formally a damped Hamiltonian flow of the relative entropy functional in the Wasserstein space, whereas the log-Sobolev inequality can be understood as the Polyak-{\L}ojasiewicz inequality of the relative entropy functional in the Wasserstein space. Thus one might hope that making rigorous the formal geometric interpretation can shed light on the asymptotic behavior of VFP and related equations.

This paper focuses on a modest question towards better understanding the dissipative Hamiltonian formalism: 
\emph{can we construct solutions of VFP using a time-discrete scheme that is consistent with the \textbf{geometric} structure of the equation?} To answer this question, we introduce \emph{the coordinate-wise minimizing movement scheme} in the Wasserstein space. This variational scheme combines the symplectic Euler scheme and the implicit steepest descent method respectively accounting for the Hamiltonian dynamics and the (degenerate) dissipation in the velocity variable.

The novelty of this scheme is threefold:
\begin{listi}
    \item The coordinate-wise minimizing movement scheme leverages the geometric structure of the Wasserstein space. Namely, the variational scheme involves the length metric induced by the Wasserstein distance on subspaces with fixed marginals in the velocity or the position variable. Moreover, the conservative dynamics is captured by a pair of functionals arising from the Hamiltonian functional and the Poisson structure of the Wasserstein space.
    
    \item At each iteration, the energy functional involved in the variational problem is \emph{geodesically-convex} with respect to the metric, which leads to the unique solvability of the implicit minimization problem for any time step.
    \item The `correct' dissipation of the Hamiltonian is observed at the discrete level up to higher order errors controlled by the second moments.
\end{listi}

When the confinement and interaction potentials have Lipschitz gradient, we show that the discrete solution to the coordinate-wise minimizing movement scheme converges to the distributional solution of the Vlasov-Fokker-Planck equation as the step-size vanishes.

\subsection{Setting}\label{ssec:setting}
For $\Omega=\R^{d},\R^{2d}$ we denote by $\P(\Omega)$ the set of all Borel probability measures on $\Omega$. We write $\P_2(\Omega)$ to refer to the space of probability measures with bounded second moments -- i.e.
\begin{equation}\label{def:P2}
    \P_2(\Omega)=\{\sigma\in\P(\Omega):\;\int_{\Omega}|\omega|^2\,d\sigma(\omega)<+\infty\}.
\end{equation}
We denote by $\P_2^r(\Omega)$ the set of absolutely continuous probability measures -- i.e.
\begin{equation}\label{def:P2r}
    \P_2^r(\Omega)=\{\sigma\in\P_2(\Omega):\;\sigma\ll\Leb_\Omega\} \text{ where } \Leb_\Omega \text{ is the Lebesgue measure on }\Omega.
\end{equation}
Recall that the sequence $(\sigma_n)_{n\in\N}$ in $\P(\Omega)$ converges narrowly to $\mu\in\P(\Omega)$ if for each continuously bounded function $\varphi\in C_b(\R^d)$
\begin{equation}\label{def:narrow_conv}
    \lim_{n\rightarrow\infty}\int_{\Omega} \varphi(\omega)\,d\mu_n(\omega) = \int_{\Omega} \varphi(\omega)\,d\mu(\omega).
\end{equation}
Given a functional $\E:\P(\Omega)\rightarrow (-\infty,+\infty]$ we say $\E$ is lower semicontinuous with respect to the narrow convergence in some subset $\P'(\Omega)\subset\P(\Omega)$ (which we sometimes abbreviate as l.s.c. w.r.t. narrow convergence) if for every $(\mu_n)_{n\in\N}$ and $\mu_0$ in $\P'(\Omega)$ such that $\mu_n\rightharpoonup \mu_0$ narrowly,
\begin{equation}\label{def:lsc_narrowconv}
    \E(\mu_0)\leq \liminf_{n\nearrow\infty}\E(\mu_n).
\end{equation}

For probability measures  $\mu,\nu\in \P_2(\Omega)$, the $2$-Wasserstein distance $W_2$, is defined as follows: 
\begin{equation}\label{def:Wp}
        W_2(\mu,\nu) :=  \inf_{\gamma\in\Gamma(\mu,\nu)}\left(\int_{\Omega\times \Omega} |\omega_1-\omega_2|^2 \,d\gamma(\omega_1,\omega_2)\right)^{1/2}
\end{equation}
where $\Gamma(\mu,\nu)$ is set of couplings of $\mu,\nu$
\begin{equation}\label{def:Gamma}
    \Gamma(\mu,\nu)=\{\gamma\in\P(\Omega\times\Omega):\;\pi^1_\#\gamma=\mu \text{ and } \pi^2_\#\gamma=\nu\}.
\end{equation}
Henceforth we refer to $W_2$ simply as the Wasserstein distance. We denote by $\Gamma_o(\mu,\nu)$ the set of optimal couplings of $\mu,\nu$
\begin{equation}\label{def:Gammao}
    \Gamma_o(\mu,\nu)=\{\gamma\in\Gamma(\mu,\nu):\;\iint_{\Omega\times\Omega}|z-\tilde z|^2\,d\gamma(z,\tilde z)=W_2^2(\mu,\nu)\}.
\end{equation}
When there exists $T_{\mu}^{\nu}:\Omega\rightarrow\Omega$ such that $(T_{\mu}^{\nu}\times\id_\Omega)_\#\mu\in\Gamma_o(\mu,\nu)$, where $\id_\Omega:\Omega\rightarrow\Omega$ is the identity map on $\Omega$, we say $T_{\mu}^{\nu}$ is an optimal transport map from $\mu\in\P_2(\Omega)$ to $\nu\in\P_2(\Omega)$. We refer to $(\P_2(\R^{2d}),W_2)$ as the 2-Wasserstein space or simply the Wasserstein space. 

We use $z,\tilde z\in\R^{2d}$ to refer to points in $\R^{2d}$, and write $z=(x,v),\tilde z=(y,w)\in\R^d\times\R^d$ when it is convenient to separate the two $\R^d$-coordinates, with the first $d$-coordinates $x,y\in\R^d$ denoting the positions and the latter $d$-coordinates $v,w\in\R^d$ the velocities. We denote by $\Leb^k$ the $k$-dimensional Lebesgue measure on $\R^k$, and $\Leb=\Leb^{2d}$ for simplicity.

Let $\id:\R^{2d}\rightarrow\R^{2d}$ be an identity map and $\id_x,\id_v:\R^{2d}\rightarrow\R^d$ the projections onto the respective coordinates -- i.e.
\begin{equation}\label{def:id}
    \id_x(x,v)=x, \qquad \id_v(x,v)=v, \text{ and }
    \id(x,v)=(\id_x(x,v),\id_v(x,v))=(x,v).
\end{equation}

For each $\mu\in\P_2(\R^{2d})$ we introduce the following notation for their $x$- and $v$-marginals
\begin{equation}\label{def:PixPip}
    \Pi^x\mu=(\id_x)_\#\mu\in\P_2(\R^d), \qquad \Pi^v\mu=(\id_v)_{\#}\mu\in\P_2(\R^d),
\end{equation}
and we denote by $\{\mu^x\}_{x\in\R^d}$ (resp. $\{\mu^v\}_{v\in\R^d}$) its disintegration with respect to $\Pi^x\mu$ (resp. $\Pi^v\mu$)  -- i.e. the $\Pi^x\mu$-a.e. uniquely determined Borel family of probability measures such that for every Borel map $f:\R^{2d}\rightarrow[0,+\infty]$ \cite[Theorem 5.3.1]{AGS}
\begin{equation}\label{def:disintegration}
    \int_{\R^d}\left(\int_{\R^d} f(x,v)\,d\mu^x(v)\right)\,d\Pi^x\mu(x)=\int_{\R^{2d}} f(x,v)\,d\mu(x,v).
\end{equation}

Given $\sigma\in\P_2(\R^d)$, we introduce the subspace of $\P_2(\R^{2d})$ with fixed marginals
\begin{equation}\label{def:P2x;P2p}
    \P_2^x(\R^{2d};\sigma)=\{\mu\in\P_2(\R^{2d}):\: \Pi^v\mu=\sigma\},\;
    \P_2^v(\R^{2d};\sigma)=\{\mu\in\P_2(\R^{2d}):\: \Pi^x\mu=\sigma\}.
\end{equation}

Now we define the Wasserstein distances with fixed marginals $W_{2,v}$ and $W_{2,x}$ by
\begin{equation}\label{def:partial_Wass}
\begin{split}
    W_{2,v}(\mu,\nu)=
    \begin{cases}
        \left(\int_{\R^d} W_2^2(\mu^x,\nu^x)\,d\Pi^x \mu\right)^{1/2} &\text{ when } \Pi^x \mu=\Pi^x \nu,\\
        +\infty &\text{ otherwise, and }
    \end{cases}\\
    W_{2,x}(\mu,\nu)=
    \begin{cases}
    \left(\int_{\R^d} W_2^2(\mu^v,\nu^v)\,d\Pi^v \mu\right)^{1/2} &\text{ when } \Pi^v \mu=\Pi^v \nu,\\
    +\infty &\text{ otherwise}.
    \end{cases}
\end{split}
\end{equation}
In fact, as we will see in Theorem~\ref{thm:BenBre} that $W_{2,v}(\mu,\nu)$ (resp. $W_{2,x}(\mu,\nu)$) is the length metric induced by $W_2$ on the subspace $\P_2^v(\R^{2d};\Pi^x\mu)$ (resp. $\P_2^x(\R^{2d};\Pi^v\mu)$).

We will consider the energy functional $\H:\P_2(\R^{2d})\rightarrow(-\infty,+\infty]$ of the form
\begin{equation}\label{def:H=V+U+W}
    \H(\mu)=\V(\mu)+\W(\mu)+\U(\mu);
\end{equation}
given $V,W:\R^d\rightarrow(-\infty,+\infty]$ with $W$ additionally satisfying $W(x)=W(-x)$, the potential, interaction, and internal energy functionals $\V,\W,\U$ are defined by
\begin{equation}\label{def:VWU}
\begin{split}
    \V(\mu)&=\int_{\R^d\times\R^d} V(x)+\frac{|v|^2}{2}\,d\mu(x,v),\\
    \W(\mu)&=\frac12\iint_{\R^{2d}\times\R^{2d}} W(x-y)\,d\mu(x,v)\,d\mu(y,w),   \\
    \U(\mu)&= 
    \begin{cases}
    \int_{\R^{2d}} \rho\log\rho\,d\Leb^{2d} &\text{ if } \mu=\rho\Leb^{2d}, \\
    +\infty &\text{ otherwise.}
    \end{cases}
\end{split}
\end{equation}
We restrict our attention to the case where $V,W$ are continuously differentiable, which we write $V,W\in C^1(\R^d)$; note that this does not require $V,W$ to be bounded.
The Vlasov-Fokker-Planck equation (VFP) we consider takes the form
\begin{equation}\label{eq:VFP}
    \partial_t\mu_t + v\cdot \nabla_x \mu_t - (\nabla_x V+\nabla_x W\ast \Pi^x \mu_t) \cdot \nabla_v \mu_t - \alpha(\nabla_v\cdot(v\mu_t)+\Delta_v \mu_t)=0,
\end{equation}
where often we consider a finite time horizon $[0,T]$ for some $T>0$ and for each $t\in[0,T]$ we have $\mu_t\in\P_2(\R^{2d})$. Here $\alpha>0$ is a given friction parameter. It is well-known  that VFP \eqref{eq:VFP} formally has a dissipative Hamiltonian structure \cite[Chapter 8.3.2 ]{Vil03} in terms of the energy functional $\H=\V+\W+\U$ defined in \eqref{def:H=V+U+W}. Indeed, we can rewrite \eqref{eq:VFP} in the form
\begin{equation}\label{eq:VFP;gradform}
        \partial_t \mu_t+\nabla\cdot\left(\mu_t J\grad(\V+\W)(\mu_t)\right)+\nabla\cdot\left(\mu_t S \grad \H(\mu_t)\right)=0.
\end{equation}
where $\grad\E(\mu)$ is the Wasserstein gradient of each functional $\E:\P_2(\R^{2d})\rightarrow(-\infty,+\infty]$ at $\mu\in\P_2(\R^{2d})$ and
\begin{equation}\label{def:JS}
\begin{split}
    J=\begin{pmatrix}
    0 & 1 \\ -1&0
    \end{pmatrix},\quad 
    S=\begin{pmatrix}
    0 & 0 \\ 0& -\alpha    
    \end{pmatrix}.
\end{split}
\end{equation}
Furthermore, noting that $\grad\U(\mu)=\frac{\nabla\rho}{\rho}$ for $\mu=\rho\Leb^{2d}$,
\begin{equation}\label{eq:divJU=0}
    \nabla\cdot\left(\mu J \grad\U(\mu)\right)=\nabla_v\cdot(\mu(\nabla_x\rho/\rho))-\nabla_x\cdot(\mu(\nabla_v\rho/\rho))=0
\end{equation}
we can rewrite \eqref{eq:VFP;gradform} as
\begin{equation}\label{eq:VFP;DisHam}
    \partial_t \mu_t+\nabla\cdot\left(\mu_t (J+S)\grad\H(\mu_t)\right)=0.
\end{equation}
The dissipative or damped Hamiltonian structure is more apparent in the formulation \eqref{eq:VFP;DisHam}: the antisymmetric matrix $J$ captures the conservative or Hamiltonian dynamics, whereas the positive semidefinite matrix $S$ captures the (degenerate) dissipation. Alternatively we can see $S$ as a damping of the Hamiltonian system, hence we use the terms `damped Hamiltonian system' and `dissipative Hamiltonian system' interchangeably. Formally applying the chain rule in the Wasserstein space, one can readily see that the solution $(\mu_t)_{t\geq 0}$ of \eqref{eq:VFP;DisHam} satisfies
\[\frac{d}{dt}\H(\mu_t)=-\alpha^{-1}\|S\grad\H(\mu_t)\|_{L^2(\mu_t)}^2.\]
Indeed, under suitable assumptions the solution of \eqref{eq:VFP} converges to the minimizer of the Hamiltonian $\H$ as $t\rightarrow\infty$.

\subsection{Summary of results}\label{ssec:prelim}
Inspired by the symplectic Euler algorithm for Hamiltonian systems, we introduce the coordinate-wise minimizing movement scheme in the Wasserstein space, discrete solutions of which converge to the distributional solution of VFP \eqref{eq:VFP}.

Let $\mu_0\in\P_2(\R^{2d})$ and fix a time step $h>0$ and the number of iterations $N\in\N$. Then the discrete solutions $(\mu_{ih}^N)_{i=0}^N$ are defined iteratively via the variational scheme
\begin{equation}\label{def:SIE}
\begin{split}
    \bar\mu_{(i+1)h}^N&\in \argmin_{\nu\in\P_2(\R^{2d})} \frac{W_{2,v}^2(\mu_{ih}^N,\nu)}{2h}+\Lfun_v(\nu)+\alpha\H(\nu),\\
    \mu_{(i+1)h}^N&\in\argmin_{\nu\in\P_2(\R^{2d})} \frac{W_{2,x}^2(\bar\mu_{(i+1)h}^N,\nu)}{2h}-\Lfun_x(\nu)
\end{split}
\end{equation}
where $\mu_0^N=\mu_0$, $\H$ is as defined in \eqref{def:H=V+U+W}, and the linear functionals $\Lfun_x,\Lfun_v:\P_2(\R^{2d})\rightarrow\R$ are defined by
\[
    \Lfun_x(\mu)=\int_{\R^{2d}} x\cdot v\,d\mu(x,v),\qquad \Lfun_v(\mu)=\int_{\R^{2d}} v\cdot (\nabla_x V(x)+\nabla_x W\ast \Pi^x\mu)\,d\mu(x,v).
\]
We note that $\Lfun_v(\mu)-\Lfun_x(\mu)$ is the (formal) Poisson bracket of $\H$ and the second moment functional $\mu\mapsto \frac12 \int_{\R^{2d}} |\id|^2\,d\mu$ in the Wasserstein space, introduced by Lott\cite{Lott08} and Gangbo-Kim-Pacini \cite{GanKimPac11}; see Remark~\ref{rmk:Lp-Lx=poisson} for details.

Throughout this paper, we will refer to \eqref{def:SIE} as \emph{the coordinate-wise minimizing movement scheme in the Wasserstein space}, or simply the minimizing movement scheme. We refer to the curve $(\mu_t^N)_{t\in [0,Nh]}$ defined by
\[
        \mu_t^N=\mu_{ih}^N \text{ for } t\in[ih,(i+1)h)
\]
as the piecewise constant interpolation between discrete solutions. The measure $\bar\mu_{(i+1)h}^N$ in \eqref{def:SIE} is an intermediate configuration between $\mu_{ih}^N$ and $\mu_{(i+1)h}^N$, and is not involved in the interpolation. Letting $T>0$ and $h=h_N=T/N$, we call any pointwise narrow limit of $(\mu_t^N)_{t\in[0,T]}$ as $N\rightarrow\infty$ \emph{a coordinate-wise minimizing movement}, or simply a minimizing movement.

The algorithm can be seen as the symplectic Euler algorithm with dissipation in the Wasserstein space. In fact, when we eliminate interaction force between particles by setting $\W=\U\equiv 0$, one can check that this algorithm is equivalent to flowing each particle according to the the symplectic Euler scheme with dissipation \eqref{def:SIE+damp;Rd;sep;var} which we discuss in Section~\ref{ssec:motivation} as a motivation.

Our main theorem, Theorem~\ref{thm:main}, states that when $V,W$ are in $C^1(\R^d)$ with Lipschitz gradients and $\H(\mu_0)<+\infty$ the piecewise constant interpolation between the discrete solutions \eqref{def:SIE}
\[
        \mu_t^N=\mu_{ih_N}^N \text{ for } t\in[ih_N,(i+1)h_N)
\]
pointwise narrowly converges to the distributional solution of VFP \eqref{eq:VFP}.

The proof of our main theorem is split into four parts: metric and differential structure of the Wasserstein distance with fixed marginals, the well-posedness of the discrete problem, the existence of minimizing movements, and the coincidence of minimizing movements with the distributional solutions of the Vlasov-Fokker-Planck equation. We outline the results below.

\vspace{3mm}
\emph{Metric and differential structure of the Wasserstein distance with fixed marginals.}
Section~\ref{sec:W2p} introduces and examines the Wasserstein distance with fixed marginals, and Sections~\ref{sec:partial_metslope} -~\ref{sec:psubdiff} respectively study metric and differential properties of the distance as well as convexity of relevant energy functionals with respect to this distance. As pointed out earlier, $W_{2,v}$ and $W_{2,x}$ are length metric induced by $W_2$ on the suitable subspace of $\P_2(\R^{2d})$ with fixed marginals. Thus many of the useful properties of the Wasserstein distance are transferred, including lower semicontinuity with respect to the narrow convergence (Proposition~\ref{prop:W2p_lsc}), stability of $W_{2,v}$-optimal transport maps (Proposition~\ref{prop:stability}), geodesic-convexity and (partial-)subdifferentials of common energy functionals, estimates on metric slope along solutions of discrete variational problems (Proposition~\ref{prop:pslope_est}). 

Preliminary results on the metric slopes and subdifferentials form a basis for the analysis of the minimizing movement scheme in later sections. While the arguments in Sections~\ref{sec:partial_metslope} and~\ref{sec:psubdiff} are technical, they are largely adaptations of arguments by Ambrosio, Gigli, and Savar\'e \cite{AGS} to our setting. Therefore, readers primarily interested in VFP and the minimizing movement scheme may skip the two sections and refer back to the statements when necessary.

\vspace{3mm}
\emph{Well-posedness of the time-discrete variational problem.}
In Section~\ref{sec:KFP_JKO} we establish the existence and uniqueness of the solution to the time-discrete variational problem \eqref{def:SIE} for any time step $h>0$. Proposition~\ref{prop:LxJKO;Phih} states that, as a consequence of convexity and linearity of the functionals involved in the minimizing movement scheme, each variational step starting at suitable $\mu\in\P_2(\R^{2d})$ is uniquely solvable under relatively weak assumptions. Furthermore, under stronger assumption of $\nabla_x V,\nabla_x W$ being Lipchitz-continuous, Theorem~\ref{thm:iter_solv} asserts that for any $h>0$ and number of iterations $N\in\N$ the iterative algorithm produces a unique solution $(\mu_{ih}^N)_{i=0}^N$.

\vspace{3mm}
\emph{Existence of minimizing movements.}
We prove in Section~\ref{sec:SIE;cpct} the existence of minimizing movements, that is the limit of the constant interpolation between the discrete problem \eqref{def:SIE}. The main result of the section, Theorem~\ref{thm:disc_sol;cpct2}, states that the limiting curve exists in $AC^2([0,T];\P_2(\R^{2d}))$ whenever the sequence of initial data $(\mu_0^N)_{N\in\N}$ narrowly converges to $\mu_0$ satisfying $\H(\mu_0)<+\infty$. This follows from energy estimates along the discrete solutions, which is largely inherited from the connection of the variational scheme to the symplectic Euler algorithm. To simplify the proof we additionally require the functional $\H$ to be bounded from below, which is satisfied in most physically meaningful cases where $V$ is a confining potential and the interaction potential $W$ is bounded from below.

We also state a compactness result under different assumptions. Theorem~\ref{thm:disc_sol;subseq_conv} asserts the existence of a limiting curve that is Lipschitz continuous in $(\P_2(\R^{2d}),W_2)$ given that the $v$-Fisher information at the initial datum is finite. As noted in Remark~\ref{rmk:mu0_xsingular}, this initial condition allows the initial measure of the product form $\mu_0=\sigma\otimes\upsilon$ with any $\sigma\in\P_2(\R^d)$ as long as $\upsilon\in\P_2(\R^d)$ has finite Fisher information. 
 
We note that our argument for the existence of minimizing movements depends only on Lipschitz continuity of $\nabla_x V,\nabla_x W$ and the convexity of the internal energy $\U$ along $W_{2,v}$-geodesics. As seen in Section~\ref{ssec:partial_subdiff;U}, $W_{2,v}$-geodesic convexity of $\U$ is a consequence of the same argument that implies their convexity property with respect to the Wasserstein distance. Thus the results in this section can easily generalize to a more general class of energy functionals; see Remark~\ref{rmk:generalU}.

\vspace{3mm}
\emph{Minimizing movements are weak solutions of the Vlasov-Fokker-Planck equation.}
Proposition~\ref{prop:conv_PDE} states that, if $\nabla_x V,\nabla_x W$ are Lipschitz continuous, then limiting curve resulting from the variational scheme \eqref{def:SIE} is a distributional solution of VFP \eqref{eq:VFP} under suitable assumptions. As the distributional solution of the Vlasov-Fokker-Planck equation is unique when $\nabla_x V,\nabla_x W$ are Lipschitz continuous (see for instance \cite[Theorem 2.2]{Meleard96}), Proposition~\ref{prop:conv_PDE} along with the existence of minimizing movements in Section~\ref{sec:SIE;cpct} imply our main result, namely the convergence of discrete solutions to the distributional solution of VFP. 

We note here that when the interaction is absent, continuous differentiability of $V$ suffices for minimizing movements to coincide with weak solutions of VFP (see Remark~\ref{rmk:conv_PDE}), whereas our the existence of minimizing movements relies on Arzel\`a-Ascoli-type theorem hence requires Lipschitz continuity of $\nabla_x V$. This leaves possibility of yielding the same result under weaker assumptions on $V$ once energy estimates of Section~\ref{sec:SIE;cpct} are further refined.

Furthermore, while we have restricted our attention to the case of the uniform time step, with obvious modifications the statements can generalize to any partition of the the time interval $[0,T]$ with vanishing modulus. This justifies our use of the term \emph{minimizing movements} instead of \emph{generalized minimizing movements}.

\vspace{3mm}
\emph{Two-sided bounds on the decay of energy over discrete solutions.}
In Section~\ref{sec:Hdecay;disc} we establish Proposition~\ref{prop:Hdecay;disc}, which states that the discrete solutions $(\mu_{ih}^N)_{i=0}^N$ of \eqref{def:SIE} with time step $h>0$ satisfy
\begin{align*}
    -\frac{\alpha h}{1+\alpha h}\|\id_v+\nabla_v\log\rho_{ih}^N\|_{L^2(\mu_{ih}^N)}^2 + C h^2&\leq \H(\mu_{(i+1)h}^N)-\H(\mu_{ih}^N)\\
    &\leq -\alpha h\|\id_v+\nabla_v\log\rho_{(i+1)h}^N\|_{L^2(\mu_{(i+1)h}^N)}^2+C h^2
\end{align*}
for some $C>0$, where $\rho_{ih}^N$ and $\rho_{(i+1)h}^N$ are the $\Leb^{2d}$-densities of $\mu_{ih}^N$ and $\mu_{(i+1)h}^N$ respectively. Moreover, when $V,W$ have Lipschitz gradients, the constant $C$ are controlled by the second moments of $\mu_{ih}^N,\mu_{(i+1)h}^N$. This resembles the dissipation of the Hamiltonian along the solution of VFP
\[\frac{d}{dt}\H(\mu_t)=-\alpha\|\id_v+\nabla_v\log\rho_t\|_{L^2(\mu_t)}^2.\]

\medskip
\subsection{Related works}\label{ssec:literature}
\emph{Gradient flows in the Wasserstein space.}
The seminar work of Jordan, Kinderlehrer, and Otto~\cite{JKO98} initiated the study of gradient flows in the Wasserstein space, by establishing that the piecewise-constant interpolations between the solution $(\mu_{ih}^N)_{i=0}^N$ of the iterative variational problem
\begin{equation}\label{eq:JKO}
    \mu_{(i+1)h}^N \in\argmin_{\nu\in\P_2(\R^d)} \frac{W_2^2(\mu_{ih}^N,\nu)}{2h}+\int_{\R^d} V(x)\,d\nu(x) + \int_{\R^d}\frac{d\nu}{d\Leb^d}\log\frac{d\nu}{d\Leb^d}\,d\Leb^d
\end{equation}
converge to the distributional solution of the Fokker-Planck equation as $h\searrow 0$. The variational scheme \eqref{eq:JKO}, often referred to as the JKO scheme due to its considerable impact in the community, falls under a broad class of algorithms called \emph{the minimizing movement scheme}, introduced by De Giorgi~\cite{DeGiorgi93}, that can be described as follows: in metric space $(X,m)$, iteratively solve
\begin{equation}\label{eq:MMS}
    \mu_{(i+1)h}^N \in \argmin_{\nu\in X} \frac{m^2(\mu_{ih}^N,\nu)}{2h}+\E(\nu).
\end{equation}
More generally, $m$ need not be a metric and it can be set to be a measure of discrepancy between elements in $X$.

Convexity properties of the potential, interaction, and internal energy functionals along Wasserstein geodesics, also known as displacement convexity, was studied earlier by McCann \cite{McCann97}. Convexity of the energy functional with respect to the metric plays a key role not only in the convergence rate of the gradient flows to equilibrium \cite{CarMcVil03,CarMcVil05}, but also in the convergence of the minimizing movement scheme in metric spaces. Ambrosio, Gigli, and Savar\'e \cite{AGS} thoroughly investigated the convergence of minimizing movement schemes \eqref{eq:MMS} in metric spaces, with particular focus on the Wasserstein space, exploiting the geodesic convexity of the energy functional $\E$ with respect to $m$. In case $\E$ is (semi)-convex along geodesics in $(X,m)$, desirable discrete energy estimates follow naturally, allowing passage to the limit in a general setting.

Our work is largely inspired both in idea and technique by this rich body of literature in gradient flows, seeking analogous description for damped Hamiltonian systems. In fact, many of the crucial general lemmas in Sections~\ref{sec:partial_metslope} and~\ref{sec:psubdiff} are translations of the robust theory of \cite{AGS} to our setting.

We add that in case the energy functional $\E$ and the distance squared $m^2$ are convex along a common connecting curve, generalized geodesics in the case of the Wasserstein space, the minimizing movement scheme \eqref{eq:MMS} leads to generation of the evolution semigroup with several desirable properties such as regularizing effect, exponential contraction, and explicit discrete-to-continuum convergence rates; see \cite[Chapter 4]{AGS}. In fact, we note in Remark~\ref{rmk:Ucvx;gengeo} that functionals involved in our algorithm \eqref{eq:SIE} are in fact convex along generalized $W_{2,x}$- and $W_{2,v}$-geodesics. We leave to future work whether this stronger convexity property can be exploited to yield more refined results for our variational scheme.

\vspace{3mm}
\emph{Hamiltonian flows in the Wasserstein space.}
In their pioneering work \cite{AmbGan08}, Ambrosio and Gangbo  studies a general form of Hamiltonian ODEs in the Wasserstein space
\begin{equation}\label{eq:AmbGan:HamODE}
    \partial_t \mu_t+\nabla\cdot\left(\mu_t J\grad\H(\mu_t)\right) \text{ for } t\in(0,T)
\end{equation}
where $J$ is a $2d\times 2d$ symplectic matrix, $\H:\P_2(\R^{2d})\rightarrow\R$ is a Hamiltonian energy functional, and $\grad\H(\mu)$ is either the minimal element in the Wasserstein subdifferential $\partial\H(\mu)$ or an element in the intersection of tangent space at $\mu$ and $\partial\H(\mu)$. Later work by Gangbo, Kim, and Pacini \cite{GanKimPac11} constructs a general theory of differential forms and symplectic structures in the Wasserstein space, and in particular further justifies viewing the equation \eqref{eq:AmbGan:HamODE} as a Hamiltonian flow in the Wasserstein space. This symplectic structure is formally equivalent to Poisson structures considered earlier by Marsden and Weinstein \cite{MarsWein81}, Lott \cite{Lott08}, and Khesin and Lee \cite{KhesinLee08}.

Ambrosio and Gangbo establish the existence of solutions of \eqref{eq:AmbGan:HamODE} as a limit of two different time-discretized schemes given initial data $\mu_0\in\P_2(\R^{2d})$, and shows that the Hamiltonian remains constant along the solution $(\mu_t)_{t\in(0,T)}$ of \eqref{eq:AmbGan:HamODE}. Roughly speaking, both algorithms can be seen as explicit schemes in the sense that they provide no obvious way to control the size of $\grad\H$ at a step in terms of the gradient at the previous step. Thus, in order to obtain equicontinuity of time-discretized schemes, strong assumptions on the growth bounds of $\grad\H$, such as Lipschitz continuity of $\mu\mapsto \|\grad\H(\mu)\|_{L^2(\mu)}$ with respect to the Wasserstein distance. This condition is not satisfied by the entropy functional; even within each sublevel set of the entropy functional, Fisher information can be made arbitrarily large with small oscillatory perturbations. Of course, in light of \eqref{eq:divJU=0}, there is no need to consider Hamiltonian functionals involving the entropy functional for Hamiltonian ODEs, yet this assumption makes it difficult to apply fully explicit schemes to obtain dissipative Hamiltonian equations involving internal energy functionals.

Our algorithm \eqref{eq:SIE} leverages the structure of internal energy functionals to circumvent this issue. By adopting a coordinate-wise update, inspired by the symplectic Euler scheme, we are able to implement an implicit Euler step in the momentum variable, which allows us to control the $v$-partial metric slope $|\partial_v\U|$. On the other hand, update in the position variable is simply the pushforward by the map $(x,v)\mapsto (x+hv,v)$, the Jacobian of which has determinant 1, hence leaving both the internal energy $\U$ and the $v$-partial metric slope $|\partial_v\U|$ invariant. On the other hand, we still require Lipschitz continuity of $\nabla_x V,\nabla_x W$ to control the growth of $\grad\V,\grad\W$ along the position update.

We also mention an alternative formalism that allows interpretation of various PDEs as Newtonian equations in the Wasserstein space. Von Renesse \cite{vonRenesse12} used the Madelung transform to cast the Schr\"odinger equation as a Newtonian equation of the sum of a potential energy functional and the Fisher information. Khesin, Miso\l ek, and Schnirelman \cite{KheMisShn23} considered many other important PDEs as Newtonian equations in this context. See also the work of Chow, Li, and Zhou \cite{ChowLiZhou20} for a further list of examples in this perspective.

\vspace{3mm}
\emph{Time-discrete variational formulation of VFP and related equations.}
Jordan and Huang \cite{JorHua00} studied a variational scheme of the form \eqref{eq:MMS} converging to the weak solution of the (regularized) Vlasov-Poisson-Fokker-Planck equation. The energy functional $\E$ is of the sum of potential, interaction, and internal energies as in \eqref{def:H=V+U+W} with suitable choice of $V,W$, whereas the space is $X=\P_2(\R^{2d})$, and the `metric' $m=m_h$ is the optimal transport cost associated to the cost function
$c_h((x,v),(y,w))=|w-v|^2/2+\frac12|(y-x)/h-(v+w)|^2$. 
Observe that $m_h$ depends on the time step $h$ and is not a bona fide metric on $\P_2(\R^{2d})$. This approach can be seen as providing interpretation of VFP as a gradient flow with respect to an adjusted `metric' $m_h$.

Duong, Peletier, and Zimmer \cite{DPZ14} proposed different variational schemes that separate the conservative (Hamiltonian) and dissipative effects. In the simplified case of the kinetic Fokker-Planck equation corresponding to $\W\equiv 0$ in \eqref{eq:VFP}, this algorithm is also of the form \eqref{eq:MMS}, but uses the optimal transport cost associated to a cost function depending on the time step and $\nabla V$ to capture the conservative dynamics, whereas the energy functional $\E(\mu)=\int_{\R^{2d}} V(x)\,d\mu(x,v)+\U(\mu)$ captures the dissipative part. The cost function has close connections to the rate functional appearing in the large deviation principle corresponding to VFP, which is studied in greater generality in a separate work~\cite{DuPeZi13} by the same authors. Yet, incorporating $\nabla V$ in the cost function introduces third order derivatives of $V$ in the estimates and makes the argument difficult. We stress that their goal is to capture the \emph{underlying physics} as accurately as possible in discretization, which is disparate from our motivation of preserving desirable \emph{geometric} properties when discretizing in time.

Adams, Duong, and Dos Reis \cite{ADR22} further considers an operator-splitting scheme which incorporates the conservative effect directly by solving the transport equation and the dissipative effect in the form \eqref{eq:MMS} where $m$ is an optimal transport cost associated to a cost function dependent on the time step. 

Carlen and Gangbo \cite{CarGan04} introduced the flow-and-descend algorithm for a nonlinear kinetic-Fokker-Planck equation which is closely connected to the Boltzmann equation. Their algorithm first updates the position (`flows') by taking the push-forward with respect to the map $\Phi_h(x,v)=(x+hv,v)$, then `descends' each conditional density $\mu^x\in\P_2(\R^d)$ for fixed $x\in\R^d$ to dissipate the entropy of $\mu^x$ relative to the Maxwellian density in the Wasserstein space $(\P_2(\R^d),W_2)$. In fact, our position update in \eqref{def:SIE} coincides with the flow step, whereas the implementing the steepest descent with respect to the Wasserstein distance with fixed marginals shares the same philosophy as descending for each conditional density; Remark~\ref{rmk:Jp_alt} observes that we can reformulate our velocity update purely in form of the conditional densities using the chain rule for the entropy functional. However, as the energy functional $\H$ defined by \eqref{def:H=V+U+W} is a Lyapunov functional for VFP \eqref{eq:VFP}, it is more natural to take variational step using the energy functional involving the (full) entropy instead of the conditional entropy, which is made possible in our work by introducing $W_{2,v}$.

We also note that, by adding a new variable, VFP \eqref{eq:VFP} can be reformulated as a GENERIC (General Equations for Non-Equilibrium Reversible-Irreversible Coupling) system. We refer the readers to \cite[Section 3]{DuPeZi13} for this interpretation. J\"ungel, Stefanelli, Trussardi \cite{JST21}, based on their earlier work \cite{JST19}, proposed a minimizing movement scheme for GENERIC systems set in reflexive Banach spaces. Using Fenchel's relations, the authors recast the GENERIC equation as an equivalent inequality, which form the basis of the proposed variational scheme. 

\vspace{3mm}
\emph{The Wasserstein distance with fixed marginals.} 
We introduce the Wasserstein distance with fixed marginals in order to implement the coordinate-wise minimizing movement scheme \eqref{def:SIE}. Upon completion of the first version of the manuscript, Jeremy Sheung-Him Wu kindly pointed out to us that these distances have been previously considered by Peszek and Poyato~\cite{PesPoy23}, who refer to the metrics as the fibered Wasserstein distances. In particular, some of the basic metric properties we establish in Sections~\ref{sec:W2p}-\ref{sec:partial_metslope} were proven in their work. As our proofs rely on standard but different arguments, we have included them 
for completeness. We also note that a recent work of Kitagawa and Takatsu~\cite{KitTak24} studies the disintegrated optimal transport problem, which includes the fibered Wasserstein distance as a special case.

Moreover, Peszek and Poyato develop a theory of fibered gradient flows based on this metric, and use the framework to study various models with heterogeneous gradient flow structures. On the other hand, we use the fibered Wasserstein distance on $\R^d\times\R^d$ in both fibers in our minimizing movement-type scheme~\eqref{def:SIE} to capture the dissipative Hamiltonian structure of VFP.  

\section{The Wasserstein distance with fixed marginals}\label{sec:W2p}
In this section we introduce and examine basic properties of the Wasserstein distances with fixed marginals, $W_{2,v}$ and $W_{2,x}$. Section~\ref{ssec:motivation} motivates the introduction of $W_{2,v},W_{2,x}$ as the metric with respect to which to perform steepest descent via a variational reformulation of the symplectic Euler scheme in Euclidean spaces. In Section~\ref{ssec:W2p_defn} we define these distances and observe that $W_{2,v}$ (resp. $W_{2,x}$) is the length metric induced by $W_2$ on the subspace of $\P_2(\R^{2d})$ with fixed $x$-marginals (resp. $v$-marginals); see \eqref{eq:W2p_as_length} of Theorem~\ref{thm:BenBre}. Then we will proceed to establish some useful properties such as lower semicontinuity with respect to narrow topology (Proposition~\ref{prop:W2p_lsc}) and stability of $W_{2,v}$-optimal transport maps (Proposition~\ref{prop:stability}).

\subsection{Motivation: symplectic Euler scheme in Euclidean spaces}\label{ssec:motivation}

To motivate our time-discretization, we first start by considering damped Hamiltonian ODEs in the Euclidean space. Given a suitable Hamiltonian $H:\R^{2d}\rightarrow\R$, the corresponding damped Hamiltonian ODE with friction parameter $\alpha>0$ is
\begin{equation}\label{def:HamODE;Rd}
\begin{split}
    \dot v_t &= -\partial_x H(x_t,v_t)-\alpha\partial_v H(x_t,v_t),\\
    \dot x_t &= \partial_v H(x_t,v_t)
\end{split}
\end{equation}
In the purely Hamiltonian case with $\alpha=0$, the corresponding symplectic Euler scheme with time step $h>0$ is
\begin{equation}\label{def:SIE;Rd}
\begin{split}
    v_{i+1}&=v_i - h\partial_x H(x_i,v_{i+1}),\\
    x_{i+1}&=x_i+h \partial_v H(x_{i+1},v_{i+1}).
\end{split}
\end{equation}
The symplectic Euler scheme \eqref{def:SIE;Rd;sep}, also known as the semi-implicit Euler scheme, is a first order symplectic integrator \cite{Vog56}, meaning that it is a first order numerical approximation of \eqref{def:HamODE;Rd} and the map $(x_i,v_i)\mapsto (x_{i+1},v_{i+1})$ is symplectic -- i.e. denoting by $\frac{\partial(x_{i+1},v_{i+1})}{\partial(x_i,v_i)}$ the Jacobian of this map and by $^\ast$ the matrix transpose, we have
\[\frac{\partial(x_{i+1},v_{i+1})}{\partial(x_i,v_i)}^\ast J \frac{\partial(x_{i+1},v_{i+1})}{\partial(x_i,v_i)} = J.\]
This identity can be easily verified, and implies preservation of volume (and orientation) under the map, which is the characterizing structural property of Hamiltonian systems. As a result, the symplectic Euler scheme almost preserves the Hamiltonian; this is in contrast to the explicit Euler scheme or the implicit Euler scheme which respectively tend to increase or damp the energy over time. We refer the readers to a extensive exposition on symplectic and geometric integration by Hairer, Lubich, and Wanner \cite[Chapter 6]{HaiLubWan02} for further information. We also note that Leimkuller and Matthews \cite{LeimMat13} suggested a splitting scheme for the underdamped Langevin dynamics, where the Hamiltonian dynamics is captured by the symplectic Euler scheme (or a higher order symplectic scheme) and the degenerate diffusion by the exact solver for the Ornstein–Uhlenbeck process. We take a different approach of reformulating the symplectic Euler scheme variationally, which will allow a direct implementation in the Wasserstein space. 

When $H$ is separable, for instance $V:\R^d\rightarrow\R^d$ is some suitable potential function and
\begin{equation}\label{def:H;sep_euc}
H(x,v)= V(x)+\frac12 |v|^2
\end{equation}
then the symplectic Euler scheme \eqref{def:SIE;Rd} is
\begin{equation}\label{def:SIE;Rd;sep}
\begin{split}
    v_{i+1}&=v_i - h\nabla_x V(x_i),\\
    x_{i+1}&=x_i+h v_{i+1}
\end{split}
\end{equation}
which is an explicit scheme. Alternatively, we may write \eqref{def:SIE;Rd;sep} as
\begin{equation}\label{def:SIE;Rd;sep;var}
\begin{split}
v_{i+1}&=\argmin_{w\in\R^d} \frac{|w-v_i|^2}{2h} + \nabla_x V(x_i)\cdot w,\\
x_{i+1}&=\argmin_{y\in\R^d} \frac{|y-x_i|^2}{2h} - y\cdot v_{i+1}.
\end{split}
\end{equation}
The variational formulation \eqref{def:SIE;Rd;sep;var} has two advantages: firstly, it is easier to translate to the Wasserstein setting, as it only relies on the metric rather than the vector space structure; moreover, dissipative effect in one of the variables can be naturally incorporated in the following way
\begin{equation}\label{def:SIE+damp;Rd;sep;var}
\begin{split}
v_{i+1}&=\argmin_{w\in\R^d} \frac{|w-v_i|^2}{2h} + \nabla_x V(x_i)\cdot w +\alpha H(x_i,w)\\
x_{i+1}&=\argmin_{y\in\R^d} \frac{|y-x_i|^2}{2h} - y\cdot v_{i+1}.
\end{split}
\end{equation}
One can easily check that the solution of \eqref{def:SIE+damp;Rd;sep;var} satisfies
\begin{align*}
    v_{i+1}&=v_i -h\nabla_x V(x_i) - \alpha h \partial_v H(x_i,v_{i+1}),\\
    x_{i+1}&=x_i+h v_{i+1},
\end{align*}
which is a discretization of the damped Hamiltonian system \eqref{def:HamODE;Rd}.

The coordinate-wise minimizing movement scheme this paper studies can be seen as an implementation of \eqref{def:SIE+damp;Rd;sep;var} in the Wasserstein space. Note that \eqref{def:SIE+damp;Rd;sep;var} first takes a steepest descent with respect to the Euclidean distance in the $v$-variable with $x_i$ fixed, then updates the position variable $x$ by taking steepest descent with respect to the distance in the $x$ variable with the velocity $v_{i+1}$ fixed. This motivates us to introduce \emph{the Wasserstein distance with fixed marginals}.

\subsection{Definition and basic properties}\label{ssec:W2p_defn}

Let $\mu,\nu\in\P_2(\R^{2d})$. Then we define the partial Wasserstein distances $W_{2,v}$ and $W_{2,x}$ by
\begin{equation*}
\begin{split}
    W_{2,v}(\mu,\nu)=
    \begin{cases}
        \left(\int_{\R^d} W_2^2(\mu^x,\nu^x)\,d\Pi^x \mu\right)^{1/2} &\text{ when } \Pi^x \mu=\Pi^x \nu,\\
        +\infty &\text{ otherwise }
    \end{cases}\\
    W_{2,x}(\mu,\nu)=
    \begin{cases}
    \left(\int_{\R^d} W_2^2(\mu^v,\nu^v)\,d\Pi^v \mu\right)^{1/2} &\text{ when } \Pi^v \mu=\Pi^v \nu,\\
    +\infty &\text{ otherwise}.
    \end{cases}
\end{split}
\end{equation*}
It is clear from the definitions that $W_{2,v}$ and $W_{2,x}$ are extended metrics on $\P_2(\R^{2d})$. Furthermore, note that whenever $\Pi^x\mu=\Pi^x\nu$, applying the triangle inequality in $W_2$ for each disintegration,
\begin{align*}
    W_{2,v}(\mu,\nu)&\leq \left(\int_{\R^d} W_2^2(\mu^x,\delta_0)\,d\Pi^x \mu\right)^{1/2}+\left(\int_{\R^d} W_2^2(\delta_0,\nu^x)\,d\Pi^x \nu\right)^{1/2} \\
    &\leq \left(\int_{\R^{2d}} |v|^2\,d\mu(x,v)\right)^{1/2}+\left(\int_{\R^{2d}} |v|^2\,d\nu(x,v)\right)^{1/2}<+\infty.
\end{align*}
Thus, for each $\sigma\in\P_2(\R^d)$, $(\P_2^v(\R^{2d};\sigma),W_{2,v})$ and $(\P_2^x(\R^{2d};\sigma),W_{2,x})$ are metric spaces.

Let us now fix $\sigma\in\P_2(\R^d)$ and define the set of couplings with fixed-marginals
\begin{equation}\label{def:Gammapx}
\begin{split}
    \Gamma^v(\mu,\nu)=\left\{\gamma=\int_{\R^d}\gamma^x\,d\sigma(x):\;\gamma^x\in\Gamma(\mu^x,\nu^x) \text{ for } \sigma\text{-a.e. } x\in\R^d\right\} \text{ given } \mu,\nu\in\P_2^v(\R^{2d};\sigma),\\
    \Gamma^x(\mu,\nu)=\left\{\gamma=\int_{\R^d}\gamma^v\,d\sigma(v):\;\gamma^v\in\Gamma(\mu^v,\nu^v) \text{ for } \sigma\text{-a.e. } x\in\R^d\right\} \text{ given } \mu,\nu\in\P_2^x(\R^{2d};\sigma).
\end{split}
\end{equation}
So we have
\begin{equation}\label{eq:W2pW2x;coupling}
\begin{split}
    W_{2,v}(\mu,\nu)&=\inf_{\gamma\in\Gamma^v(\mu,\nu)}\left(\int_{\R^d}\left(\iint_{\R^{d}\times\R^{d}}|w-v|^2\,d\gamma^x(v,w)\right)\,d\sigma(x)\right)^{1/2} \text{ for } \mu,\nu\in\P_2^v(\R^{2d};\sigma),\\
    W_{2,x}(\mu,\nu)&=\inf_{\gamma\in\Gamma^v(\mu,\nu)}\left(\int_{\R^d}\left(\iint_{\R^{d}\times\R^{d}}|y-x|^2\,d\gamma^v(y,x)\right)\,d\sigma(v)\right)^{1/2} \text{ for } \mu,\nu\in\P_2^x(\R^{2d};\sigma)
\end{split}
\end{equation}
As $\Gamma_o(\mu^x,\nu^x)$ (resp. $\Gamma_o(\mu^v,\nu^v)$) is non-empty for $\sigma$-a.e. $x\in\R^d$ (resp. $v\in\R^d$), the set of optimal couplings with fixed marginals
\begin{equation}\label{def:Gammapxo}
\begin{split}
    \Gamma^v_o(\mu,\nu)=\{\gamma\in\Gamma^v(\mu,\nu):\;\gamma^x\in\Gamma_o(\mu^x,\nu^x) \text{ for } \sigma\text{-a.e.} x\in\R^d\},\\
    \Gamma^x_o(\mu,\nu)=\{\gamma\in\Gamma^x(\mu,\nu):\;\gamma^v\in\Gamma_o(\mu^v,\nu^v) \text{ for } \sigma\text{-a.e.} v\in\R^d\}
\end{split}
\end{equation}
is nonempty and the infima in \eqref{eq:W2pW2x;coupling} are attained respectively by couplings in $\Gamma^v_o(\mu,\nu)$ and $\Gamma^x_o(\mu,\nu)$.

As all the definitions are symmetric in $x$ and $v$, without loss of generality we establish results for
$W_{2,v}$. As we shall see in Section~\ref{sec:KFP_JKO}, more `interesting' variational steps happen in the $v$-coordinate in our algorithm for VFP.

\begin{remark}[Relationship of $W_{2,v},W_{2,x}$ with $W_2$]\label{rmk:W2p_leq_W2}
We can see $W_{2,v}$ (resp. $W_{2,x}$) as the infimum of the quadratic cost over couplings that `move mass only along $v$-coordinate' (resp. $x$-coordinate). Thus we have 
\begin{equation}\label{eq:W2_leq_W2p}
    W_2(\mu,\nu)\leq W_{2,v}(\mu,\nu), W_{2,x}(\mu,\nu) \text{ for all }\mu,\nu\in\P_2(\R^{2d})
\end{equation}
in general. 

Furthermore, $W_{2,v}$ induces a stronger topology on $\P_2^v(\R^{2d};\sigma)$, $\sigma\in\P_2(\R^d)$. To see this, note that $W_{2,v}(\mu_n,\mu)\xrightarrow[]{n\rightarrow\infty} 0$ implies that for $\Pi^x\mu$-a.e. $x\in\R^d$ $W_{2}(\mu_n^x,\mu^x)\xrightarrow[]{n\rightarrow\infty}$ which in turn implies that $\mu_n^x$ converges narrowly to $\mu^x$. This is stronger than convergence with respect to $W_2$ in $\P_2^v(\R^{2d};\sigma)$. To take a simple example, consider $\P_2^v(\mathbb{T}^2;\sigma)$ with $\sigma$ being the Lebesgue measure on the 1-dimensional torus. Let $\mu_n$ be the normalized measure concentrated on the map $x\mapsto nx$ mod 1. One can readily check that $\Pi^x\mu_n=\Leb^1\restr_{[0,1]}$ and $\mu_n$ narrowly converges to the 2-dimensional Lebesgue measure on $\mathbb{T}^2$. On the other hand, each disintegration with respect to $x$ is a dirac mass, hence $\mu_n^x$ cannot narrowly converge to $\Leb^1\restr_{[0,1]}$. Indeed, one can check $W_{2,v}^2(\mu_n,\mu_{(n+1)})=\int_0^1 |x|^2\,dx$.
\end{remark}

We first check that $(\P_2^v,W_{2,v})$ is a complete metric space for each $\sigma\in\P_2(\R^d)$.
\begin{proposition}[Completeness of $(\P_2^v,W_{2,v})$]\label{prop:P2p_W2p_complete}
    Given $\sigma\in\P_2(\R^d)$, each
    $(\P_2^v(\R^{2d};\sigma),W_{2,v})$ is a complete metric space.
\end{proposition}

\begin{proof}
From considering the disintegration with respect to $\sigma$, it is clear that $W_{2,v}$ satisfies the metric axioms on $\P_2^v(\R^{2d};\sigma)$, thus it remains to show completeness. Suppose $(\mu_n)_n$ is a Cauchy sequence in $\P_2^v(\R^{2d};\sigma)$ with respect to $W_{2,v}$. Then by Remark~\ref{rmk:W2p_leq_W2} the sequence $(\mu_n)_n$ is also Cauchy with respect to $W_{2}$, thus it has a subsequential Wasserstein limit $\mu_0\in\P_2(\R^{2d})$. Considering the convergent subsequence without relabling, note that $\mu_n$ narrowly converges to $\mu_0$. As narrow convergence is preserved by projection, $\Pi^x\mu_n=\sigma$ narrowly converges to $\Pi^x\mu_0$, thus $\Pi^x\mu_0=\sigma$ and we deduce completeness.

\end{proof}

As the narrow convergence behaves poorly with respect to disintegration, a priori it is not clear that the Wasserstein distance with fixed marginals are l.s.c. w.r.t. the topology of narrow convergence. Thus we turn to the dynamic formulation for the Wasserstein distance with fixed marginalss to establish the lower semicontinuity.

First we recall the definition of absolutely continuous curves in metric spaces. 
\begin{definition}\label{def:ACcurve}
Given a complete metric space $(X,m)$ and an interval $I\subset\R$ we say a curve $\bmu:I\mapsto \mu_t\in X$ belongs to $AC(I; X,m)$ if there exists $m\in L^1(I)$ such that
\begin{equation}\label{def:ACp_curve}
    m(\mu_s,\mu_t)\leq \int_s^t m(r)\,dr \quad \forall s,t\in I \text{ with }s<t.
\end{equation}
Furthermore, for any $\bmu\in AC(I; X,m)$, the metric derivative
\begin{equation}\label{def:met_der}
    |\mu'|_m(t):=\lim_{s\rightarrow t}\frac{m(\mu(s),\mu(t))}{|s-t|}.
\end{equation}
exists for $\Leb^1$-a.e. $t\in(a,b)$, and $|\mu'|_m\in L^1(I)$. In case $|\mu'|_m\in L^p(I)$ for $1\leq p<+\infty$ we write $\bmu\in AC^p(I;X,m)$.
\end{definition}

In most cases we will deal with $I=[0,T]$ for some finite time horizon $T>0$. We will often write $AC(I;X)$ when the choice of the metric $m$ is obvious from the context. When $(X,m)=(\P_2(\R^{2d}),W_2)$, we simply write $|u'|_{W_2}=|u'|$. 

\begin{theorem}[Benamou-Brenier formula for $W_{2,v}$]\label{thm:BenBre}
Let $\sigma\in\P_2(\R^d)$ and $\mu,\nu\in\P_2^v(\R^{2d};\sigma)$. Then
\begin{equation}\label{eq:BenBre}
    \begin{split}
    W_{2,v}(\mu,\nu)=\min\left\{\left(\int_{0}^1\|u_t\|_{L^2(\mu_t;\R^d)}^2\,dt\right)^{1/2}:\;\partial_t\mu_t+\nabla\cdot\left(\mu_t\begin{pmatrix} 0 \\ u_t
    \end{pmatrix}\right)=0 \text{ with } \mu_0=\mu,\;\mu_1=\nu\right\},
    \end{split}
\end{equation}
where the infimum is over narrowly continuous curves $\mu:t\mapsto\mu_t\in\P_2(\R^{2d};\sigma)$ and Borel vector fields $u:(t,x,v)\mapsto u_t(x,v)\in\R^d$ with $u_t\in L^2(\mu_t;\R^d)$ satisfying the continuity equation in the distributional sense
\begin{equation}\label{def:CEv;dist}
    0= \int_0^1 \int_{\R^{2d}} [\partial_t\varphi(t,x,v) +\nabla_v \varphi(t,x,v)\cdot u_t(x,v)]\,d\mu_t(x,v)\,dt \text{ for each } \varphi\in C_c^\infty((0,1)\times\R^{2d}).
\end{equation}

Moreover, $W_{2,v}$ is the length metric induced by $W_2$ in $\P_2^v(\R^{2d};\sigma)$ -- i.e.
\begin{equation}\label{eq:W2p_as_length}
    W_{2,v}(\mu,\nu)=\inf\left\{\int_0^1 |\mu'|(t)\,dt:\;(\mu_t)_{t\in[0,1]}\in AC([0,1];\P_2^v(\R^{2d};\sigma),W_2) \text{ with } \mu_0=\mu,\;\mu_1=\nu\right\}.
\end{equation}
\end{theorem}

\begin{proof}

\noindent\emph{Step 1$^o$.} Let us first suppose $(\mu_t,(0,u_t)^\ast)_{t\in[0,1]}$ solve the continuity equation in the sense of \eqref{def:CEv;dist} with $\mu_0=\mu$ and $\mu_1=\nu$, where $^\ast$ denotes the transpose. Let the Borel families $\{\mu^x\}_{x\in\R^d},\{\nu^x\}_{x\in\R^d}$ be the respective disintegrations of $\mu,\nu\in\P_2^v(\R^{2d};\sigma)$ with respect to $\sigma$. Writing $u_t^x(v):=u_t(x,v)$, $u^x:(t,v)\mapsto u_t^x(v)$ is a Borel vector field, and from the disintegration theorem we know that $\mu^x,\nu^x$ are uniquely defined for $\sigma$-a.e. $x\in\R^d$. Furthermore, for each $\varphi\in C_c^\infty((0,1)\times\R^d\times\R^d)$
\begin{align*}
    0&= \int_0^1 \int_{\R^{2d}} [\partial_t\varphi(t,x,v) +\nabla_v \varphi(t,x,v)\cdot u_t(x,v)]\,d\mu_t(x,v)\,dt \\
    &= \int_0^1 \int_{\R^d} \left(\int_{\R^d} [\partial_t\varphi(t,x,v) +\nabla_v \varphi(t,x,v)\cdot u_t^x(v)]\,d\mu_t^x(v)\right)d\sigma(x)\,dt.
\end{align*}
As
\begin{align*}
    &\int_0^1 \iint_{\R^d\times\R^d} |\partial_t\varphi(t,x,v) +\nabla_v \varphi(t,x,v)\cdot u_t(x,v)|\,d\mu_t^x(v)\,dv\,d\Pi^x \mu(x)\,dt\\
    &\leq C \int_0^1 \iint_{\R^d\times\R^d} (1+|u_t(x,v)|)\,d\mu_t^x(v)\,d\Pi^x \mu(x)\,dt
    \leq C \left(1+\int_{\R^d} W_2^2(\mu^x,\nu^x)\,d\Pi^x\mu(x)\right)<+\infty,
\end{align*}
we may consider test functions of the form $\varphi(t,x,v)=\psi(x)\tilde\varphi(t,v)$ with $\psi\in C_c^\infty(\R^d)$ and $\tilde\varphi\in C_c^\infty((0,1)\times\R^d)$ and apply Fubini's theorem to deduce
\[\int_{\R^d} \psi(x) \left(\int_0^1  \int_{\R^d} [\partial_t\tilde\varphi(t,v) +\nabla_v \tilde\varphi(t,v)\cdot u_t^x(v)]\,d\mu_t^x(v)\,dt\right)d\sigma(x)=0.\]
As $\psi\in C_c^\infty(\R^d)$ is arbitrary, we deduce that for $\sigma$-a.e. $x\in\R^d$
\[\int_0^1  \int_{\R^d} [\partial_t\tilde\varphi(t,v) +\nabla_v \tilde\varphi(t,v)\cdot u_t^x(v)]\,d\mu_t^x(v)\,dt =0 \text{ for all } \tilde\varphi\in C_c^\infty((0,1)\times\R^d).\]
Thus by the (usual) Benamou-Brenier formula \cite{BenBre00} for each $\mu^x,\nu^x$, we obtain
\begin{align*}
    W_2^2(\mu^x,\nu^x) \leq \int_0^1 \int_{\R^d} |u_t(x,v)|^2\,d\mu_t^x(v)\,dt \text{ for } \sigma\text{-a.e. }x\in\R^d.
\end{align*}
Integrating over $\sigma$ we obtain the inequality ``$\leq$'' of \eqref{eq:BenBre}.

\vspace{3mm}
\noindent\emph{Step 2$^o$.}
To see that in \eqref{eq:BenBre} the minimum is attained and the equality holds, for each $\sigma$-a.e. $x\in\R^d$ let $(\mu_t^x)_{t\in [0,1]}$ be the constant-speed displacement interpolation from $\mu^x$ to $\nu^x$. Then define $\mu_t:=\int_{\R^d} \mu_t^x\,d\sigma(x)$ for each $t\in[0,1]$, and note
\begin{equation}\label{eq:W2v-geo}
W_2(\mu_s,\mu_t)\leq W_{2,v}(\mu_s,\mu_t) = (t-s)W_{2,v}(\mu,\nu) \text{ for all } 0\leq s< t\leq 1.
\end{equation}
Replace $(\mu_t)_{t\in[0,1]}$ with its constant-speed reparametrization with respect to $W_2$ (see for instance \cite[Lemma 1.1.4]{AGS}), the same inequality holds and $(\mu_t)_{t\in[0,1]}\in AC^2([0,1];\P_2(\R^{2d}))$, and therefore \cite[Theorem 8.3.1]{AGS} allows us to find a Borel vector field $\bv:(t,x,v)\mapsto \bu_t(x,v)\in\R^{2d}$ such that for each $\varphi\in C_c^\infty((0,1)\times\R^{2d})$
\begin{equation}\label{eq:CE;phi}
    \int_0^1 \int_{\R^{2d}} \partial_t\varphi(t,x,v)+ \bu_t(x,v)\cdot\nabla\varphi(t,x,v)\,d\mu_t(x,v)\,dt=0
\end{equation}
and $\|\bu_t\|_{L^2(\mu_t;\R^{2d})}= |\mu'|(t)$ for a.e. $t\in(0,1)$. As $\Pi^x\mu_t=\sigma\in\P_2(\R^d)$ for a.e. $t\in[0,1]$, choosing $\tilde\varphi\in C_c^\infty((0,1)\times\R^d)$, we can argue by approximation that \eqref{eq:CE;phi} holds with $\varphi(t,x,v):=\tilde\varphi(t,x)$. Then writing $\bu_t=(w_t,u_t)$
\[0=\frac{d}{dt}\int_0^1 \int_{\R^{2d}}\tilde\varphi(t,x)\,d\mu_t(x,v)=- \int_{\R^{2d}}w_t(x,v)\cdot\nabla_x\tilde\varphi(t,x)\,d\mu_t(x,v)\,dt.\]
As $\tilde\varphi$ was arbitrary, we see $w_t\equiv 0$ for a.e $t\in[0,1]$, and thus
\begin{align*}
    |\mu'|(t)=\|\bu_t\|_{L^2(\mu_t)}=\|(0,u_t)\|_{L^2(\mu_t)} \text{ for a.e. } t\in[0,1].
\end{align*}
On the other hand, \eqref{eq:W2v-geo} also implies
\[|\mu'|(t)\leq W_{2,v}(\mu,\nu) \text{ for a.e. } t\in[0,1].\]
As $(\mu_t)_{t\in[0,1]}$ is constant-speed with respect to $W_2$,
\[W_{2,v}(\mu,\nu)\geq \int_0^1 |\mu'|(t)\,dt = \left(\int_0^1 |\mu'|^2(t)\,dt\right)^{1/2} = \left(\int_0^1 \|u_t\|_{L^2(\mu_t)}^2\,dt\right)^2 \]
and conclude the proof of \eqref{eq:BenBre}, and at the same time establish the bound ``$\geq$'' of \eqref{eq:W2p_as_length}.

\vspace{3mm}
\noindent\emph{Step 3$^o$.}
To establish the remaining inequality of \eqref{eq:W2p_as_length}, take any
absolutely continuous curve $(\mu_t)_{t\in[0,1]}$ in $AC([0,1];\P_2^v(\R^{2d};\sigma),W_2)$  with $\mu_0=\mu$, $\mu_1=\nu$. By considering a constant-speed reparametrization of $(\mu_t)_{t\in[0,1]}$ with respect to $W_2$ and repeating the argument in Step 2$^\circ$, we can produce $(\mu_t,(0,u_t))_{t\in[0,1]}$ of the continuity equation such that
\[|\mu'|(t)=\|(0,u_t)\|_{L^2(\mu_t)} \text{ for a.e. } t\in[0,1].\]
Then by the inequality ``$\leq$'' of \eqref{eq:BenBre} and that $t\mapsto|\mu'|(t)$ is constant, we deduce
\begin{align*}
    W_{2,v}(\mu,\nu)\leq \left(\int_0^1 \|u_t\|_{L^2(\mu_t;\R^d)}^2\,dt\right)^{1/2} = \left(\int_0^1 |\mu'|^2(t)\,dt\right)^{1/2} = \int_0^1 |\mu'|(t)\,dt.
\end{align*}

\end{proof}

We state a refined Arzel\`a-Ascoli theorem due to Ambrosio, Gigli, and Savar\'e \cite[Proposition 3.3.1.]{AGS} that allows us to construct pointwise narrow limit of equicontinuous curves. We will use this theorem to construct $W_{2,v}$-geodesics, which can be then used to deduce the lower semicontinuity of $W_{2,v}$ in Proposition~\ref{prop:W2p_lsc}. Moreover, this Arzel\`a-Ascoli-type theorem will also be pivotal in showing existence of minimizing movements in Section~\ref{sec:SIE;cpct}.

\begin{proposition}[A refined Arzel\`a-Ascoli Theorem \cite{AGS}]\label{prop:ArzelaAscoli-weak}
Let $(X,m)$ be a complete metric space. Let $T>0$ and $K\subset X$ be a sequentially compact set with respect to topology $\sigma$, and let $\mu^n:[0,T]\rightarrow X$ be curves such that
\begin{equation}\label{eq:u_n;eqcts}
\mu^n(t)\in K\quad\forall n\in\N,\;t\in[0,T],\\
\limsup_{n\rightarrow\infty} m(\mu^n(s),\mu^n(t))\leq \omega(s,t)\quad \forall s,t\in[0,T],
\end{equation}
for a symmetric function $\omega:[0,T]\times[0,T]\rightarrow[0,\infty)$, such that
\[\lim_{(s,t)\rightarrow(r,r)}\omega(s,t)=0\quad\forall r\in[0,T]\setminus\mathscr{C}\]
where $\mathscr{C}$ is an (at most) countable subset of $[0,T]$. Then there exists an increasing subsequence $k\mapsto n(k)$ and a limit curve $\mu:[0,T]\rightarrow X$ such that
\[\mu^{n(k)}(t)\xrightharpoonup[]{\sigma}\mu(t)\quad\forall t\in[0,T],\;\mu \text{ is continuous with respect to }  m \text{ in } [0,T]\setminus\mathscr{C}.\]
\end{proposition}

Now we are ready to show the lower semicontinuity of $W_{2,v}$, which will be crucial in later sections. We note that this result was previously established by Peszek and Poyato \cite[Proposition 3.12 (i)]{PesPoy23} via an argument using transport plans.

\begin{proposition}[Lower semicontinuity of $W_{2,v}$ with respect to the narrow topology]\label{prop:W2p_lsc}
The map $(\mu,\sigma)\mapsto W_{2,v}(\mu,\sigma)$ from $\P_2(\R^{2d})\times\P_2(\R^{2d})$ to $[0,+\infty]$ is l.s.c. w.r.t. the narrow topology.
\end{proposition}

\begin{proof}
Let $\sigma\in\P_2(\R^d)$ and let $(\mu^k)_{k},(\nu^k)_{k}$ be a sequence in $\P_2(\R^{2d})$ such that $\mu^k\rightharpoonup\mu$ and $\nu_k\rightharpoonup\nu$ narrowly as $k\rightarrow\infty$. Without loss of generality, we may assume that the sequences are in $\P_2^v(\R^{2d};\sigma)$, as otherwise the statement of lower semicotinuity is trivial.

Let $\gamma^k\in\Gamma^v_o(\mu^k,\nu^k)$. Then $\mu_t^x=((1-t)\pi^1+t\pi^2)_\#\gamma^x$ is a constant speed $W_{2,v}$-geodesic from $\mu^k$ to $\nu^k$ and thus
\begin{align*}
    &\int_{\R^{2d}}|z|^2\,d\mu^k_t(z)
    =\int_{\R^{2d}}|x|^2+|v|^2\,d((1-t)\pi^1+t\pi^2)_\#(\gamma^k)^x(v)\,d\sigma(x)  \\
    &\qquad=\int_{\R^{2d}}|x|^2+|(1-t)v+tw|^2\,d(\gamma^k)^x(v)\sigma(x)
    \leq 2 \left[\int_{\R^{2d}}|z|\,d\mu^k(z) + t^2\int_{\R^{2d}}|z|^2\,d\nu^k(z)\right].
\end{align*}
As the second moment is lower semicontinuous w.r.t. narrow convergence, the right-hand side is uniformly bounded in $k$, and thus second moment of $\mu_t^k$ is bounded uniformly in $t\in[0,1]$ and $k$. Due to Prokhorov's theorem, any sublevel set of the second moment in $\P(\R^{2d})$ is precompact with respect to narrow convergence. Moreover,
\begin{align*}
    W_{2,v}(\mu^k,\nu^k) &\leq W_{2,v}(\mu^k,\sigma\otimes\delta_0) + W_{2,v}(\sigma\otimes\delta_0,\nu^k) \\
    &= \left(\int_{\R^{2d}} |v|^2\,d\mu^k(x,v)\right)^{1/2}+\left(\int_{\R^{2d}} |v|^2\,d\nu^k(x,v)\right)^{1/2}
\end{align*}
hence $\sup_k W_{2,v}(\mu^k,\nu^k)<+\infty$. Thus $(\mu^k_t)_{t\in[0,1]}$ are uniformly equicontinuous in $(\P_2^v(\R^{2d};\sigma),W_{2,v})$. 

Then we may apply Proposition~\ref{prop:ArzelaAscoli-weak} with the narrow topology and $(X,m)=(\P_2^v(\R^{2d};\sigma),W_{2,v})$, which we know is complete from Proposition~\ref{prop:P2p_W2p_complete}, to obtain $(\mu_t)_{t\in[0,1]}$ such that over a suitable subsequence, which we do not relabel, we have
\[\mu^k_t\rightharpoonup \mu_t \text{ narrowly for all } t\in[0,1].\]
Clearly by the closure of $\P_2^v(\R^{2d};\sigma)$ in the narrow convergence, $\mu_t\in\P_2^v(\R^{2d};\sigma)$ for all $t\in[0,1]$. In particular, $\mu_0=\mu$ and $\mu_1=\nu$ and $t\mapsto\mu_t$ is continuous in $W_{2,v}$ hence in $W_2$.

Let $L_W$ be the Wasserstein length of continuous curves $(\nu_t)_{t\in[0,1]}$ in $(\P_2(\R^{2d}),W_2)$ defined by 
\begin{equation}\label{def:LW}
    L_W((\nu_t)_{t\in[0,1]})=\sup\left\{\sum_{i=1}^N W_2(\nu_{t_{i-1}},\nu_{t_i}):\;N\in\N \text{ and } 0=t_0<\cdots<t_N=1\right\}=\int_0^1 |\mu'|(t)\,dt;
\end{equation}
for the equality on the right, see for instance \cite[Theorem 2.7.6]{BurIva01}. As $W_2$ is l.s.c. w.r.t. narrow convergence, one can readily check from definitions that
$L_W$ is l.s.c. w.r.t. pointwise narrow convergence. Thus by \eqref{eq:W2p_as_length} of Theorem~\ref{thm:BenBre} and that $(\mu_t^k)_{t\in[0,1]}$ are $W_{2,v}$-geodesics from $\mu^k$ to $\nu^k$, we deduce
\[W_{2,v}(\mu,\nu)\leq L_W((\mu_t)_{t\in[0,1]})\leq \liminf_{k\in\N} L_W((\mu^k_t)_{t\in[0,1]})= \liminf_{k\in\N} W_{2,v}(\mu^k,\nu^k).\]
\end{proof}

We now turn our attention to the optimal transport maps along fixed marginals. Often it is convenient to consider the optimal transport map with fixed marginals as a function in $L^2(\mu;\R^{2d})$. However even when $\mu$ is absolutely continuous, $W_{2,v}(\mu,\nu)$ is attained by a family $\{T_{\mu^x}^{\nu^x}\}_{x\in\R^d}$ of optimal transport maps, and a priori it is unclear if the map $(x,v)\mapsto (x,T_{\mu^x}^{\nu^x})$ is Borel-measurable. The following technical proposition allows us to avoid this measurability issue.
\begin{proposition}[Optimal transport map with fixed marginals]\label{prop:OTmap_partial}
    Let $\sigma\in\P_2(\R^d)$ and $\mu,\nu\in\P_2(\R^{2d};\sigma)$. Further suppose $\mu\in\P_2^r(\R^{2d})$. Then there exists a map $T\in L^2(\mu;\R^{2d})$ such that
    \[\|T-\id\|_{L^2(\mu;\R^{2d})}=W_{2,v}(\mu,\nu)\]
    and for all $\bu\in L^2(\mu;\R^{2d})$
    \begin{equation}\label{eq:OTmap_partial}
        \int_{\R^{2d}} \bu(x,v)\cdot T(x,v)\,d\mu(x,v) = \int_{\R^d}\left(\int_{\R^d} \bu(x,v)\cdot (x,T_{\mu^x}^{\nu^x}(v))\,d\mu^x(v) \right)\,d\sigma(x).
    \end{equation}
    Henceforth we simply refer to this map as $(x,v)\mapsto (x,T_{\mu^x}^{\nu^x}(v))$.
\end{proposition}

\begin{proof}
    Let $\mu,\nu$ be as given. As $\mu\in\P_2^r(\R^{2d})$, each $\mu^x\ll\Leb^d$ hence the optimal transport map $T_{\mu^x}^{\nu^x}$is well-defined \cite{GangboMcCann96}. Define the linear functional $L:L^2(\mu;\R^{2d})\rightarrow\R$ by
    \[L(\bu)=\int_{\R^d}\left(\int_{\R^d}\bu(x,v)\cdot (x,T_{\mu^x}^{\nu^x}(v))\,d\mu^x(v)\right)\,d\sigma(x).\]
    Note
    \begin{align*}
    \|(\id_x,T_{\mu^x}^{\nu^x})\|_{L^2(\mu;\R^{2d})}&=\|(\id_x,T_{\mu^x}^{\nu^x})-\id+\id\|_{L^2(\mu;\R^{2d})}\\
    &\leq \|T_{\mu^x}^{\nu^x}-\id_p\|_{L^2(\mu;\R^d)}+\|\id\|_{L^2(\mu;\R^{2d})}= W_{2,v}(\mu,\nu)+\|\id\|_{L^2(\mu;\R^{2d})}<+\infty
    \end{align*}
    and thus
    \begin{align*}
        L(\bu) &\leq \int_{\R^d} \|\bu(x,\cdot)\|_{L^2(\mu^x;\R^{2d})} \|(x,T_{\mu^x}^{\nu^x})\|_{L^2(\mu^x;\R^{2d})}\,d\sigma(x)    \\
        &\leq \|\bu\|_{L^2(\mu;\R^{2d})} \|(\id_x,T_{\mu^x}^{\nu^x})\|_{L^2(\mu;\R^{2d})}
        \leq \|\bu\|_{L^2(\mu;\R^{2d})}(W_{2,v}(\mu,\nu)+\|\id\|_{L^2(\mu;\R^{2d})}).
    \end{align*}
    Thus $L$ is a bounded operator on $L^2(\mu;\R^{2d})$, and thus by the Riesz representation theorem we can find $T\in L^2(\mu;\R^{2d})$ such that
    \[L(\bu)=\int_{\R^{2d}} \bu\cdot T\,d\mu.\]
    Working with Borel representatives of $T,\bu$, we conclude by the disintegration theorem that
    \begin{align*}
        \int_{\R^d}\left(\int_{\R^d} \bu(x,v)\cdot T(x,v)\,d\mu^x(v)\right)\,d\sigma(x)&=\int_{\R^{2d}} \bu\cdot T\,d\mu\\
        =L(\bu)
        &= \int_{\R^d}\left(\int_{\R^d}\bu(x,v)\cdot (x,T_{\mu^x}^{\nu^x}(v))\,d\mu^x(v)\right)\,d\sigma(x)
    \end{align*}
    for any $\bu\in L^2(\mu;\R^{2d})$.
\end{proof}

We conclude this section with the $L^2$-stability of $W_{2,v}$-optimal transport maps.  
\begin{proposition}[Stability of $W_{2,v}$-optimal-transport maps]\label{prop:stability}
    Let $\mu\in\P_2^r(\R^{2d})$, and let $\nu,\nu_n\in\P_2^v(\R^{2d};\Pi^x\mu)$ such that $W_{2,v}(\nu_n,\nu)\xrightarrow[]{n\rightarrow\infty} 0$. Let $T,T_n$ be the corresponding optimal transport maps with fixed marginals defined as
    \[T(x,v)=(x,T_{\mu^x}^{\nu^x}(v)),\qquad T_n(x,v):=(x,T_{\mu^x}^{\nu_n^x}(v))\]
    in the sense of Proposition~\ref{prop:OTmap_partial}. Then
    \begin{equation}\label{eq:W2p-map;L2stability}
        \|T_n-T\|_{L^2(\mu;\R^{2d})}\rightarrow 0 \text{ as } n\rightarrow\infty.
    \end{equation}
\end{proposition}

\begin{proof}
We argue as in the analogous statement for the Wasserstein case \cite[Corollary 5.23]{Vil09}. Define $\pi,\pi_n\in\P_2(\R^d\times\R^d\times\R^d)$ by
\[d\pi(x,v,w)=\int_{\R^d} (\id_v(x,\cdot)\times T^x)_{\#}\mu^x(v,w)\,d\Pi^x(x),\; d\pi_n(x,v,w)=\int_{\R^d} (\id_v(x,\cdot)\times T_n^x)_{\#}\mu^x(v,w)\,d\Pi^x(x).\]
For $\Pi^x\mu$-a.e. $x\in\R^d$, $\nu_n^x\rightarrow \nu^x$ as $W_{2,v}(\nu_n,\nu)\rightarrow 0$. Thus by the stability of the optimal transport map (see for instance \cite[Proposition 7.1.3]{AGS}) and the uniqueness of the optimal transport map,
\[(\id_v(x,\cdot)\times T_n^x)_{\#}\mu^x \rightharpoonup (\id_v(x,\cdot)\times T^x)_{\#}\mu^x \text{ narrowly as } n\rightarrow\infty \text{ for } \Pi^x\mu\text{-a.e. } x\in\R^d.\]
Hence by the dominated convergence theorem $\pi^n$ narrowly converges to $\pi$ as $n\rightarrow\infty$.

Since $T:\R^{2d}\rightarrow\R^{2d}$ is a Borel map, for each $\delta>0$ by Lusin's theorem we can fix a compact set $K\subset\R^{d}\times\R^d$ such that $T\restr_K$ is continuous. Thus for any $\eps>0$ the set
\[
 A_\eps:=\{(x,v,w)\in K\times\R^d:\;|T(x,v)-w|\geq \eps\}
\]
is closed in $K\times\R^d$ hence in $\R^{d}\times\R^d\times\R^d$. Thus, by the narrow convergence of $\pi_n$ to $\pi$,
\begin{align*}
    0=\pi(A_\eps)&\geq \limsup_{n\rightarrow\infty} \pi_n(A_\eps)    
    = \limsup_{n\rightarrow\infty}\mu(\{(x,v)\in K:\;|T(x,v)-T_n(x,v)|\geq \eps\}) \\
    &= \limsup_{n\rightarrow\infty}\mu(\{(x,v)\in \R^d\times\R^d:\;|T(x,v)-T_n(x,v)|\geq \eps\})-\delta,
\end{align*}
which by letting $\delta\downarrow 0$ shows that $T_n$ converges to $T$ in $\mu$-measure. 

On the other hand, $W_2(\nu_n,\nu)\leq W_{2,v}(\nu_n,\nu)$ thus by \cite[Definition 6.8, Theorem 6.9]{Vil09} we have 2-uniform integrability of $(\nu_n)_{n\in\N}$ -- i.e.
\[\lim_{s\rightarrow\infty}\limsup_{n\rightarrow\infty}\int_{\{|z|\geq s\}} |z|^2\,d\nu_n(z)=0.\]
As $(T_n)_\#\mu=\nu_n$, 
\[\int_{\{|T_n|\geq s\}} |T_n(x,v)|^2\,d\mu(x,v) = \int_{\{|z|\geq s\}} |z|^2\,d\nu_n(z)\]
and thus $(T_n)_{n\in\N}$ are $2$-uniformly integrable with respect to $\mu$, and thus by Vitali's convergence theorem we conclude \eqref{eq:W2p-map;L2stability}.
\end{proof}

\section{Metric slope with respect to the Wasserstein distance with fixed marginals}\label{sec:partial_metslope}
As the Wasserstein distance with fixed marginals is a length metric induced by $W_2$ on each subspace of $\P_2(\R^{2d})$ with fixed marginals, it inherits many of metric differential properties of the Wasserstein space with suitable modifications. In this section we record some useful properties on the metric slope of functionals with respect to $W_{2,v}$. Of particular importance is Proposition~\ref{prop:pJKO_minimizer}, which concerns the unique solvability of the problem
\begin{equation}\label{def:partial_JKO;v}
    \mu_{h,v} \in\argmin_{\nu\in\P_2(\R^{2d})} \frac{W_{2,v}^2(\mu,\nu)}{2h} + \E(\nu) \text{ given } h>0 \text{ and }\E:\P_2(\R^{2d})\rightarrow (-\infty,+\infty].
\end{equation}
Note that the minimization above happens in $\P_2^v(\R^{2d},\Pi^\mu)$, as $W_{2,v}(\mu,\nu)=+\infty$ outside of this set.

Then Proposition~\ref{prop:pslope_est} provides slope estimates over the solutions of the variational problem \eqref{def:partial_JKO;v} which play a pivotal role in later sections. We note that most results and ideas in this section are essentially from \cite[Chapters 2,3]{AGS}; when the necessary modifications are obvious, we skip the proofs and refer to the corresponding results of \cite{AGS}.

Given a functional $\E:\P_2(\R^{2d})\rightarrow (-\infty,+\infty]$, we denote by $\D(\E)$ the effective domain of $\E$
\begin{equation}\label{def:D(E)}
    D(\E)=\{\mu\in \P_2(\R^{2d}):\;\E(\mu)<+\infty\}.
\end{equation}

We say that a functional $\E:\P_2(\R^{2d})\rightarrow (-\infty,+\infty]$ is \textbf{lower semicontinuous} with respect to narrow convergence 
\begin{equation}\label{ass:LSC}
   \E(\mu)\leq\liminf_{n\rightarrow\infty} \E(\mu_n) \text{ if }  \mu_n \text{ converges narrowly to } \mu \text{ in }\P_2(\R^{2d})
\end{equation}

Recall that metric slope $|\partial\E|(\mu)$ of $\E$ with respect to the Wasserstein distance is defined by
\begin{equation}\label{def:met_slope}
     |\partial\E|(\mu)=\limsup_{\nu\rightarrow \mu \text{ in }W_2}\frac{[\E(\mu)-\E(\nu)]_+}{W_2(\mu,\nu)}.
\end{equation}

We define $v$-partial metric slope $|\partial_v\E|(\mu)$ of $\E$ at $\mu\in\P_2(\R^{2d})$ by
\begin{equation}\label{def:partial_slope}
    |\partial_v \E|(\mu)=\limsup_{\nu\rightarrow\mu \text{ in } W_{2,v}}\frac{[\E(\mu)-\E(\nu)]_+}{W_{2,v}(\mu,\nu)}.
\end{equation}
    
    We say $\E$ is $\lambda$-geodesically convex along the $v$-variable (or $\lambda$-convex along $W_{2,v}$ geodesics) if for any $\sigma\in\P_2(\R^d)$ and $\mu,\nu\in\P_2^v(\R^{2d};\sigma)$, the constant-speed $W_{2,v}$-geodesic $(\mu_t)_{t\in[0,1]}$ satisfies
    \begin{equation}\label{def:partial_geo_cvx}
        \E(\mu_t)\leq t\E(\mu_0) + (1-t)\E(\mu_1)-\frac{\lambda t(1-t)}{2}W_{2,v}^2(\mu_0,\mu_1).
    \end{equation}

    Note that the condition \eqref{def:partial_geo_cvx} is simply the $\lambda$-geodesic convexity in each $(\P_2^v(\R^{2d};\Pi^x\mu),W_{2,v})$, which is complete (Proposition~\ref{prop:P2p_W2p_complete}). Thus Theorem 2.4.9 of Ambrosio, Gigli, and Savar\'e \cite{AGS} immediately provides the following alternative characterization of the partial local slope.
    \begin{proposition}\label{prop:partial_slope;global}
    Let $\E$ be $\lambda$-geodesically convex along the $v$-variable. Then 
        \begin{equation}\label{eq:partial_slope;global}
        |\partial_v \E|(\mu)=\sup_{\substack{\nu\neq \mu,\\ \nu\in\P_2^v(\R^{2d};\Pi^x\mu)}}\left(\frac{\E(\mu)-\E(\nu)}{W_{2,v}(\mu,\nu)}+\frac12\lambda W_{2,v}(\mu,\nu)\right)^+.
    \end{equation}
    \end{proposition}
    
In order to study the variational problem \eqref{def:partial_JKO;v}
\[
    \mu_{h,v} \in\argmin_{\nu\in\P_2^v(\R^{2d};\Pi^x\mu)} \frac{W_{2,v}^2(\mu,\nu)}{2h} + \E(\nu).
\]
We introduce the following notation: for each $\mu\in\P_2(\R^{2d})$ and $h>0$
\begin{equation}\label{def:Epmuh}
    \E_v(\nu;\mu,h)=\frac{W_{2,v}^2(\mu,\nu)}{2h}+\E(\nu),\qquad \E_x(\nu;\mu,h)=\frac{W_{2,x}^2(\mu,\nu)}{2h}+\E(\nu)
\end{equation}
and
\begin{equation}\label{def:Ehp}
\begin{split}
    \E_{h,v}(\mu)&=\inf_{\nu\in\P_2^v(\R^{2d};\Pi^x\mu)} \E_v(\nu;\mu,h),\\
    \E_{h,x}(\mu)&=\inf_{\nu\in\P_2^x(\R^{2d};\Pi^v\mu)} \E_x(\nu;\mu,h).
\end{split}
\end{equation}

When $\E$ is $\lambda$-convex along $W_{2,v}$-geodesics, observe that for any $\mu,\nu\in\P_2^v(\R^{2d};\sigma)$ and the constant-speed $W_{2,v}$-geodesic $(\mu_t)_{t\in[0,1]}$ from $\mu_0=\mu$ to $\mu_1=\nu$,
\begin{align*}
    \E_v(\mu_t;\mu,h)&=\frac{W_{2,v}^2(\mu,\nu_t)}{2h}+\E(\mu_t)
    \leq\frac{t^2 W_{2,v}^2(\mu,\nu)}{2h} + (1-t)\E(\mu)+t\E(\nu)-\frac{\lambda}{2}t(1-t)W_{2,v}^2(\mu,\nu) \\
    &=(1-t)\E(\mu) + t\left(\E(\nu)+\frac{W_{2,v}^2(\mu,\nu)}{2h}\right)-\frac12\left(\frac1h+\lambda\right)t(1-t)W_{2,v}^2(\mu,\nu).
\end{align*}
Thus $\E_v(\cdot;\mu,h)$ is $(h^{-1}+\lambda)$-convex along $W_{2,v}$-geodesics -- i.e.
\begin{equation}\label{eq:Epmuh_cvx}
    \E_v(\mu_t;\mu,h)\leq (1-t)\E_v(\mu;\mu,h)+t\E_v(\nu;\mu,h)-\frac12\left(\frac1h+\lambda\right)t(1-t)W_{2,v}^2(\mu,\nu).
\end{equation}

We now establish the existence of minimizers of the discrete variational problem \eqref{def:partial_JKO;v}. 
\begin{proposition}[Unique solvability of the variational problem]\label{prop:pJKO_minimizer}
Let $\E:\P_2(\R^{2d})\rightarrow (-\infty,+\infty]$ be $\lambda$-convex along $W_{2,v}$-geodesics for some $\lambda\in\R$. Then for any $\mu\in\P_2(\R^{2d})$ such that
\begin{listi}
    \item there exists $\mu^\ast\in D(\E)\cap\P_2^v(\R^{2d};\Pi^x\mu)$
    \item $\E$ is lower semicontinuous w.r.t. narrow convergence of measures in sublevel sets
    \begin{equation}\label{def:Esublevel}
        \{\nu\in\P_2^v(\R^{2d};\Pi^x\mu):\;\E(\nu)\leq c\}
    \end{equation}
    for each $c\in\R$.
\end{listi} 
there exists a unique minimizer $\mu_{h,v}\in\P_2^v(\R^{2d};\Pi^x\mu)$ to the problem \eqref{def:partial_JKO;v} for any $h< \frac{1}{\lambda_-}$.
\end{proposition}

\begin{proof}
Fix $\mu\in\P_2(\R^{2d})$ and $\mu^\ast\in D(\E)\cap\P_2^v(\R^{2d};\Pi^x\mu)$. We first use $\lambda$-convexity along $W_{2,v}$-geodesics to show the coercivity in $v$ as in \cite[Lemma 2.4.8]{AGS}, namely that $h<\frac{1}{\lambda_-}$ we have
\begin{equation}\label{eq:pcoercive;E}
    \E_{h,v}(\mu^\ast)>-\infty.
\end{equation}
Then using the coercivity we deduce the existence of $\mu_{h,v}$, via the argument in \cite[Lemma 2.2.1, Corollary 2.2.2]{AGS}. Finally, uniqueness follows from the geodesic convexity.
\medskip

\noindent\emph{Step 1$^o$.} (Coercivity and existence of a minimizer) Let $\nu\in\P_2^v(\R^{2d};\Pi^x\mu^\ast)$. Note that
\[m_\ast:=\inf\{\E(\nu):\;W_{2,v}(\mu^\ast,\nu)\leq 1\}>-\infty.\]
Indeed, $W_2\leq W_{2,v}$ and $W_{2,v}$ is narrowly lower semicontinuous, thus the unit ball in $W_{2,v}$ around $\mu^\ast$ is narrowly compact. Moreover, $\E(\cdot;\mu,h)$ is l.s.c. w.r.t. the narrow convergence in the sublevel sets by assumption (ii), hence $m_\ast=\E(\nu_\ast)$ for some $\nu_\ast$ in the unit $W_{2,v}$-ball, and $m_\ast>-\infty$ as $\E$ takes values in $(-\infty,+\infty]$. Thus, independently of $h>0$
\[\inf\{\E_v(\nu;\mu,h):\;W_{2,v}(\mu^\ast,\nu)\leq 1\}\geq m_\ast.\]

Now suppose $W_{2,v}(\mu^\ast,\nu)>1$, and let $(\nu_t)_{t\in[0,1]}$ be the constant-speed $W_{2,v}$-geodesic with $\nu_0=\mu^\ast$ and $\nu_1=\nu$. Then by the convexity assumption on $\E$, reorganizing \eqref{def:partial_geo_cvx} we obtain
\[
\frac{\E(\nu_t)-\E(\mu^\ast)}{t}+\E(\mu^\ast)\leq\E(\nu)-\frac{\lambda}{2}(1-t)W_{2,v}^2(\mu^\ast,\nu)\leq \E(\nu)+\frac{\lambda_-}{2}(1-t)W_{2,v}^2(\mu^\ast,\nu).
\]

For any $\eps\in(0,1)$ choose $t=\eps/W_{2,v}(\mu^\ast,\nu)<\eps$. As $W_{2,v}(\mu^\ast,\nu_t)=tW_{2,v}(\mu^\ast,\nu)\leq \eps<1$, $\E(\nu_t)\geq m_\ast$, hence
\[\frac{m_\ast-\E(\mu^\ast)}{\eps}W_{2,v}(\mu^\ast,\nu)+\E(\mu^\ast)\leq \frac{\E(\nu_t)-\E(\mu^\ast)}{t}+\E(\mu^\ast).\]

If $\lambda\geq 0$ hence $\lambda_-=0$, letting $t=1$ in the above inequality yields, for any $h>0$
\begin{align*}
    \E(\nu)+\frac{1}{2h}W_{2,v}^2(\mu^\ast,\nu)
    &\geq \E(\mu^\ast) + \frac{1}{2h}W_{2,v}^2(\mu^\ast,\nu)+\frac{m_\ast-\E(\mu^\ast)}{\eps}W_{2,v}(\mu^\ast,\nu)   \\
    &\geq \E(\mu^\ast) - \frac{h}{2}\left(\frac{m_\ast-\E(\mu^\ast)}{\eps}\right)^2
\end{align*}
by the Young's product inequality, which verifies \eqref{eq:pcoercive;E}. In case $\lambda_->0$,
\begin{align*}
    \frac{m_\ast-\E(\mu^\ast)}{\eps}W_{2,v}(\mu^\ast,\nu)+\E(\mu^\ast)
    &\leq \frac{\E(\nu_t)-\E(\mu^\ast)}{t}+\E(\mu^\ast)  \\
    &\leq\E(\nu)+\frac{\lambda_-}{2}(1-t)W_{2,v}^2(\mu^\ast,\nu) 
    \leq  \E(\nu)+\frac{\lambda_-}{2}(1-\eps)W_{2,v}^2(\mu^\ast,\nu).
\end{align*}
Thus by applying Young's product inequality and reorganizing,
\begin{align*}
    -\frac{1}{\eps\lambda_-}\left(\frac{m_\ast-\E(\mu^\ast)}{\eps}\right)^2+\E(\mu^\ast)
    \leq \E(\nu)+\frac{\lambda_-}{2}(1-2\eps)W_{2,v}^2(\mu^\ast,\nu).
\end{align*}
As the left-hand side only depends on $\eps>0$, this concludes the proof of \eqref{eq:pcoercive;E} for any $h<\frac{1}{\lambda_-}$.

Let $h_\ast=\lambda_-^{-1}$ if $\lambda<0$ and $+\infty$ if $\lambda\geq 0$. Without loss of generality, we can assume $h_\ast<+\infty$ and show that for $h\in(0,h_\ast)$ \eqref{def:partial_JKO;v} admits a minimizer. By \cite[Lemma 2.2.1, Corollary 2.2.2]{AGS}, coercivity \eqref{eq:pcoercive;E} and lower semicontinuity of $\E$ on sublevel sets imply the existence of the minimizer. 
\medskip

\noindent\emph{Step 2$^o$.} (Uniqueness) Suppose $\mu_{h,v}^1,\mu_{h,v}^2$ are minimizers --i.e. 
\[\E_{v,h}(\mu)=\E_v(\mu_{h,v}^1;\mu,h)=\E_v(\mu_{h,v}^2;\mu,h).\]
Let $(\mu_t)_{t\in[0,1]}$ be a constant-speed $W_{2,v}$-geodesic from $\mu_0=\mu_{h,v}^1$ to $\mu_1=\mu_{h,v}^2$. Then by \eqref{eq:Epmuh_cvx}
\[\E_{v,h}(\mu)\leq \E_v(\mu_t;\mu,h)\leq \E_{v,h}(\mu) - \frac12\left(\frac1h + \lambda\right)t(1-t)W_{2,v}^2(\mu_{h,v}^1,\mu_{h,v}^2).\]
As $h<1/\lambda_-$, we have $\frac1h+\lambda>0$ and thus $W_{2,v}(\mu_{h,v}^1,\mu_{h,v}^2)=0$.
\end{proof}

Due to \cite[Lemma 3.1.5]{AGS} we also have the following characterization of the metric slope in terms of minimizers $\mu_{h,v}$.
\begin{proposition}[Duality formula for metric slope]\label{prop:met_slope;dual}
We have
\begin{equation}\label{eq:met_slope;dual}
    \frac12|\partial_v\E|^2(\mu)=\limsup_{h\searrow 0}\frac{\E(\mu)-\E_{h,v}(\mu)}{h}.
\end{equation}
Moreover, if the infimum of \eqref{def:partial_JKO;v} is attained at $\mu_{h,v}$, there exists a sequence $h_n\searrow 0$ such that
\begin{equation}\label{eq:met_slope;dual_seq}
    |\partial_v\E|^2(\mu)=\lim_{n\rightarrow\infty}\frac{W_{2,v}^2(\mu_{h_n,v},\mu)}{h_n^2}=\lim_{n\rightarrow\infty}\frac{\E(\mu)-\E(\mu_{h_n,v})}{h_n}\geq \liminf_{h\searrow 0}|\partial_v\E|^2(\mu_{h,v}).
\end{equation}
\end{proposition}

\begin{proof}
This is a direct consequence of \cite[Lemma 3.1.5]{AGS} applied to $(\P_2^v(\R^{2d};\Pi^x\mu),W_{2,v})$ which is a complete metric space for each $\mu\in\P_2(\R^{2d})$.
\end{proof}

In general, the minimizers of the variational problems satisfy the following slope bounds.
\begin{proposition}
Let $\mu_{h,v}$ be a minimizer of the variational problem \eqref{def:partial_JKO;v}. Then
\begin{equation}\label{eq:partial_slope;W2bound}
    |\partial_v\E|(\mu_{h,v})\leq \frac{W_{2,v}(\mu_{h,v},\mu)}{h}.
\end{equation}
\end{proposition}
\begin{proof}
    We include the short proof of \cite[Lemma 3.1.3]{AGS} for completeness. \eqref{eq:partial_slope;W2bound} follows from an application of triangle inequality as for all $\nu\in\P_2^v(\R^{2d};\Pi^x\mu)$,
    \begin{align*}
    \E(\mu_{h,v})-\E(\nu)&\leq \frac{1}{2h}(W_{2,v}^2(\nu,\mu)-W_{2,v}^2(\mu_{h,v},\mu))\\
    &= \frac{1}{2h}(W_{2,v}(\nu,\mu)+W_{2,v}(\mu_{h,v},\mu))(W_{2,v}(\nu,\mu)-W_{2,v}(\mu_{h,v},\mu))\\
    &\leq \frac{1}{2h}(W_{2,v}(\nu,\mu)+W_{2,v}(\mu_{h,v},\mu))W_{2,v}(\nu,\mu_{h,v}).
    \end{align*}
    We conclude by dividing both sides by $W_{2,v}(\nu,\mu_{h,v})$ and letting $\nu\rightarrow \mu_{h,v}$.
\end{proof}

Under assumptions of $\lambda$-geodesic convexity in each coordinate, we have a more refined chain of inequalities for the metric slope.
\begin{proposition}[Partial slope estimates]\label{prop:pslope_est}
Let $\E$ be $\lambda$-partial geodesically convex in $v$, and $\mu_{h,v}$ a solution to the variational problem \eqref{def:partial_JKO;v}. Then
\begin{equation}\label{eq:partial_slope_est}
    (1+\lambda h)|\partial_v\E|^2(\mu_{h,v})\leq (1+\lambda h)\frac{W_{2,v}^2(\mu_{h,v},\mu)}{h^2}\leq 2\frac{\E(\mu)-\E_h(\mu)}{h}\leq \frac{1}{1+\lambda h}|\partial_v\E|^2(\mu).
\end{equation}
\end{proposition}

\begin{proof}
We refer the readers to the analogous statement \cite[Theorem 3.1.6.]{AGS} for the proof. 
\end{proof}

\section{Partial subdifferential calculus in the Wasserstein space}\label{sec:psubdiff}
In this section we study the partial-subdifferential calculus in the Wasserstein space. For simplicity we limit our exposition mostly to the case where the base point $\mu\in\P_2^r(\R^{2d})$ has a $\Leb^{2d}$-density, as this allows us to formulate our statements in terms of the (unique) optimal transport maps rather than general transport plans. This will be sufficient for our results in later sections as the entropy functional $\U$ forces our measures to be absolutely continuous. 

After introducing the partial-subdifferentials in the Wasserstein space and establishing basic properties, we will study the partial-subdifferentials of functionals involved in our variational scheme \eqref{eq:SIE}. Section~\ref{ssec:partial_subdiff;U},\ref{ssec:psubdiff:LpLx}, and~\ref{ssec:subdiff_sum} respectively examine the entropy functional $\U$, the linear functionals $\Lfun_v,\Lfun_x$, and the sum of the functionals. 

While many of the results in this section parallel those in \cite{AGS}, often the disintegration introduces subtle technical difficulties. Thus, unlike in Section~\ref{sec:partial_metslope} we provide complete proofs in this section.

We first begin by defining the Fr\'echet subdifferentials with respect to $W_{2,v}$, which we refer to as $v$-partial subdifferentials; an equivalent definition can be found in the work of Peszek and Poyato \cite[Definition 3.26]{PesPoy23}.
\begin{definition}\label{def:psubdiff}
    Given $\mu\in\P_2(\R^{2d})$, we say $\xi=(0,\xi_v)\in L^2(\mu;\R^{2d})$ belongs to the $v$-partial subdifferential $\partial_v\E$ if for any $\nu\in D(\E)\cap\P_2^v(\R^{2d};\Pi^x\mu)$ there exists $\gamma\in\Gamma_o^v(\mu,\nu)$ such that
    \begin{equation}\label{eq:psubdiff}
        \E(\nu)-\E(\mu)\geq \int_{\R^d}\iint_{\R^d\times\R^d}\xi_v(x,v)\cdot (w-v)\,d\gamma^x(v,w)\,d\Pi^x\mu(x)+  o(W_{2,v}(\mu,\nu))
    \end{equation}
\end{definition}

Recall \cite[Definition 10.1.1]{AGS} that $\xi\in L^2(\mu;\R^{2d})$ is a strong subdifferential of $\E$ at $\mu$ if for any 
\[
    \E(T_\#\mu)-\E(\mu)\geq \int_{\R^{2d}} \xi\cdot (T-\id)\,d\mu + o(\|T-\id\|_{L^2(\mu;\R^{2d})}) \text{ for any } T\in L^2(\mu;\R^{2d}).
\]
Thus if $\xi=(\xi_x,\xi_v)$ is a strong subdifferential of $\E$ at $\mu$, then by setting $T\in L^2(\mu;\R^{2d})$ to be the $W_{2,v}$-optimal transport map from $\mu$ to $\nu$ to deduce that $(0,\xi_v)$ is a strong $v$-partial-subdifferential of $\E$ at $\mu$. This justifies the use of the term $v$-partial (or $x$-partial)-subdifferentials.

\begin{proposition}\label{prop:slope_subdiff}
    Let $\mu\in D(|\partial_v\E|)$ and $(0,\xi^v)\in \partial_v\E(\mu)$. Then 
    \begin{equation}\label{eq:slope_subdiff}
        |\partial_v\E|(\mu)\leq\|\xi^v\|_{L^2(\mu)}
    \end{equation}
\end{proposition}

\begin{proof}
Fix $\mu\in\P_2(\R^{2d})$. For each $\nu\in\P_2(\R^{2d})$, choosing $\gamma\in\Gamma_o^v(\mu,\nu)$ such that \eqref{eq:psubdiff}, we see
    \begin{align*}
        \frac{\E(\mu)-\E(\nu)}{W_{2,v}(\mu,\nu)}\leq \frac{\int_{\R^d}\int_{\R^d\times\R^d} \xi^v(x,v)\cdot(w-v)d\gamma^x(v,w)\,d\Pi^x\mu(x)}{W_{2,v}(\mu,\nu)}+ o(1)\leq \|\xi_v\|_{L^2(\mu;\R^d)}+o(1),
    \end{align*}
    hence $|\partial_v\E|(\mu)\leq\|\xi_v\|_{L^2(\mu;\R^d)}$.
\end{proof}

This brings up a natural question of whether there exists an element in $\partial_v\E(\mu)$ whose $L^2(\mu;\R^{2d})$-norm is equal to $|\partial_v\E|(\mu)$. In the case of the Wasserstein space there exists a unique vector field that saturates inequality analogous to \eqref{eq:slope_subdiff}, as documented thoroughly in \cite[Chapter 10]{AGS}. For simplicity, we make the following definition.

\begin{definition}\label{def:minimalsubdiff}
    We say $(0,\xi^v)\in\partial_v^\circ\E(\mu)$ if $(0,\xi^v)\in\partial_v\E(\mu)$ and $\|\xi^v\|_{L^2(\mu)}=|\partial_v\E|(\mu)$.
\end{definition}

The following Euler equations provide a connection between the minimizers of the variational problem \eqref{def:partial_JKO;v} and the $W_{2,v}$-optimal transport maps.
\begin{proposition}\label{prop:euler_eqn}
If $\mu_{h,v}\in\P_2^r(\R^{2d})$ is a minimizer of \eqref{def:partial_JKO;v}, then
\begin{equation}\label{eq:euler_eqn}
    \begin{pmatrix}
    0\\
    \frac{1}{h}(T_{\mu_{h,v}^x}^{\mu^x}-\id_{v})
    \end{pmatrix} \in \partial_v\E(\mu_{h,v}).
\end{equation}
\end{proposition}

\begin{proof}
We modify the proof of \cite[Lemma 10.1.2]{AGS}. Let $\nu\in\P_2^v(\R^{2d};\Pi^x\mu)$. By the definition of $\mu_{h,v}$,
\begin{align*}
    \E(\nu)-\E(\mu_{h,v})\geq \frac{1}{2h}(W_{2,v}^2(\mu_{h,v},\mu)-W_{2,v}^2(\nu,\mu)).
\end{align*}
For any $\nu=(x,T^x(v))_\# \mu_{h,v}$, we have
\begin{align*}
    W_2^2(\mu_{h,v},\mu)=\int_{\R^{2d}} |T_{\mu_{h,v}^x}^{\mu^x}(v)-v|^2\,d\mu_{h,v}(x,v),\qquad W_2^2(\nu,\mu) \leq \int_{\R^{2d}} |T^x(v)-T_{\mu_{h,v}^x}^{\mu^x}(v)|^2\,d\mu_{h,v}(x,v).
\end{align*}
By the identity $\frac12|a|^2-\frac12|b|^2=\langle a,a-b\rangle - \frac12|a-b|^2$,
\begin{equation}\label{eq:euler_eqn;test}
\begin{split}
    \E(\nu)-\E(\mu_{h,v}) &\geq \frac{1}{2h}\int_{\R^{2d}}\left(|T_{\mu_{h,v}^x}^{\mu^x}(v)-v|^2-|T^x(v)-v|^2\right)\,d\mu_{h,v}(x,v)   \\
    &= \int_{\R^{2d}}\langle \frac{1}{h} T_{\mu_{h,v}^x}^{\mu^x}(v)-v,T^x(v)-v\rangle\,d\mu_{h,v}(x,v) - \frac{1}{2h}\int_{\R^{2d}}|T^x(v)-v|^2\,d\mu_{h,v}(x,v).
\end{split}
\end{equation}
In particular, choosing $T^x=T_{\mu_{h,v}^x}^{\nu^x}$ we obtain \eqref{eq:euler_eqn}.
\end{proof}

\subsection{Convexity and partial subdifferentials of the entropy functional}\label{ssec:partial_subdiff;U}
The minimizing movement scheme \eqref{def:SIE} minimizes $\Lfun_v+\alpha\H$ in the first step using the metric $W_{2,v}$, and minimizes in the second step $-\Lfun_x$ penalizing deviation in $W_{2,x}$. All other functionals but $\U$ have relatively simple structure with respect to the relevant variable: $\Lfun_v$ is linear in $v$ and $-\Lfun_x$ in $x$; moreover, $\H=\V+\W+\U$ depends on $v$ only through $\mu\mapsto \int_{\R^{2d}} |v|^2/2\,d\mu(x,v)$ and $\U$. While $v\mapsto |v|^2/2$ is clearly $1$-convex along $W_{2,v}$-geodesics, but convexity and $v$-subdifferential structure of $\U$ deserve careful studies as $\U$ depends on the full density of $\mu$, not just of $\Pi^v\mu$. 

Thus we dedicate this section to carefully examine the convexity and $v$-partial subdifferential of the entropy functional
\begin{equation}\label{def:U}
    \U(\mu)= 
    \begin{cases}
    \int_{\R^{2d}} \rho\log\rho\,d\Leb^{2d} &\text{ if } \mu=\rho\Leb^{2d}, \\
    +\infty &\text{ otherwise. }
    \end{cases}
\end{equation}
We note that the results in this section are suitable adaptation of results from \cite[Chapter 10.4.3]{AGS} for out setting.

Rigorous definitions of the directional derivative and subdifferentials of $\U$ require the notion of approximate differentials, which we define here. See \cite[Definition 5.5.1]{AGS} and following discussions as well as \cite{Federer69} for further details.
\begin{definition}[Approximate differentials]\label{def:tildenabla}
    Let $\Omega\subset\R^k$ be an open set and $f:\Omega\rightarrow \R^m$. We say that approximate limit of $f$ at $x\in\Omega$ is $z\in\R^m$ if all sets
    \[\{y:\;|f(y)-z|>\eps\} \quad \eps >0\]
    have density $0$ at $x$. In this case we write $\tilde f(x)=z$.

    We say a linear map $L:\R^k\rightarrow\R^m$ is the approximate differential of $f$ at $x$ if $f$ has an approximate limit $\tilde f(x)$ at $x$ and all sets
    \[\left\{y:\;\frac{|f(y)-\tilde f(x)-L(y-x)|}{|y-x|}>\eps\right\} \quad \eps>0\]
    have density $0$ at $x$, and we write $\tilde\nabla f(x)=L$. 
\end{definition}
When $f:(x,v)\mapsto f(x,v)$ with $(x,v)\in\R^{2d}$, we often write $\tilde\nabla_v f$ to emphasize that the differentiation takes place with respect to the velocity variable $v$.

Now we present a useful lemma on the directional derivative of $\U$ along the $v$-coordinate. As the proof is an adaptation of the analogous statement \cite[Lemma 10.4.4]{AGS}, we delay it to Section~\ref{app:proof;U}.

\begin{lemma}[Partial-directional derivative of $\U$]\label{lem:pdirder;U}
Let $\U$ be the entropy functional defined in \eqref{def:U}. Let $\mu\in D(\U)$ with $\mu=\rho\Leb^{2d}$. 

For $\Pi^x \mu$-a.e. $x\in\R^d$, let $\br\in L^2(\mu;\R^{2d})$ with $\br(x,v)=(x,\br^x(v))$ and $\bar t>0$ such that
\begin{listi}
 \item For $\Pi^x\mu$-a.e. $x\in\R^d$, $\br^x$ is approximately differentiable $\rho(x,\cdot)\Leb^{d}$-a.e. and $\br_t^x:=(1-t)\id+t\br^x$ is $\rho\Leb^{d}$-injective with $|\det\tilde\nabla_v\br_t^x(v)|>0$ $\rho(x,\cdot)\Leb^{d}$-a.e., for any $t\in[0,\bar t]$;
\item $\tilde\nabla_v\br^x_{\bar t}$ is diagonalizable with positive eigenvalues;
\item $\U((\br_{\bar t})_{\#}\mu)<+\infty$.
\end{listi}
Then the map $t\mapsto \frac{\U((\br_t)_{\#}\mu)-\U(\mu)}{t}$ is nondecreasing; in particular, $\br^x(v)=T_{\mu^x}^{\nu^x}$ for any $\nu\in D(\U)$ satisfies (i)-(iii) and thus $\U$ is convex along $W_{2,v}$-geodesics.

Furthermore, we have the directional derivative
\begin{equation}\label{eq:pdirder;U}
    +\infty>\lim_{t\downarrow 0}\frac{\U((\br_t)_{\#}\mu)-\U(\mu)}{t}=-\int_{\R^{2d}} \tr\tilde\nabla_v(\br^x-\id_v)\,d\mu.
\end{equation}

Moreover, if (ii) is replaced by
\begin{equation}\label{cond:xi_v;Linftybd}
\|\tilde\nabla_v(\br^x-\id_v(x,\cdot))\|_{L^\infty(\mu;\R^{2d})}<+\infty
\end{equation}
the conclusion \eqref{eq:pdirder;U} still holds.
\end{lemma}

\begin{remark}[Convexity of $\U$ along generalized $W_{2,v}$-geodesics]\label{rmk:Ucvx;gengeo}
    The entropy functional $\U$ satisfies the following stronger convexity property called \emph{convexity along generalized $W_{2,v}$-geodesics}: for any $W_{2,v}$-optimal transport maps $\br_0,\br_1\in L^2(\mu;\R)$ defining $\br_t=(1-t)\br_0+t\br_1$, the map $t\mapsto U\circ\br_t$ is convex. One can easily check that the properties (i)-(iii) are satisfied (see the end of proof of \cite[Proposition 9.3.9]{AGS} for the Wasserstein case). The same convexity property can be easily established for other functionals involved in our algorithm \eqref{eq:SIE}. 
    
    Convexity along generalized geodesics is of importance as the map $\mu\mapsto W_{2,v}^2(\mu,\nu)$ for some fixed $\nu$ is convex along \emph{some} generalized geodesics, while it is in fact concave along geodesics. This forms the basis of the stronger results in \cite[Chapter 4]{AGS}.
\end{remark}

Following is a modification of \cite[Theorem 10.4.6.]{AGS} for partial subdifferentials of the entropy functional. For completeness we include the proof in Section~\ref{app:proof;U}.
\begin{theorem}[Slope and partial subdifferential of $\U$]\label{thm:psubdiff;U}
Let $\U$ be the entropy functional defined in \eqref{def:U}. Let $\mu\in\P_2^r(\R^{2d})$ with $\mu=\rho\Leb^{2d}$. Then
$|\partial_v\U|(\mu)<+\infty$ \textbf{if and only if} $\nabla_v\rho\in L^1(\R^{2d})$ and $\nabla_v\rho=w_v\rho$ for some $w_v\in L^2(\mu;\R^d)$. 

Moreover, in this case $(0,w_v)\in\partial_v^\circ\U(\mu)$, and we write $w_v=\nabla_v\rho/\rho$.
\end{theorem}

\begin{remark}[Logarithmic partial gradient]\label{rmk:dvlogrho}
Note that whenever $\nabla_v\rho\in L^1(\R^{2d})$, we have for any $\xi\in C_c^\infty(\R^{2d};\R^d)$
\begin{equation}\label{eq:dvlogrho;IBP}
\int_{\R^{2d}} \nabla_v\rho/\rho(x,v)\cdot\xi(x,v)\,d\mu=-\int_{\R^{2d}} \nabla_v\cdot\xi^v(x,v)\,d\mu(x,v)
\end{equation}
and by standard approximation this holds for bounded continuous vector fields $\xi$ with bounded deriatives in $v$.

Note also that when $\mu\in D(|\partial_v\U|)$, the $v$-logarithmic gradient $\nabla_v\rho/\rho$ of the density $\rho$ of $\mu$ satisfies
\begin{equation}\label{eq:xidvlogrho=0}
\langle \xi,\nabla_v\rho/\rho\rangle_{L^2(\mu)} = 0 \text{ for any } \xi\in L^2(\mu;\R^d) \text{ depending only on the } x\text{-variable}.
\end{equation}
Indeed, $\int_{\R^d}\nabla_v\rho\,d\Leb^d=0$ by \eqref{eq:dvlogrho;IBP}, and thus by the definition of $\nabla_v\rho/\rho$ we have
\begin{align*}
   \langle \xi,\nabla_v\rho/\rho\rangle_{L^2(\mu)} = \int_{\R^{2d}} \xi(x)\cdot \nabla_v\rho(x,v)\,d\Leb^{2d}
   = \int_{\R^d} \xi(x)\cdot \left(\int_{\R^d} \nabla_v\rho(x,v)\,d\Leb^d(x)\right)\,d\Leb^d(x)  =0.
\end{align*}
The application of Fubini's theorem is justified by the fact that $|\langle \xi,\nabla_v\rho/\rho\rangle_{L^2(\mu)}|<+\infty$.
\end{remark}

\begin{remark}\label{rmk:psubdiff;Up}
Consider $\U_v$ as defined in \eqref{def:UpHp} instead of $\U$. Then for each $\Pi^x\mu$-a.e. $x\in\R^d$ the analogous result to Theorem~\ref{thm:psubdiff;U} (i),(ii) with
\[w_v(x,v)=\frac{\nabla_v\rho^x}{\rho^x}(v) \text{ where } \rho^x(v)d\Leb^d(v)=d\mu^x(v)\]
follows immediately from applying \cite[Theorem 10.4.6]{AGS} to each $\mu^x$ and integrating with respect to $\Pi^x\mu$. Note that if $\mu=\rho\Leb^{2d}$, then $\rho^x(v)=\frac{\rho(v)}{\int \rho(x,v)\,dv}$ and $\frac{\nabla_v\rho^x}{\rho^x}=\frac{\nabla_v\rho(x,\cdot)}{\rho(x,\cdot)}$. Thus the $v$-subdifferential is consistent with that of $\U$, as
\begin{align*}
    \U_v(\nu)-\U_v(\mu)\geq \int_{\R^d}\left(\int_{\R^d} w_v\cdot(T_{\mu^x}^{\nu^x}-\id_v)\,\rho^x(v)\,dv\right)\,d\Pi^x\mu(x)
    = \int_{\R^{2d}} w_v\cdot (T_{\mu^x}^{\nu^x}-\id_v)\,d\mu(x,v).
\end{align*}
\end{remark}

\begin{corollary}\label{cor:psubdiff;UPhih}
    Let $\mu=\rho\Leb^{2d}\in D(\U)\cap D(|\partial_v\U|)$ with $(0,w_v)\in\partial_v^\circ\U(\mu)$. Define $\Phi:\R^{2d}\rightarrow\R^{2d}$ by
    \[\Phi_h(x,v)=(x+hv,v).\]
    Then
    \[(0,w_v\circ\Phi_h^{-1})\in\partial_v^\circ\U((\Phi_h)_\#\mu).\]
\end{corollary}

\begin{proof}
Let $\mu_h=(\Phi_h)_\#\mu$ and let $\rho_h$ be its $\Leb^{2d}$-density. Then by the push-forward formula \cite[Lemma 5.5.3]{AGS}
\begin{equation}\label{eq:rho_vhih}
    \rho_h(\Phi_h(x,v))= \frac{\rho(x,v)}{|\det\nabla\Phi_h(x,v)|}=\rho(x,v)
\end{equation}
as $\det \nabla\Phi_h(x,v)= \det\begin{pmatrix} I_d & h I_d\\ 0 & I_d\end{pmatrix} = 1$. As $\rho_h=\rho\circ\Phi_h^{-1}$ and $\det \nabla\Phi_h \equiv 1$, we deduce $\U(\mu_h)=\U(\mu)$.

As $\mu\in D(|\partial_v\U|)$, by Theorem~\ref{thm:psubdiff;U} we have $\int_{\R^d}|\nabla_v\rho|\,dz<+\infty$ and $\nabla_v\rho/\rho\in L^2(\mu;\R^d)$. As $\nabla_v\Phi_h^{-1}=I_d$, the chain rule in $W^{1,1}(\R^{2d})$ allows us to deduce
\[\nabla_v\rho_h = \nabla_v\rho\circ\Phi_h^{-1}\in L^1(\R^{2d}).\]
Hence
\[\nabla_v\rho_h = \nabla_v\rho\circ\Phi_h^{-1} = ((\nabla_v\rho/\rho)\circ\Phi_h^{-1}) (\rho\circ\Phi^{-1})=(\nabla_v\rho/\rho)\circ\Phi_h^{-1} \rho_h,\]
and further
\[\int_{\R^{2d}}|(\nabla_v\rho/\rho)\circ\Phi_h^{-1}|^2\,d\mu_h= \int_{\R^{2d}}|\nabla_v\rho/\rho|^2\,d\mu=|\partial_v\U|(\mu).\]
Thus, by the last part of Theorem~\ref{thm:psubdiff;U} we deduce $(0,(\nabla_v\rho/\rho)\circ\Phi_h^{-1})\in\partial_v^\circ\U(\mu_h)$.
\end{proof}

\subsection{Partial subdifferential of the linear functionals}\label{ssec:psubdiff:LpLx}

In this section we consider the partial subdifferentials of the two functionals $\Lfun_x,\Lfun_v:\P_2(\R^{2d})\rightarrow\R$ linear respectively in $x$ and $v$,
\begin{equation}\label{def:LxLp}
    \Lfun_x(\mu)=\int_{\R^{2d}} x\cdot v\,d\mu(x,v),\qquad \Lfun_v(\mu)=\int_{\R^{2d}} v\cdot (\nabla_x V(x)+\nabla_x W\ast \Pi^x\mu)\,d\mu(x,v).
\end{equation}

For the functional $\Lfun_v$ to be well-defined on $\P_2(\R^{2d})$, we need to ensure 
\begin{equation}\label{eq:W_ass}
    \nabla_x W\ast\sigma\in L^1(\sigma) \text{ for all } \sigma\in\P_2(\R^d).
\end{equation}
Unless otherwise stated, we will assume throughout the paper that \eqref{eq:W_ass}. Remark~\ref{rmk:W_assump} provides sufficient conditions, one of which is our main assumption that $\nabla_x W$ is Lipschitz continuous, which implies $\nabla_x W\ast\sigma\mu\in L^2(\sigma)$ for all $\sigma\in\P_2(\R^d)$.

\begin{remark}[Assumptions on $W$]\label{rmk:W_assump}
Suppose $\nabla_x W$ is Lipschitz. Then $\nabla_x W\ast\Pi^x\mu\in L^2(\mu)$ at any $\mu\in\P_2(\R^{2d})$, as $\nabla_x W$ grows at most linearly and thus for some $C>0$
\begin{align*}
    \int_{\R^d}\left|\int_{\R^{d}} \nabla_x W(x-y)\,d\Pi^x\mu(y)\right|^2d\Pi^x\mu(x)\leq C \int_{\R^d}\int_{\R^{d}} |\nabla_x W(0)+|x-y||^2\,d\Pi^x\mu(y)\,d\Pi^x\mu(x)\\
    \leq C \int_{\R^{d}} 1+ |x|^2+|y|^2\,d\Pi^x\mu(y)\,d\Pi^x\mu(x)<+\infty.
\end{align*}

Alternatively, let $W\in C^1(\R^d)$ and suppose there exists $M>0$ such that $W_M(x):=W(x)+\frac{M}{2}|x|^2$ is convex, bounded from below, and satisfy the doubling condition -- i.e. there exists $C=C(W)>0$ such that
    \begin{equation}\label{eq:doubling_cond;W}
        W_M(x+y)\leq C(1+W_M(x)+W_M(y)) \text{ for all }x,y\in\R^d.
    \end{equation}
Then we can apply \cite[Lemma 10.4.10]{AGS} to $W_M$ to deduce that $\nabla_x W\ast\sigma\in L^1(\sigma)$ for any $\sigma\in\P_2(\R^d)$ such that $\int_{\R^d} W\ast\sigma\,d\sigma<+\infty$. 
\end{remark}

\begin{remark}\label{rmk:Jp_convexity}
    As $\Lfun_v$ is linear in $v$ whereas $\alpha\H$ is $\alpha$-geodesically convex, the minimization problem $J_h^v$ is $\alpha$-convex along $W_{2,v}$-geodesics \emph{independently} of the choice of $V$ and $W$. Similarly, the minimization problem $J_h^x$ is linear along $W_{2,x}$-geodesics.
\end{remark}

While $|\Lfun_x(\mu)|\leq \frac12\int_{\R^{2d}}|\id|^2\,d\mu<+\infty$ for any $\mu\in\P_2(\R^{2d})$, the same is not true for $\Lfun_v$ in full generality. Thus we first establish some basic results about the linear functional $\Lfun_v$.

\begin{proposition}\label{prop:dom_dvLp}
    Let $V,W\in C^1(\R^d)$ and $\mu\in\P_2(\R^{2d})$, and suppose $\nabla_x W\ast\Pi^x\mu \in L^1(\mu)$. Then $\mu\in D(|\partial_v\Lfun_v|)$ \textbf{ if and only if } $\nabla_x V+\nabla_x W\ast\Pi^x\mu \in L^2(\mu)$.
\end{proposition}

\begin{proof}
    By assumption $\nabla_x W\ast\Pi^x \mu\in L^1(\mu)$. For convenience we write
    \[F_{\Pi^x\mu}:=\nabla_x V+\nabla_x W\ast\Pi^x\mu.\]

    Suppose $|\partial_v\Lfun_v|(\mu)<+\infty$. Fix $u\in C_c^\infty(\R^{2d};\R^d)$ and let $\mu_t=(\id_v+t u)_\#\mu$. Then $W_{2,v}(\mu,\mu_t)\leq t\|u\|_{L^2(\mu;\R^d)}$ and thus
    \begin{align*}
        \int_{\R^{2d}} F_{\Pi^x\mu}\cdot v(x,v)\,d\mu = \frac{\Lfun_v(\mu_t)-\Lfun_v(\mu)}{t}
        \leq \frac{\Lfun_v(\mu_t)-\Lfun_v(\mu)}{W_{2,v}(\mu,\mu_t)}\|u\|_{L^2(\mu)} \leq |\partial_v\Lfun_v|(\mu)\|u\|_{L^2(\mu)}.
    \end{align*}
    Thus by duality in $L^2(\mu)$ we deduce $F_{\Pi^x\mu}\in L^2(\mu;\R^d)$.

    Now suppose $F_{\Pi^x\mu}\in L^2(\mu;\R^d)$. Define
    \[\Psi_h(x,v)=\begin{pmatrix}x \\ v+h F_{\Pi^x\mu}(x)\end{pmatrix}.\]
    Let $\mu_h=(\Psi_h)_{\#}\mu$ and note
    \[\int_{\R^{2d}}|x|^2+|v|^2\,d\mu_h(x,v)\leq\int_{\R^{2d}} |x|^2 2|v|^2+ 2h^2 |F_{\Pi^x\mu}(x)|^2\,d\mu(x,v)<+\infty,\]
    thus $\mu_h\in\P_2(\R^{2d})$. It suffices to show that $\mu_h$ is a minimizer of $\Lfun_v(\cdot;\mu,h)$, as then by the duality of slope (Proposition~\ref{prop:met_slope;dual})
    \[\frac12|\partial_v\Lfun_v|=\limsup_{h\searrow 0} \frac{\Lfun_v(\mu)-\Lfun_v(\mu_h)}{h}=\|F_{\Pi^x\mu}\|_{L^2(\mu;\R^d)}<+\infty.\]
    To this end, note that for any $\nu\in\P_2^v(\R^{2d};\Pi^x\mu)$ and $\gamma\in\Gamma_o^v(\mu,\nu)$,
    \begin{align*}
        \Lfun_v(\nu)-\Lfun_v(\mu) = \int_{\R^d} \int_{\R^d} F_{\Pi^x\mu}(x)\cdot(w-v)\,d\gamma^x(v,w)\,d\Pi^x(x) \geq -\frac{h}{2}\|F_{\Pi^x\mu}\|_{L^2(\mu)}^2 - \frac{W_{2,v}^2(\mu,\nu)}{2h}.
    \end{align*}
    Thus, reorganizing,
    \begin{align*}
        \Lfun_v(\nu;\mu,h)\geq -\frac{h}{2}\|F_{\Pi^x\mu}\|_{L^2(\mu)}^2
        =\frac{\|h F_{\Pi^x\mu}\|_{L^2(\mu;\R^d)}^2}{2h}-h\|F_{\Pi^x\mu}\|_{L^2(\mu_h;\R^d)}^2
        =\Lfun_v(\mu_h;\mu,h).
    \end{align*}
\end{proof}

\subsection{Partial subdifferential of the sum of the functionals}\label{ssec:subdiff_sum}
Based on the study of each functionals in Sections~\ref{ssec:partial_subdiff;U}-\ref{ssec:psubdiff:LpLx}, we finally examine the partial subdifferential of the sum of functionals $\Lfun_v+\alpha\H$ which appears in the variational scheme \eqref{eq:SIE}.

We first use the argument of Jordan-Kinderlehrer-Otto \cite{JKO98} (see also \cite[Lemma 2.1]{CarGan04}) to derive the Euler-Lagrange equation associated to the steepest descent in the velocity step.
\begin{lemma}[Euler-Lagrange equation for $\Lfun_v+\alpha\H$]\label{lem:euler-lagrange;Jp}
    Let $\mu\in D(\Lfun_v+\alpha\H)$ and $\mu_{h,v}\in J_h^v(\mu)$. Then the optimal transport map $T_{\mu_{h,v}^x}^{\mu^x}$ from $\mu_{h,v}^x$ to $\mu^x$ is given by
    \begin{equation}\label{eq:euler-lagrange;Jp}
        T_{\mu_{h,v}^x}^{\mu^x}(v)=v+h\left(\nabla_x V(x)+(\nabla_x W\ast\Pi^x\mu)(x)+\alpha v+\alpha\frac{\nabla_v \rho_{h,v}}{\rho_{h,v}}(x,v)\right).
    \end{equation}
\end{lemma}
\begin{proof}
    Let $\xi=(0,\xi_v)$ with $\xi_v\in C_c^\infty(\R^{2d};\R^d)$. Letting $\nu=\nu_t=(\id+t\xi)_{\#}\mu_{h,v}$ in \eqref{eq:euler_eqn;test}, we have
    \begin{align*}
        &\frac{(\Lfun_v+\alpha\H)(\nu_t)-(\Lfun_v+\alpha\H)(\mu_{h,v})}{t}\\
        &\qquad\qquad\geq\frac{1}{h}\int_{\R^{2d}} (T_{\mu_{h,v}^x}^{\mu^x}(v)-v,\xi^v(x,v))\cdot\xi^v(x,v)\,d\mu_{h,v}(x,v) - \frac{t}{2h}\|\xi\|_{L^2(\mu_{h,v})}^2.
    \end{align*}
    Arguing for instance as in the proof of \cite[Theorem 5.1]{JKO98} we can verify
    \begin{align*}
        &\lim_{t\searrow 0}\frac{(\Lfun_v+\alpha\H)(\nu_t)-(\Lfun_v+\alpha\H)(\mu_{h,v})}{t}\\
        &\qquad=\int_{\R^{2d}} (\nabla_x V(x)+(\nabla_x W\ast\Pi^x\mu)(x,v)+\alpha v+\alpha\frac{\nabla_v \rho}{\rho}(x,v))\cdot\xi^v(x,v)\,d\mu_{h,v}(x,v).
    \end{align*}
    Thus letting $t\searrow 0$ in the first displayed equation of the proof, we have
    \begin{align*}
        \int_{\R^{2d}} (\nabla_x V(x)+(\nabla_x W\ast\Pi^x\mu)(x,v)+\alpha v+\alpha (\nabla_v \rho_{h,v}/\rho_{h,v})(x,v))\cdot\xi^v(x,v)\,d\mu_{h,v}\\
        \geq \frac{1}{h}\int_{\R^{2d}} (T_{\mu_{h,v}^x}^{\mu^x}(v)-v)\cdot\xi^v(x,v)\,d\mu_{h,v}(x,v),
    \end{align*}
    where $\nabla_v\rho/\rho$ satisfies, by \eqref{eq:dvlogrho;IBP}
\[\int_{\R^{2d}} \nabla_v\rho/\rho(x,v)\cdot\xi^v(x,v)\,d\mu_{h,v}=-\int_{\R^{2d}} \nabla_v\xi^v(x,v)\,d\mu_{h,v}(x,v).\]
    By the same computation using $-\xi$ instead of $\xi$ we deduce the inequality above is in fact an equality -- i.e. for any $\xi^v\in C_c^\infty(\R^{2d};\R^d)$
    \begin{equation}\label{eq:euler-lagrange;Jp;test}
    \begin{split}
        \int_{\R^{2d}} \left[\left(\nabla_x V(x)+(\nabla_x W\ast\Pi^x\mu)(x,v)+\alpha(v+\nabla_v \rho_{h,v}/\rho_{h,v}\right)(x,v))\right.& \\
        \left.-\frac1h(T_{\mu_{h,v}^x}^{\mu^x}(v)-v)\right]&\cdot\xi^v(x,v)\,d\mu_{h,v}=0.
    \end{split}
    \end{equation}
    From this we deduce \eqref{eq:euler-lagrange;Jp}.
    
\end{proof}

\begin{lemma}[$v$-partial-subdifferential of $\Lfun_v+\alpha\H$]\label{lem:psubdiff;velfunc}
    Let $\nabla_x V,\nabla_x W$ be Lipschitz continuous, and let $\mu\in D(|\partial_v\U|)$. Then $\mu\in D(|\partial_v(\Lfun_v+\alpha \H)|$ and
    \begin{equation}\label{eq:dvLpH}
        \begin{pmatrix} 0 \\ \nabla_x V+\nabla_x W\ast\Pi^x\mu + \alpha\id_v + \alpha\frac{\nabla_v\rho}{\rho}\end{pmatrix}\in\partial_v^\circ(\Lfun_v+\alpha\H)(\mu).
    \end{equation}
    Further suppose $\nabla_x V$ is $M$-Lipschitz. Then for each $\mu\in D(\H)\cap D(|\partial_v(\Lfun_v+\alpha\H)|)$
    \begin{equation}\label{eq:dvE_xstep;ubd}
        |\partial_v(\Lfun_v+\alpha\H)|((\Phi_h)_\#\mu)\leq |\partial_v(\Lfun_v+\alpha\H)|(\mu)+3Mh\|\id_v\|_{L^2(\mu)}.
    \end{equation}

\end{lemma}

\begin{proof}
As $\nabla_x V,\nabla_x W$ are Lipschitz continuous, $\nabla_x V + \nabla_x W\ast\Pi^x\mu\in L^2(\mu)$ by Remark~\ref{rmk:W_assump}. Furthermore, writing $\mu=\rho\Leb^{2d}$, $\mu\in D(|\partial_v\U|)$ implies, by Theorem~\ref{thm:psubdiff;U}, that $\nabla_v\rho/\rho\in L^2(\mu)$. Thus the left-hand side of \eqref{eq:dvLpH} is in $L^2(\mu)$.
\medskip

\noindent\emph{Step 1$^o$.} In this step we show that
    \[
    \begin{pmatrix} 0 \\ \nabla_x V + \nabla_x W\ast\Pi^x\mu + \alpha\id_v + \alpha\frac{\nabla_v\rho}{\rho}\end{pmatrix}\in \partial_v(\Lfun_v+\alpha\H)(\mu).\]
To see this, note that for any $\Pi^x\nu=\Pi^x\mu$, we have 
\begin{align*}
    (\V+\W)(\nu)-(\V+\W)(\mu)=\frac12\int_{\R^{2d}} |v|^2\,d(\nu-\mu)=\frac12\int_{\R^{2d}} |T_{\mu^x}^{\nu^x}(v)|^2-|v|^2\,d\mu(x,v)   \\
    =\int_{\R^{2d}} v\cdot(T_{\mu^x}^{\nu^x}(v)-v)+\frac12|T_{\mu^x}^{\nu^x}-v|^2\,d\mu(x,v).
\end{align*} 
Thus, by linearity of $\Lfun_v$ and the convexity and the $v$-partial differential of $\U$ (Theorem~\ref{thm:psubdiff;U}),
\begin{align*}
    &(\Lfun_v+\alpha\H)(\nu)-(\Lfun_v+\alpha\H)(\mu)=\Lfun_v(\nu-\mu) + \alpha(\V+\W)(\nu)-\alpha(\V+\W)(\mu) + \alpha\U(\nu)-\alpha\U(\mu)\\
    &\geq \int_{\R^{2d}} (\nabla_x V+\nabla_x W\ast\Pi^x\mu+\alpha \id_v)\cdot (T_{\mu^x}^{\nu^x}-\id_v)\,d\mu + \frac{\alpha}{2}W_{2,v}^2(\mu,\nu)
    +\int_{\R^{2d}} \alpha\frac{\nabla_v\rho}{\rho}\cdot (T_{\mu^x}^{\nu^x}-\id_v)\,d\mu.
\end{align*}
Hence by Proposition~\ref{prop:slope_subdiff}
\[|\partial_v(\Lfun_v+\alpha\H)|(\mu)\leq \|\nabla_x V + \nabla_x W\ast\Pi^x\mu+\alpha\id_v+\alpha\nabla_v\rho/\rho\|_{L^2(\mu;\R^d)}<+\infty.\]
On the other hand,
\begin{equation}\label{eq:subdiff_leq_dvE}
    \|\nabla_x V + \nabla_x W\ast\Pi^x\mu+ \alpha\id_v + \alpha\nabla_v\rho/\rho\|_{L^2(\mu)}\leq |\partial_v(\Lfun_v+\alpha\H)|(\mu)
\end{equation}
follows by obvious modification of the argument in proof of Theorem~\ref{thm:psubdiff;U} (i), as
fixing $v\in C_c^\infty(\R^d;\R^d)$ and $\br(x,v):=(x,\br^x(v))$ we have
\begin{align*}
    -&\int_{\R^{2d}} (\nabla_x V +\nabla_x W\ast\Pi^x\mu + \alpha\id_v+\alpha\nabla_v\rho/\rho)\cdot v\,d\mu \\
    &= -\int_{\R^{2d}} (\nabla_x V + \nabla_x W\ast\Pi^x\mu + \alpha\id_v)\cdot v\,d\mu + \alpha\int_{\R^{2d}} \int_{\R^{2d}}\tr(\nabla_v v)\rho\,dz \\
    &= \lim_{t\downarrow 0} \frac{(\Lfun_v+\alpha\H)(\mu)-(\Lfun_v+\alpha\H)((\br_t)_{\#}\mu)}{t} 
    \\
    &\leq \limsup_{t\downarrow\infty}\frac{(\Lfun_v+\alpha\H)(\mu)-(\Lfun_v+\alpha\H)((\br_t)_{\#}\mu)}{W_{2,v}(\mu,(\br_t)_{\#}\mu)}\frac{t\|v\|_{L^2(\mu)}}{t}
    \leq |\partial_v(\Lfun_v+\alpha\H)|(\mu)\|v\|_{L^2(\mu)}.
\end{align*}
Thus by duality in $L^2(\mu)$ we deduce \eqref{eq:subdiff_leq_dvE}.
\medskip

\noindent\emph{Step 2$^o$.}
Now it remains to show \eqref{eq:dvE_xstep;ubd}. From Corollary~\ref{cor:psubdiff;UPhih} and \eqref{eq:subdiff_leq_dvE}, we know
\[\nabla_x V+\nabla_x W\ast\Pi^x((\Phi_h)_\#\mu) + \alpha\id_v + \alpha(\nabla_v \rho/\rho)\circ\Phi_h^{-1} \in \partial_v^\circ(\Lfun_v+\alpha\H)((\Phi_h)_{\#}\mu).\]
Thus by the triangle inequality
\begin{align*}
&|\partial_v(\Lfun_v+\alpha\H)|((\Phi_h)_\#\mu)-|\partial_v(\Lfun_v+\alpha\H)|(\mu)  \\
&=   
\|\nabla_x V+\nabla_x W\ast \Pi^x((\Phi_h)_\#(\mu))+\alpha\id_v+\alpha(\nabla_v\rho/\rho)\circ\Phi_h^{-1}\|_{L^2((\Phi_h)_\#\mu)}\\
&\qquad-\|\nabla_x V+\nabla_x W\ast \Pi^x\mu+\alpha\id_v + \alpha\nabla_v\rho/\rho\|_{L^2(\mu)}    \\
&\leq \|\nabla_x V\circ\Phi_h - \nabla_x V\|_{L^2(\mu)}+\|(\nabla_x W\ast \Pi^x((\Phi_h)_\#(\mu)))\circ\Phi_h- \nabla_x W\ast \Pi^x\mu\|_{L^2(\mu)}.
\end{align*}
Observe
\[|(\nabla_x V\circ\Phi_h)(x)-\nabla_x V(x)|=|\nabla_x V(x+hv)-\nabla_x V(x)|\leq Mh\|\id_v\|_{L^2(\mu)}\]
whereas
\begin{align*}
    &|[\nabla_x W\ast \Pi^x((\Phi_h)_\#(\mu))](x+hv)-[\nabla_x W\ast \Pi^x\mu](x)|
    \\
    &=\left|\int_{\R^{2d}} \nabla_x W(x+hv-y-hw)-\nabla_x W(x-y)\,d\mu(y,w)\right|
    \leq Mh\int_{\R^{2d}}|v-w|\,d\mu(y,w),
\end{align*}
and thus
\begin{equation}\label{eq:Wgrad_vhih_error}
\begin{split}
    \|(\nabla_x W\ast \Pi^x((\Phi_h)_\#(\mu)))\circ\Phi_h- \nabla_x W\ast \Pi^x\mu\|_{L^2(\mu)}^2
    \leq M^2 h^2 \int_{\R^{2d}} \int_{\R^{2d}} |v-w|^2 \,d\mu(y,w)\,d\mu(x,v)   \\
    \leq 2 M^2 h^2 \int_{\R^{2d}} \int_{\R^{2d}} |v|^2 + |w|^2\,d\mu(y,w)\,d\mu(x,v)
    \leq 4M^2 h^2 \|\id_v\|_{L^2(\mu)}^2.
\end{split}
\end{equation}
This concludes the proof of \eqref{eq:dvE_xstep;ubd}.
\end{proof}

\section{A minimizing movement scheme for the Vlasov-Fokker-Planck equation}\label{sec:KFP_JKO}

In this section we introduce \emph{the coordinate-wise minimizing movement scheme} for the Vlasov-Fokker-Planck equation and show that the time-discrete variational problem has unique solutions.

We introduce the following notation for the minimizers of the variational problem:
\begin{equation}\label{def:JpJx}
\begin{split}
    J^v_h[\mu]&:=\argmin_{\nu\in\P_2^v(\R^{2d};\Pi^x\mu)} \frac{W_{2,v}^2(\mu,\nu)}{2h}+\Lfun_v(\nu)+\alpha\H(\nu) \text{ given } \mu\in D(\Lfun_v+\alpha\H),\\
    J^x_h[\bar\mu]&:=\argmin_{\nu\in\P_2^x(\R^{2d};\Pi^v\bar\mu)} \frac{W_{2,x}^2(\bar\mu,\nu)}{2h}-\Lfun_x(\nu) \text{ given } \bar\mu\in D(\Lfun_x).
\end{split}
\end{equation}
The coordinate-wise minimizing movement scheme is defined as follows:
\begin{equation}\label{eq:SIE}
    \bar\mu_{(i+1)h}^N\in J^v_h[\mu_{ih}^N],\qquad \mu_{(i+1)h}^N\in J^x_h[\bar\mu_{(i+1)h}^N].
\end{equation}
We refer to the first step as the velocity update, and the second as the position update.

We first comment on the connection of the linear functionals $\Lfun_v$ and $\Lfun_x$ to the Poisson structure of the Wasserstein space.

\begin{remark}[$\Lfun_v-\Lfun_x$ and the Poisson structure of $\P_2(\R^{2d})$]\label{rmk:Lp-Lx=poisson}
    Given functionals $\H,\E$ on $\P_2(\R^{2d})$, the (formal) Poisson bracket $\{\H,\E\}_{\P_2(\R^{2d})}(\mu)$ is defined by
    \begin{equation}\label{def:PoissonBracket}
        \{\H,\E\}_{\P_2(\R^{2d})}(\mu)=\left.\frac{d}{dt}\right|_{t=0}\H(\mu_t) \text{ where } \partial_t\mu_t+\nabla\cdot\left(\mu_t \begin{pmatrix}0 & I_d\\ -I_d & 0 \end{pmatrix} \grad\E(\mu_t)\right)=0 \text{ with } \mu_0=\mu.
    \end{equation}
    Lott \cite{Lott08} and Gangbo, Kim, and Pacini \cite{GanKimPac11} studied the Poisson bracket \eqref{def:PoissonBracket} in the Wasserstein space, which arises naturally from the canonical Poisson bracket $\{\;,\;\}_{\R^{2d}}$ of $\R^{2d}$. Temporarily writing $M_2(\mu):=\frac12 \int_{\R^{2d}} |x|^2+|v|^2\,d\mu(x,v)$ the second moment of $\mu$, we have 
    \[\Lfun_v(\mu)-\Lfun_x(\mu)=\{\H,M_2\}_{\P_2(\R^{2d})}(\mu).\]
    Indeed, by definition \eqref{def:PoissonBracket} and the chain rule \cite[Proposition 10.3.18]{AGS},
    \[\{\U,M_2\}_{\P_2(\R^{2d})}(\mu)=\int_{\R^{2d}} \nabla_x\rho\cdot v - \nabla_v\rho\cdot x\,dxdv=0.\]
    In fact for all internal energy functionals the same argument can be formally extended to deduce $\{\U,\cdot\}_{\P_2(\R^{2d})}\equiv 0$ whereas \eqref{eq:divJU=0} implies $\{\cdot,\U\}_{\P_2(\R^{2d})}\equiv 0$; in a sense internal energy functionals can be seen as \emph{first integrals}, and this has been noticed in a greater generality by \cite{KhesinLee08}.
    
    Thus, noting
    \[\nabla_x V+ \nabla_x W\ast\Pi^x\mu\in\partial_x(\V+\W)(\mu) \text{ and } \id_v\in \partial_v(\V+\W)(\mu),\]
    and appealing again to the chain rule,
    \begin{align*}
        \{\H,M_2\}_{\P_2(\R^{2d})}(\mu)&=\{\V+\W,M_2\}_{\P_2(\R^{2d})}(\mu)\\
        &= \int_{\R^{2d}} (\nabla_x V+ \nabla_x W\ast\Pi^x\mu)(x)\cdot v - v\cdot x \,d\mu(x,v) = \Lfun_v(\mu)-\Lfun_x(\mu). 
    \end{align*}
    Thus $\Lfun_v$ and $-\Lfun_x$ arises naturally from the Poisson structure of the Wasserstein space and the energy functional $\H$.

    This also provides a direct connection with the variational reformulation \eqref{def:SIE+damp;Rd;sep;var} of the symplectic Euler algorithm in the Euclidean space with dissipation. Writing $H=V+|\id_v|^2/2$, note that the linear functionals appearing in \eqref{def:SIE+damp;Rd;sep;var} are exactly $\{H,|\id|^2/2\}_{\R^{2d}}=\{V+|\id_v|^2/2,|\id|^2/2\}_{\R^{2d}}$, whereas the the variational step in $v$-variable dissipates $H$. This is in exact parallel for the minimizing movement scheme in the Wasserstein space with fixed marginals \eqref{def:JpJx}.
\end{remark}

We can also reformulate the velocity update of \eqref{eq:SIE} in terms of the conditional entropy.
\begin{remark}[Alternative characterization of the velocity update]\label{rmk:Jp_alt}
    Let $\mu\in\P_2(\R^{2d})$. Then the entropy functional $\U$ satisfies the so-called chain rule
    \begin{equation}\label{eq:entropy;chainrule}
        \U(\mu)=\int_{\R^d} \log\frac{d\Pi^x\mu}{d\Leb^d}(x)\,d\Pi^x\mu(x) + \int_{\R^d}\left(\int_{\R^d}\log\frac{d\mu^x}{d\Leb^d}(v)\,d\mu^x(v)\right)\,d\Pi^x\mu(x).
    \end{equation}
    If $\mu_{h,v}\in J_h^v[\mu]$ then $\Pi^x\mu_{h,v}=\Pi^x\mu$, thus the first term on the right-hand side is not affected. Further noting $\V(\mu_{h,v})=\V(\mu_{h,v})$ and $\W(\mu_{h,v})=\W(\mu)$ we may define
    \begin{equation}\label{def:UpHp}
    \begin{split}
        \U_v(\mu)= \int_{\R^d}\left(\int_{\R^d}\log\frac{d\mu^x}{d\Leb^d}(v)\,d\mu^x(v)\right)\,d\Pi^x\mu(x),
        \H_v(\mu)= \int_{\R^{2d}} \frac{|v|^2}{2}\,d\mu(x,v) + \U_v(\mu).
    \end{split}
    \end{equation}
    and rewrite $J_h^v$ as
    \begin{equation}\label{eq:Jp_alt}
        J_h^v[\mu]=\argmin_{\nu\in\P_2^v(\R^{2d};\Pi^x\mu)}\frac{W_{2,v}^2(\mu,\nu)}{2h}+\Lfun_v(\nu)+\H_v(\nu).
    \end{equation}
    Note that this allows us to consider $J_h^v[\mu]$ even when $\Pi^x\mu$ is not absolutely continuous, as long as $\H_v(\mu)<+\infty$. Considering disintegration with respect to $\Pi^x\mu$, one can readily verify that 
    this is equivalent to taking the steepest descent for each conditional density for fixed $x\in\R^d$, as done in Carlen and Gangbo \cite{CarGan04} for a different nonlinear kinetic Fokker-Planck equation.

    As the conditional entropy $\U_v(\mu)$ is simply the entropy of $\mu^x(dv)$ weighted by $\Pi^x\mu$, it is convex along $W_{2,v}$-geodesics and thus $\Lfun_v+\alpha H$ is $\alpha$-geodesically convex in the $v$-variable. Hence we conclude from Proposition~\ref{prop:pslope_est}
    \begin{equation}\label{eq:Jp_alt;est}
        (1+\alpha h)|\partial_v(\Lfun_v+\alpha\H_v)|^2(\mu_{h,v})\leq (1+\alpha h)\frac{W_{2,v}^2(\mu_{h,v},\mu)}{h^2}\leq \frac{1}{1+\alpha h}|\partial_v(\Lfun_v+\alpha\H_v)|^2(\mu).
    \end{equation}

    However, as $\H$ is the Lyapunov functional for the Vlasov-Fokker-Planck dynamics, the formulation \eqref{def:JpJx} can be seen as more natural.
\end{remark}

Now we turn to the unique solvability of the discrete algorithm. We first establish the lower semicontinuity of $\Lfun_v+\alpha\H$ in its sublevel sets. 

\begin{lemma}\label{lem:Lp+aH;LSC}
    Let $V,W\in C^1(\R^d)$ and let $\mu\in D(|\partial_v\Lfun_v|)$ with $\nabla_x W\ast\Pi^x\mu\in L^1(\mu)$. Then $\Lfun_v+\alpha\H$ is l.s.c. w.r.t. the narrow convergence in each sublevel set -- i.e. for any $c\in\R$, $\sigma\in\P_2(\R^d)$ and a sequence $(\mu_n)_{n\in\N}$ and $\mu_0$ in $\{\nu\in\P_2^v(\R^{2d};\sigma):\;(\Lfun_v+\alpha\H)(\nu)\leq c\}$ such that $\mu_n\rightharpoonup \mu_0$ narrowly, 
    \begin{equation}\label{eq:Lp+aH;LSC}
    (\Lfun_v+\alpha\H)(\mu_0)\leq \liminf_{n\rightarrow\infty} (\Lfun_v+\alpha\H)(\mu_n).
    \end{equation}
\end{lemma}
\begin{proof}
    Let $(\mu_n)_{n\in\N},\mu_0$ be as in the lemma. As $\Pi^x\mu_n=\Pi^x\mu=\sigma$, write
    \[C_\sigma:=\int_{\R^{d}}(V+W\ast\Pi^x\mu_n)\,d\Pi^x\mu_n=\int_{\R^{d}}(V+W\ast\Pi^x\mu)\,d\Pi^x\mu\]
    By Young's product inequality,
    \begin{align*}
        (\Lfun_v+\alpha\H)(\mu_n)-\alpha C_\sigma
        =\langle \nabla_x V + \nabla_x W \ast \Pi^x\mu,\id_v\rangle_{L^2(\mu_n)} + \frac{\alpha}{2}\int_{\R^{2d}}|\id_v|^2\,d\mu_n + \alpha\U(\mu_n)    \\
        \geq -\frac{1}{\alpha}\|\nabla_x V + \nabla_x W \ast \Pi^x\mu\|_{L^2(\mu)}^2 + \frac{\alpha}{4}\int_{\R^{2d}}|\id_v|^2\,d\mu_n + \alpha\U(\mu_n).
    \end{align*}
    As $\mu\in D(|\partial_v\Lfun_v|)$, Proposition~\ref{prop:dom_dvLp} implies $\|\nabla_x V + \nabla_x W \ast \Pi^x\mu\|_{L^2(\mu)}^2<+\infty$. On the other hand, we can fix $\beta\in(0,1)$ and $C>0$ such that
    \[\U(\mu_n)\geq - C\left(\int_{\R^{2d}} |x|^2+|v|^2\,d\mu_n(x,v) +1\right)^{\beta} \]
    -- see for instance proof of \cite[Proposition 4.1]{JKO98}. This implies
    \[(\Lfun_v+\alpha\H)(\mu_n)\rightarrow +\infty \text{ as } \int_{\R^{2d}} |v|^2\,d\mu_n(x,v)\rightarrow\infty.\]
    Thus, on each sublevel set of $(\Lfun_v+\alpha\H)$ in $\P_2^v(\R^{2d};\Pi^x\mu)$ the second $v$-moments of $\mu_n$ must remain bounded.

    As $V,W\in C^1$, the map $(x,v)\mapsto (\nabla_x V(x) + \nabla_x W\ast\Pi^x\mu(x))\cdot v$ is continuous and
    \[\sup_{n\in\N} \int_{\R^{2d}} |\nabla_x V + \nabla_x W \ast \Pi^x\mu||v|\,d\mu_n\leq  \frac12 \|\nabla_x V + \nabla_x W \ast \Pi^x\mu\|_{L^2(\mu)}^2 + \frac12 \sup_{n\in\N}\int_{\R^{2d}} |v|^2\,d\mu(x,v)<+\infty.\]
    Hence by \cite[Lemma 5.1.7]{AGS} we have $\lim_{n\rightarrow\infty}\Lfun_v(\mu_n)=\Lfun_v(\mu)$. On the other hand, it is well-known that 
    \[\nu\mapsto \frac12\int_{\R^{2d}} |v|^2\,d\nu + \U(\nu)\]
    is l.s.c. w.r.t. narrow convergence, thus we conclude \eqref{eq:Lp+aH;LSC}.
\end{proof}

Existence and uniqueness of minimizers of the variational problem \eqref{def:JpJx} is then a consequence of the abstract result Proposition~\ref{prop:pJKO_minimizer}.

\begin{proposition}[Existence a unique minimizers of the variational problem]\label{prop:LxJKO;Phih}
Let $V,W\in C^1(\R^d)$ and $\mu\in D(\Lfun_v+\alpha\H)\cap\D(|\partial_v\Lfun_v|)$ such that $\nabla_x W\ast\Pi^x\mu\in L^1(\mu)$. Then for any $h>0$, $J_v^h(\mu)$ defined in \eqref{def:JpJx} contains exactly one element.

Moreover, at any $\mu\in\P_2(\R^{2d})$ $J_x^h[\mu]$ is also a singleton, and the minimizer $\mu_{h,x}\in J_x^h[\mu]$ satisfies
\begin{equation}\label{eq:LxJKO;Phih}
    \mu_{h,x} = (\Phi_h)_\#\mu \text{ where } \Phi_h(x,v)=\begin{pmatrix} x+hv\\ v\end{pmatrix}.
\end{equation}

\end{proposition}
\begin{proof}
Let $\mu\in D(\Lfun_v+\alpha \H)\cap\D(|\partial_v\Lfun_v|)$. Note $\Lfun_v+\alpha\H$ is $\alpha$-convex along $W_{2,v}$-geodesics. By Lemma~\ref{lem:Lp+aH;LSC}, $\Lfun_v+\alpha\H$ is l.s.c. w.r.t. narrow convergence in sublevel sets, hence by Proposition~\ref{prop:pJKO_minimizer} it follows that $J_h^v[\mu]$ is a singleton.

Now we show directly that $\mu_{h,x}=(\Phi_h)_\#\mu$ is the unique element of $J_h^x[\mu]$. Let $\nu\in\P_2^x(\R^{2d};\Pi^v\mu)$. Then one can check that the optimality condition
\begin{equation}\label{eq:Phihmu;LxJKOmin}
    \frac{W_{2,x}^2(\mu,(\Phi_h)_\#\mu)}{2h}-\Lfun_x((\Phi_h)_{\#}\mu)
    \leq \frac{W_{2,x}^2(\mu,\nu)}{2h}-\Lfun_x(\nu)
\end{equation}
can be reorganized and written in terms of $\gamma=\{\gamma^v\}_{v\in\R^d}\in\Gamma_o^v(\mu,\nu)$ as
\[\frac{h}{2}\|\id_v\|_{L^2(\mu)}+\frac{1}{2h}\int_{\R^d}\iint_{\R^d\times\R^d} |y-x|^2\,d\gamma^v(x,y)\,d\Pi^v\mu(v)
-\int_{\R^d}\iint_{\R^d\times\R^d} v\cdot(y-x)\,d\gamma^v(x,y)\,d\Pi^v\mu(v)\geq 0,\]
which holds by Young's inequality. 
Furthermore, the equality holds if and only if $hv=y-x$ for $\gamma$-a.e. $(x,y,v)\in(\R^{d})^3$, from which we conclude that $(\Phi_h)_\#\mu$ is the unique minimizer.

\end{proof}

We conclude this section by establishing that the iterative algorithm has a unique solution.
\begin{theorem}[Unique solvability of the iterative algorithm]\label{thm:iter_solv}
    Let $\nabla_x V,\nabla_x W$ be Lipschitz and $\mu_0\in D(\U)$. Then for any $h>0$ and $N\in\N$, there is a unique solution $(\mu_{ih}^N)_{i=0}^N$ with $\mu_0^N=\mu_0$ to the recursive variational scheme \eqref{eq:SIE}
    \[\bar\mu_{(i+1)h}^N \in J_h^v[\mu_{ih}^N],\qquad \mu_{(i+1)h}^N\in J_h^x[\bar\mu_{(i+1)h}^N].\]
\end{theorem}

\begin{proof}
It is easy to see that $D(\Lfun_v+\alpha\H)=D(\U)$, as Lipschitz continuity of $\nabla_x V,\nabla_x W$ imply that $\Lfun_v(\mu)+\V(\mu)<+\infty$ for any $\mu\in\P_2(\R^{2d})$. Moreover, from Remark~\ref{rmk:W_assump} we know $D(|\partial_v\Lfun_v|)=\P_2(\R^{2d})$. Thus
\[D(\U)=D(\Lfun_v+\alpha\H)\cap D(|\partial_v\Lfun_v|).\]

Hence whenever $\mu_0\in D(\U)$, $\bar\mu_{h}^N,\mu_{h}^N$, by Proposition~\ref{prop:LxJKO;Phih}
$\bar\mu_{h}\in J_h^v[\mu_0]$ and $\mu_h^N \in J_h^x[\bar\mu_{h}^N]$ are uniquely defined. As 
\[\Lfun_v(\bar\mu_{h}^N)+\alpha \H(\bar\mu_{h}^N)\leq \Lfun_v(\mu_0)+\alpha \H(\mu_0)<+\infty,\]
clearly $\bar\mu_{h}^N\in D(\Lfun_v+\alpha\H)=D(\U)$ and thus absolutely continuous w.r.t. $\Leb^{2d}$. 

On the other hand, from \eqref{eq:LxJKO;Phih} $\mu_h=(\Phi_h)_\#\bar\mu_{h}^N$. Clearly $\mu_h\ll\Leb^{2d}$, and denoting by $\rho_h^N,\bar\rho_{h}^N$ the $\Leb^{2d}$-densities of $\mu_h^N,\bar\mu_{h}^N$ respectively, we have \eqref{eq:rho_vhih}
\[\rho_h^N(\Phi_h(x,v))= \frac{\bar\rho_{h}^N(x,v)}{|\det\nabla\Phi_h(x,v)|}=\bar\rho_{h}^N(x,v)\]
Thus
\[\U(\mu_h^N)=\U(\bar\mu_{h}^N)<+\infty,\]
and $\mu_h^N\in D(\U)=D(\Lfun_v+\alpha\H)\cap D(|\partial_v\Lfun_v|)$. Thus we can repeat the same process with $\mu_h^N$ in place of $\mu_0$ and so on to conclude.

\end{proof}

\section{Existence of minimizing movements: compactness of the discrete solutions}\label{sec:SIE;cpct}
This section establishes the existence of minimizing movements. More precisely, we show that the piece-wise constant interpolation between discrete solutions of the coordinate-wise minimizing movement scheme \eqref{eq:SIE} with vanishing time steps converge pointwise narrowly over a subsequence. Using the Arzel\`a-Ascoli-type theorem (Proposition~\ref{prop:ArzelaAscoli-weak}), we establish the compactness result Theorem~\ref{thm:disc_sol;cpct2} given that the Hamiltonian energy functional $\H$ \eqref{def:H=V+U+W} is finite at the initial datum. We will impose the additional condition
\begin{equation}\label{eq:H;bddbelow}
    \inf_{\nu\in\P_2(\R^{2d})} \H(\nu)>-\infty.
\end{equation}
The assumption \eqref{eq:H;bddbelow} can be relaxed in various ways, for instance leveraging the lower bound on the entropy functional in terms of the second moment as in \cite{JKO98} or more generally the coercivity of $\H$ as in \cite[Chapter 3]{AGS}. However, \eqref{eq:H;bddbelow} significantly simplifies the proof and is satisfied whenever $V$ is a confining potential (i.e. $e^{-V}\in L^1(\R^d)$) and the interaction potential $W$ is nonnegative.

At the end of this section we state another compactness result, Theorem~\ref{thm:disc_sol;subseq_conv}, which does not assume \eqref{eq:H;bddbelow} and requires only the finiteness of the $v$-metric slope $|\partial_v(\Lfun_v+\alpha\H)|$ at the initial datum. As such an assumption is not as natural physically, we delay the proof to Appendix~\ref{app:cpct_alt}. However, the assumption is not in general stronger than the assumption of finite $\H$ at the initial datum; see Remark~\ref{rmk:mu0_xsingular} for further details.

We also note that results in this sections can be generalized to dissipative Hamiltonian flows involving internal energy functionals that are convex in the Wasserstein space; see Remark~\ref{rmk:generalU}. 
\medskip

We begin by providing useful preliminary estimates, namely growth bounds on metric slopes over the discrete solutions.
\begin{lemma}[Growth bounds on slopes]\label{lem:dWVip_growthbd}
    Let $\nabla_x V,\nabla_x W$ be $M$-Lipchitz. Then the discrete solutions $(\mu_{ih}^N)_{i=0}^N$ of \eqref{def:SIE} with time step $h\leq \frac12$ and initial datum $\mu_0^N\in D(\U)$ satisfy
    \begin{equation}\label{eq:Wgrowthbd}
        |\partial_x\W|(\mu_{(i+1)h}^N)\leq |\partial_x\W|(\mu_{ih}^N) +  2Mh\|\id_v\|_{L^2(\mu_{(i+1)h}^N)}
    \end{equation}
    \begin{equation}\label{eq:Vgrowthbd}
        |\partial_x\V|(\mu_{(i+1)h}^N)\leq |\partial_x\W|(\mu_{ih}^N) + Mh \|\id_v\|_{L^2(\mu_{(i+1)h}^N)}
    \end{equation}
    and
    \begin{equation}\label{eq:pmoment_growthbd}
        \|\id_v\|_{L^2(\mu_{(i+1)h}^N)}^2 \leq \frac{1}{(1+2\alpha h)} \|\id_v\|_{L^2(\mu_{ih}^N)}^2 -2h\Lfun_v(\bar\mu_{(i+1)h}^N) + 2\alpha dh
    \end{equation}
    In particular, if furthermore $h\leq \frac{1}{2M}$ then letting $E(\nu):=|\partial_x\W|^2(\nu)+|\partial_x\V|^2(\nu)+\|\id_v\|_{L^2(\nu)}^2$, we have
    \begin{equation}\label{eq:WVpmoment_growthbd}
        E(\mu_{(i+1)h}^N)\leq (1+(4+14M)h)E(\mu_{ih}^N) + 16\alpha dh.
    \end{equation}
\end{lemma}
\begin{proof}
By Theorem~\ref{thm:iter_solv}, $(\mu_{ih}^N)_{i=0}^N$ is well-defined. First we consider $|\partial_x\V|$. As $\Lip \nabla_x V\leq M$,
\begin{align*}
    |\partial_x\V|(\mu_{(i+1)h}^N)&-|\partial_x\V|(\mu_{ih}^N)
    =|\partial_x\V|(\mu_{(i+1)h}^N)-|\partial_x\V|(\bar\mu_{(i+1)h}^N)    \\
    &\leq \left(\int_{\R^{2d}} |\nabla_x V(x+hv)-\nabla_x V(x)|^2\,d\bar\mu_{(i+1)h}^N(x,v) \right)^{1/2}
    \leq M h \|\id_v\|_{L^2(\bar\mu_{(i+1)h}^N)}.
\end{align*}

Next we consider $|\partial_x\W|$. Using $\Lip\nabla_x W\leq M$, we know from \eqref{eq:Wgrad_vhih_error} that
\begin{align*}
    &\|\nabla_x W\ast\Pi^x \mu_{(i+1)h}^N\|_{L^2(\mu_{(i+1)h})}-\|\nabla_x W\ast\Pi^x \bar\mu_{(i+1)h}^N\|_{L^2(\bar\mu_{(i+1)h}^N)} \\
    &\leq \|(\nabla_x W\ast \Pi^x((\Phi_h)_\#(\bar\mu_{(i+1)h}^N)))\circ\Phi_h- \nabla_x W\ast \Pi^x\bar\mu_{(i+1)h}^N\|_{L^2(\bar\mu_{(i+1)h}^N)}\leq 2Mh\|\id_v\|_{L^2(\bar\mu_{(i+1)h}^N)}.
\end{align*}

On the other hand, by the $2$-convexity of $\nu\mapsto\|\id_v\|_{L^2(\nu)}^2$ along $W_{2,v}$-geodesics, Lemma~\ref{lem:euler-lagrange;Jp}, and \eqref{eq:xidvlogrho=0}, we have
\begin{align*}
    \|\id_v\|_{L^2(\mu_{ih}^N)}^2&\geq \|\id_v\|_{L^2(\bar\mu_{(i+1)h}^N)}^2+2h\langle \id_v, \nabla_x V+\nabla_x W\ast\Pi^x\bar\mu_{(i+1)h}^N+\alpha\id_v+\alpha\nabla_v\log\bar\rho_{(i+1)h}^N\rangle_{L^2(\bar\mu_{(i+1)h}^N)} \\
    &\geq (1+2\alpha h)\|\id_v\|_{L^2(\bar\mu_{(i+1)h}^N)}^2+ 2h\Lfun_v(\bar\mu_{(i+1)h}^N) - 2\alpha dh.
\end{align*}
Noting $\|\id_v\|_{L^2(\bar\mu_{(i+1)h}^N)}=\|\id_v\|_{L^2(\mu_{ih}^N)}$, we deduce
\begin{equation}\label{eq:pmoment_bd1}
        \|\id_v\|_{L^2(\mu_{(i+1)h}^N)}^2 \leq \frac{1}{1+2\alpha h} \|\id_v\|_{L^2(\mu_{ih}^N)}^2 -2h\Lfun_v(\bar\mu_{(i+1)h}^N) + 2\alpha dh.
\end{equation}
As $-2h\Lfun_v(\bar\mu_{(i+1)h}^N)\leq h\|\id_v\|_{L^2(\mu_{(i+1)h}^N)}^2+ h|\partial_x(\V+\W)|^2(\mu_{ih}^N)$, we can use a crude bound to deduce
\begin{align*}
    (1-h)\|\id_v\|_{L^2(\mu_{(i+1)h}^N)}^2 \leq \|\id_v\|_{L^2(\mu_{ih}^N)}^2+h|\partial_x(\V+\W)|^2(\mu_{ih}^N)+2\alpha d h.
\end{align*}
Using that $\frac{1}{1-h}\leq 1+2h$ for $h\leq \frac12$ and that $h|\partial_x(\V+\W)|^2(\mu_{ih}^N)\leq 2h(|\partial_x \V|^2(\mu_{ih}^N)+|\partial_x \W|^2(\mu_{ih}^N))$, we have
\begin{equation}\label{eq:pmomentsqr;crudebd}
    \|\id_v\|_{L^2(\mu_{(i+1)h}^N)}^2\leq (1+2h) \|\id_v\|_{L^2(\mu_{ih}^N)}^2 + 4h(|\partial_x \V|^2(\mu_{ih}^N)+|\partial_x \W|^2(\mu_{ih}^N)) + 4\alpha dh.
\end{equation}

Now suppose $h\leq \frac{1}{2M}$. Squaring and adding both sides of \eqref{eq:Wgrowthbd}, \eqref{eq:Vgrowthbd} and using $4M^2h^2\leq 2 Mh$
\begin{align*}
    |\partial_x\W|^2(\mu_{(i+1)h}^N)&+|\partial_x\V|^2(\mu_{(i+1)h}^N)\\
    &\leq (1+2Mh)|\partial_x\W|^2(\mu_{ih}^N)+(1+Mh)|\partial_x\V|^2(\mu_{ih}^N)+6Mh \|\id_v\|_{L^2(\mu_{(i+1)h}^N)}^2\\
    &\leq (1+2Mh)[|\partial_x\W|^2(\mu_{ih}^N)+|\partial_x\V|^2(\mu_{ih}^N)] + 6Mh \|\id_v\|_{L^2(\mu_{(i+1)h}^N)}^2\\
    &\leq (1+14Mh)[|\partial_x\W|^2(\mu_{ih}^N)+|\partial_x\V|^2(\mu_{ih}^N)] + 6Mh(1+2h)\|\id_v\|_{L^2(\mu_{ih}^N)}^2 + 12\alpha dh.
\end{align*}
where we have used \eqref{eq:pmomentsqr;crudebd} in the last line. Finally, we add \eqref{eq:pmomentsqr;crudebd} to both sides to obtain
\begin{align*}
    &|\partial_x\W|^2(\mu_{(i+1)h}^N)+|\partial_x\V|^2(\mu_{(i+1)h}^N)+\|\id_v\|_{L^2(\mu_{(i+1)h}^N)}^2 \\
    &\leq (1+4h+14Mh)[|\partial_x\W|^2(\mu_{ih}^N)+|\partial_x\V|^2(\mu_{ih}^N)]+(1+2h)\|\id_v\|_{L^2(\mu_{ih}^N)}^2
    + 16\alpha dh.
\end{align*}
Thus, letting $E(\nu):=|\partial_x\W|^2(\nu)+|\partial_x\V|^2(\nu)+\|\id_v\|_{L^2(\nu)}^2$ we have
\[E(\mu_{(i+1)h}^N)\leq (1+(4+14M)h)E(\mu_{ih}^N) + 16\alpha dh.\]
    
\end{proof}

Next we establish a key energy estimate.
\begin{lemma}\label{lem:energy_decay;disc}
    Suppose $\nabla_x V,\nabla_x W$ are $M$-Lipschitz, and let $(\mu_{ih}^N)_{i=0}^N$ be defined by \eqref{eq:SIE} with time step $h>0$, number of iterations $n\in\N$ and initial datum $\mu_0^N\in D(\H)$. Then
\begin{equation}\label{eq:energy_decay;ubd}
\begin{split}
     &\left(1+\frac{\alpha h}{2}\right)\frac{W_{2,v}^2(\mu_{ih}^N,\bar\mu_{(i+1)h}^N)}{h} 
    \leq \alpha(\H(\mu_{ih}^N)-\H(\mu_{(i+1)h}^N))\\
    &\qquad+h\|\nabla_x V+\nabla_x W\ast\Pi^x\bar\mu_{(i+1)h}^N\|_{L^2(\bar\mu_{(i+1)h}^N)}^2+2\alpha h\Lfun_v(\bar\mu_{(i+1)h}^N)+M\alpha h^2\|\id_v\|_{L^2(\bar\mu_{(i+1)h}^N)}^2.
\end{split}
\end{equation}
\end{lemma}

\begin{proof}
As $\nabla_x V,\nabla_x W$ are $M$-Lipschitz, $\nabla_x W\ast \Pi^x\nu$ is also $M$-Lipschitz for any $\nu\in\P_2(\R^{2d})$. Furthermore, $\mu\mapsto\U(\mu)$ is invariant under the position update. Thus
    \begin{align*}
    \H(\mu_{(i+1)h}^N)-\H(\bar\mu_{(i+1)h}^N)&=(\V+\W)(\mu_{(i+1)h}^N)-(\V+\W)(\bar\mu_{(i+1)h}^N)
    \\
    &\leq h\Lfun_v(\bar\mu_{(i+1)h}^N)+Mh^2\|\id_v\|_{L^2(\bar\mu_{(i+1)h}^N)}.
    \end{align*}
Moreover, as $\Lfun_v+\alpha\H$ is $\alpha$-convex along $W_{2,v}$-godesics, the slope estimates \eqref{eq:partial_slope_est} imply
\begin{align*}
    \alpha \H(\bar\mu_{(i+1)h}^N) + \Lfun_v(\bar\mu_{(i+1)h}^N)+\left(1+\frac{\alpha h}{2}\right)\frac{W_{2,v}^2(\mu_{ih}^N,\bar\mu_{(i+1)h}^N)}{h}\leq \alpha\H(\mu_{ih}^N)+\Lfun_v(\mu_{ih}^N).
\end{align*}
By the Euler-Lagrange equation \eqref{eq:euler-lagrange;Jp} and \eqref{eq:xidvlogrho=0} we have
\begin{equation}\label{eq:Lp_diff;eq}
\Lfun_v(\mu_{ih}^N-\bar\mu_{(i+1)h}^N)=h\|\nabla_x V+\nabla_x W\ast\Pi^x\bar\mu_{(i+1)h}^N\|_{L^2(\bar\mu_{(i+1)h}^N)}^2+\alpha h\Lfun_v(\bar\mu_{(i+1)h}^N).
\end{equation}
Thus we can combine the estimates to obtain
\begin{align*}
    &\left(1+\frac{\alpha h}{2}\right)\frac{W_{2,v}^2(\mu_{ih}^N,\bar\mu_{(i+1)h}^N)}{h}
    \leq \alpha(\H(\mu_{ih}^N)-\H(\bar\mu_{(i+1)h}^N))+\Lfun_v(\mu_{ih}^N)-\Lfun_v(\bar\mu_{(i+1)h}^N)  \\
    &\leq \alpha(\H(\mu_{ih}^N)-\H(\mu_{(i+1)h}^N))+\Lfun_v(\mu_{ih}^N-\bar\mu_{(i+1)h}^N)
    + \alpha h\Lfun_v(\bar\mu_{(i+1)h}^N)+ \alpha M h^2\|\id_v\|_{L^2(\bar\mu_{(i+1)h}^N)}^2 \\
    &= \alpha(\H(\mu_{ih}^N)-\H(\mu_{(i+1)h}^N))+h\|\nabla_x V+W\ast\Pi^x\bar\mu_{(i+1)h}^N\|_{L^2(\bar\mu_{(i+1)h}^N)}^2\\
    &\qquad+2\alpha h\Lfun_v(\bar\mu_{(i+1)h}^N)+M\alpha h^2\|\id_v\|_{L^2(\bar\mu_{(i+1)h}^N)}^2,
\end{align*}
which is precisely \eqref{eq:energy_decay;ubd}. 
\end{proof}

Now we are ready to establish the existence of minimizing movements starting from initial data in $D(\H)$ when $\H$ is bounded from below.
\begin{theorem}[Compactness of discrete solutions with finite initial Hamiltonian]\label{thm:disc_sol;cpct2}
   Let $\nabla_x V,\nabla_x W$ be Lipschitz, and suppose $\inf_{\P_2(\R^{2d})} \H>-\infty$. Fix $T>0$, and for each $N\in\N$ let $h_N=T/N$. Define the constant interpolation $(\mu_t^N)_{t\in [0,T]}$ between the discrete solutions $(\mu_{ih_N}^N)_{i=1}^N$ of \eqref{eq:SIE} by
    \[
        \mu_t^N=\mu_{ih_N}^N \text{ for } t\in[ih_N,(i+1)h_N)
    \]
    where the initial data $\mu_0^N$ converges narrowly to $\mu_0\in\P_2(\R^{2d})\cap D(\H)$ as $N\nearrow\infty$.
    
    Then there exists $(\mu_t)_{t\in[0,T]}\in AC^2([0,T];\P_2(\R^{2d}))$ such that over a subsequence
    \begin{equation}\label{eq:ptws_narrow convergence}
        \mu_t^N\rightharpoonup \mu_t \text{ narrowly for all } t\in[0,T].
    \end{equation}
\end{theorem}

\begin{proof}
Let $\Lip \nabla_x V,\Lip \nabla_x W\leq M$. As
\[\frac{1}{h}W_{2,x}^2(\bar\mu_{(i+1)h}^N,\mu_{(i+1)h}^N)=h\|\id_v\|_{L^2(\bar\mu_{(i+1)h}^N)}^2,\]
adding $\frac{1}{h}W_{2,x}^2(\bar\mu_{(i+1)h}^N,\mu_{(i+1)h}^N)$ to \eqref{eq:energy_decay;ubd} and summing over $i=0,\cdots,N-1$ we obtain
\begin{align*}
    \sum_{i=0}^{N-1}\frac{W_2^2(\mu_{ih}^N,\mu_{(i+1)h}^N)}{2h}&\leq \sum_{i=0}^{N-1}\frac{W_{2,v}^2(\mu_{ih}^N,\bar\mu_{(i+1)h}^N)}{h}+\frac{W_{2,x}^2(\bar\mu_{(i+1)h}^N,\mu_{(i+1)h}^N)}{h} \\
    &\leq \alpha(\H(\mu_0)-\H(\mu_{T}^N))
    +\sum_{i=0}^{N-1} h\|\nabla_x V+W\ast\Pi^x\bar\mu_{(i+1)h}^N\|_{L^2(\bar\mu_{(i+1)h}^N)}^2\\
    &+2\alpha h\Lfun_v(\bar\mu_{(i+1)h}^N)+h(1+M\alpha h)\|\id_v\|_{L^2(\bar\mu_{(i+1)h}^N)}^2.
\end{align*}
As $\inf_{\nu\in\P_2(\R^{2d})}\H(\nu)>-\infty$, we have $\H(\mu_0)-\H(\mu_T^N)\leq  C$ for some constant $C$ depending only on $\H$ and $\mu_0$. Additionally note that remaining terms on the right-hand side is uniformly bounded in $N$ by the growth bounds \eqref{eq:WVpmoment_growthbd} and discrete Gronwall's inequality.

Now we proceed as in the proof of \cite[Corollary 3.3.4]{AGS}. Let $m_N\in L^2([0,T])$ be defined by
\begin{equation}\label{def:mN}
    m_N(t)=\frac{W_2(\mu_{ih_N}^N,\mu_{(i+1)h_N}^N)}{h_N} \text{ for } t\in[ih_N,(i+1)h_N).
\end{equation}
In other words, $m_N$ is the metric derivative of the piecewise-geodesic interpolation constructed from the discrete solutions $(\mu_{ih_N}^N)_{i=0}^N$. Then by the previous estimate we have that $\|m_N\|_{L^2([0,T])}$ is bounded uniformly in $N$, thus we can find a weak subsequential limit $m_{\infty}\in L^2([0,T])$.

Letting 
\begin{equation}\label{def:i_vm}
    i_-(N,r)=\lfloor Nr/T-1\rfloor,i_+(N,r)=\lceil Nr/T-1\rceil,
\end{equation}
we have
\begin{equation}\label{eq:W2muN_distbd}
    W_2(\mu_s^N,\mu_t^N)\leq \int_{i_-(N,s)}^{i_+(N,r)} m_N(r)\,dr.
\end{equation}
In particular, by H\"older's inequality
\[W_2(\mu_0,\mu_{T}^N) \leq \int_0^T m_N(r)\,dr \leq \sqrt{T} \|m_N\|_{L^2([0,T])}\]
is uniformly bounded, thus $\mu_t^N$ remain in sufficiently large Wasserstein ball around $\mu_0$ which is relatively compact with respect to the narrow topology.

Further, to see equicontinuity, note that for any $0<s<t<T$
\begin{align*}
    \limsup_{N\rightarrow\infty} W_2(\mu_s^N,\mu_t^N)\leq \limsup_{N\rightarrow\infty}\int_{i_-(N,s)}^{i_+(N,r)} m_N(r)\,dr \leq \int_s^t m_\infty(r)\,dr.
\end{align*}
Thus we apply the weak Arzel\`a-Ascoli theorem (Proposition~\ref{prop:ArzelaAscoli-weak}) to obtain a limiting curve $(\mu_t)_{t\in[0,T]}$ satisfying \eqref{eq:ptws_narrow convergence}. 

Finally $(\mu_t)_{t\in[0,T]}\in AC^2([0,T];\P_2(\R^{2d}))$ follows from \eqref{eq:W2muN_distbd}, weak convergence of $m_N$ to $m_\infty$ in $L^2([0,T])$, and the lower semicontnuity of $W_2$ with respect to the narrow convergence.
\end{proof} 

We state an alternative compactness results that only requires that the metric slope in $v$ is finite at the initial datum, while delaying the proof to Appendix~\ref{app:cpct_alt}. As we will see in Remark~\ref{rmk:mu0_xsingular}, this assumption allows for singularities in the $x$-variable unlike that of Theorem~\ref{thm:disc_sol;cpct2}.
\begin{theorem}[Compactness of discrete solutions with finite metric slope in $v$]\label{thm:disc_sol;subseq_conv}
    Let $\nabla_x V,\nabla_x W$ be Lipschitz. Fix $T>0$, and for each $N\in\N$ let $h_N=T/N$.
    Define the constant interpolation $(\mu_t^N)_{t\in [0,T]}$ between discrete solutions $(\mu_{ih_N}^N)_{i=1}^N$ of \eqref{eq:SIE} by
    \begin{equation}\label{def:const_interpol}
        \mu_t^N=\mu_{ih_N}^N \text{ for } t\in[ih_N,(i+1)h_N).
    \end{equation}
    where the initial data $\mu_0^N$ satisfy
    \begin{equation}\label{ass:mu0N}
        \sup_{N\in\N} |\partial_v(\Lfun_v+\alpha\H)|(\mu_0^N),\quad \sup_{N,\tilde N\in\N}W_{2,v}(\mu_0^N,\mu_0^{\tilde N})<+\infty.
    \end{equation}
        
    Then there exists a Lipschitz curve $(\mu_t)_{t\in[0,T]}\in AC([0,T];\P_2(\R^{2d}))$ such that over a subsequence
    \[
        \mu_t^N\rightharpoonup \mu_t \text{ narrowly for all } t\in[0,T].
    \]
\end{theorem}
\begin{remark}[Initial data singular in the $x$-variable]\label{rmk:mu0_xsingular}
    Let $\mu_0=\sigma\otimes\upsilon$ with $\sigma,\upsilon\in\P_2(\R^d)$. Suppose further $\upsilon\ll\Leb^d$ with density $g$ that has finite Fisher information.
    Let $\eta:\R^d\rightarrow\R$ be any standard convolution kernel and $\eta_\eps(x)=\eps^{-d}\eta(x/\eps)$. Then $\mu_0^N:=(\sigma\ast\eta_{1/N})\otimes\upsilon$ has $\Leb^{2d}$ density $\rho_0^N(x,v)=f_\eps(x)g(v)$ for some $\int_{\R^d} f_\eps(x)\,dx=1$. Thus
\begin{align*}
    &|\partial_v\U|^2(\mu_0^N)=\int_{\R^{2d}} \left|\frac{\nabla_v \rho_0^N}{\rho_0^N}(x,v)\right|^2\,d\mu_0^N(x,v) 
    =\int_{\R^{2d}} \left|\frac{f_\eps(x)\nabla_v g(v)}{f_\eps(x)g(v)}\right|^2\,f_\eps(x)g(v)\,dxdv\\
    &=\int_{\R^{d}}\left|\frac{\nabla_v g(v)}{g(v)}\right|^2 g(v)dv \int_{\R^d} f_\eps(x)\,dx  
    =\int_{\R^{d}}\left|\frac{\nabla_v g(v)}{g(v)}\right|^2 g(v)dv.
\end{align*}
As $\mu_0^N\rightharpoonup \mu_0$ narrowly as $N\rightarrow\infty$, by lower semicontinuity of $W_2$ (see \cite[Remark 6.12]{Vil09} for instance, where they refer to the narrow convergence as weak convergence) we have $W_2(\mu_0^N,\mu_0)\xrightarrow[]{N\rightarrow \infty} 0$ , thus the condition \eqref{ass:mu0N} is satisfied and $\lim_{t\searrow 0} \mu_t = \sigma\otimes\upsilon$. In light of Theorem~\ref{thm:disc_sol;subseq_conv}, this implies the existence of weak solutions of the Vlasov-Fokker-Planck equation with initial datum of product form with arbitrary $x$-marginal, as long as the $v$-marginal of the initial datum is sufficiently regular.
\end{remark}

We note that analogous compactness results can be deduced when the entropy functional is replaced with other internal energy functionals geodesically convex in the Wasserstein space.
\begin{remark}[Generalization to other internal energies]\label{rmk:generalU}
    Consider $\U:\P_2(\R^{2d})\rightarrow(-\infty,+\infty]$ of the form
    \begin{align*}
        \U(\mu)= \int_{\R^{2d}} U(\rho)\,d\Leb^{2d} \text{ if } \mu=\rho\Leb^{2d} \text{ and } +\infty \text{ otherwise }
    \end{align*}
    where $U$ satisfies the conditions (cf. \cite[Example 9.3.6]{AGS})
    \begin{equation}\label{cond:U;cvxmonotone}
        s\mapsto s^{2d} U(s^{-2d}) \text{ is convex and nondecreasing in } (0,+\infty)
    \end{equation}
    and
    \begin{equation}\label{cond:U;growth}
        U(0)=0,\;\liminf_{s\searrow 0}\frac{U(s)}{s^\beta}>-\infty \text{ for some } \beta>\frac{2d}{2d+2}.
    \end{equation}
    For instance $U(s)=s^\beta/(\beta-1)$ for $\beta>1$ satisfies these conditions, and it is well-known that the corresponding functional is geodesically convex in the Wasserstein space. While we have only established convexity of $\U$ for $U(s)=s\log s$ in Lemma~\ref{lem:pdirder;U} and characterizations of subdifferentials in Theorem~\ref{thm:psubdiff;U} for simplicity, the proofs of these results depend only on the conditions \eqref{cond:U;cvxmonotone} and \eqref{cond:U;growth}, hence analogous results can be derived for more general internal energies with appropriate modifications; see \cite[Chapter 10.4.3]{AGS}.

    Furthermore, internal energy functionals are generally invariant under the pushforward with respect to the map $(x,v)\mapsto (x+hv,v)$, hence the same arguments in this section and Appendix~\ref{app:cpct_alt} generalize to a more general class of internal energies.
\end{remark}

\section{Convergence of the discrete solutions to the weak solution of the PDE}\label{sec:MMS=VFP}   
We now show that minimizing movements obtained as a limiting curve of discrete solutions to \eqref{eq:SIE} are distributional solutions of the Vlasov-Fokker-Planck equation \eqref{eq:VFP}. After obtaining time-discrete PDE satisfied by the discrete variational problems in Lemma~\ref{lem:phi_dmu;discrete_time}, we take the limit as time step vanishes in Proposition~\ref{prop:conv_PDE} to show that minimizing movements are solutions of VFP. Combined with the results from Section~\ref{sec:SIE;cpct}, this allows us to obtain our main result, Theorem~\ref{thm:main}, that our discrete solutions converge to the solution of VFP.

Throughout this section, we will use the big-O notation $\O$ very precisely, to mean that $f\in\O(g)$ if $|f|\leq g$ without any implicit constant.
\medskip

First we use the Euler-Lagrange equations from Lemma~\ref{eq:euler-lagrange;Jp} associated to the velocity update to deduce the time-discrete PDE satisfied by the discrete variational problems.
\begin{lemma}[Time-discrete PDE]\label{lem:phi_dmu;discrete_time}
Let $V,W\in C^1(\R^d)$ and $W$ satisfy \eqref{eq:W_ass}. let $(\mu_{ih})_{i=1}^N$ be the discrete solutions to \eqref{eq:SIE} with $\mu_0^N\in D(\Lfun_v+\alpha\H)$. Then for any $\varphi\in C_c^\infty(\R^{2d})$ and $i=0,\cdots,N-1$
\begin{equation}\label{eq:phi_dmu;discrete_time}
\begin{split}
    \frac1h \int_{\R^{2d}} \varphi\,d(\mu_{(i+1)h}^N-\mu_{ih}^N)
    &= \int_{\R^{2d}} \begin{pmatrix}
            \id_v \\ -\nabla_x V-(\nabla_x W\ast\Pi^x\bar\mu_{(i+1)h}^N) - \alpha \id_v
        \end{pmatrix}\cdot \begin{pmatrix}
            \nabla_x \varphi \\ \nabla_v \varphi
        \end{pmatrix}
    +\alpha\Delta_v \varphi\,d\bar\mu_{(i+1)h}^N \\
    &+  \frac1h W_{2,v}^2(\mu_{ih}^N,\bar\mu_{(i+1)h}^N)\O(\|\nabla_v^2 \varphi\|_\infty) + h \O(\|\nabla^2\varphi\|_\infty) \int_{\R^{2d}} |v|^2\,d\bar\mu_{(i+1)h}^N.
\end{split}
\end{equation}
\end{lemma}

\begin{proof}
Let $\gamma_{ih}\in\Gamma_o^v(\mu_{ih}^N,\bar\mu_{(i+1)h}^N)$. Fix a test function $\varphi\in C_c^\infty(\R^{2d})$. Then by the Taylor expansion,
\begin{align*}
    \int_{\R^{2d}}\varphi\,d(\bar\mu_{(i+1)h}^N-\mu_{ih}^N)
    =\int_{\R^d}\iint_{\R^d\times\R^d}(w-v)\cdot\nabla_v\varphi(x,w)\,d\gamma_{ih}^x(v,w)\Pi^x\mu_{ih}^N(x) + E_0
\end{align*}
where $E_0$ is an error term satisfying
\begin{equation}\label{eq:E0;ubd}
    |E_0|\leq \|\nabla_v^2\varphi\|_\infty W_{2,v}^2(\mu_{ih}^N,\bar\mu_{(i+1)h}^N) 
\end{equation}
Letting $\xi_v(x,v)=\nabla_v\varphi(x,v)$ in \eqref{eq:euler-lagrange;Jp;test} and integrating by parts, we have
\begin{align*}
&\int_{\R^d} \iint_{\R^d\times\R^d} (w-v)\cdot\nabla_v\varphi(x,w)\,d\gamma_{ih}^x(v,w)\,\Pi^x\pi_{ih}^N(x)
\\
&\quad = -h \int_{\R^{2d}} [\nabla_x V(x)+ (\nabla_x W\ast \Pi^x\bar\mu_{(i+1)h}^N)(x)+ \alpha v]\cdot \nabla_v\varphi(x,v)- \alpha\Delta_v \varphi(x,v)\,d\bar\mu_{(i+1)h}^N(x,v).
\end{align*}
We can thus combine the estimates and divide both sides by $h$ to deduce
\begin{equation}\label{eq:PDE;diss;2step}
\begin{split}
    &\frac1h\int_{\R^{2d}} \varphi\,d(\bar\mu_{(i+1)h}^N-\mu_{ih}^N)\\
    &\quad= -\int_{\R^{2d}} [\nabla_x V(x)+ (\nabla_x W\ast \Pi^x\bar\mu_{(i+1)h}^N)(x)+ \alpha v]\cdot \nabla_v\varphi(x,v)-\alpha\Delta_v \varphi(x,v)\,d\bar\mu_{(i+1)h}^N(x,v) \\
    &\quad+  \frac1h W_{2,v}^2(\mu_{ih}^N,\bar\mu_{(i+1)h}^N)  \O(\|\nabla_v^2 \varphi\|_\infty).
\end{split}
\end{equation}

On the other hand,
\begin{align*}
    \frac1h\int_{\R^{2d}} \varphi\,d(\mu_{(i+1)h}^N-\bar\mu_{(i+1)h}^N)
    =\frac1h\int_{\R^{2d}}\varphi(x+hv)-\varphi(x)\,d\bar\mu_{(i+1)h}^N(x,v)\\
    =\int_{\R^{2d}}\nabla_x\varphi(x,v)\cdot v\,d\bar\mu_{(i+1)h}^N(x,v)
    + h \O(\|\nabla^2\varphi\|_\infty) \int_{\R^{2d}} |v|^2\,d\bar\mu_{(i+1)h}^N(x,v) .
\end{align*}
Thus we obtain \eqref{eq:phi_dmu;discrete_time} by combining the above with \eqref{eq:PDE;diss;2step}.

\end{proof}

Using arguments inspired by \cite[Theorem 5.1]{JKO98} we take the limit as time step vanishes to show that minimizing movements are indeed distributional solutions of VFP.
\begin{proposition}[Minimizing movements are solutions of VFP]\label{prop:conv_PDE}
Let $\nabla_x V, \nabla_x W$ be Lipschitz. Fix $T>0$ and the time step $h=h(N)=T/N$ for each $N\in\N$ and
let $(\mu_{ih_N})_{i=1}^N$ be the sequence produced by the algorithm \eqref{eq:SIE} with $\mu_0^N\in D(\Lfun_v+\alpha\H)$ satisfying
\begin{equation}\label{cond:2nd_moment;unifbd}
    M_2:=\sup_{N\in\N} \int_{\R^{2d}} |z|^2\,d\mu_0^N(z) <+\infty
\end{equation}
and for some $C$ independent of $N$
\begin{equation}\label{cond:AC2;bound}
    \sum_{i=0}^{N-1} W_{2,v}^2(\mu_{ih}^N,\bar\mu_{(i+1)h}^N) \leq Ch.
\end{equation}
Let $(\mu_t^N)_{t\in[0,T]}$ the corresponding piecewise constant interpolation defined as in \eqref{def:const_interpol}.

Suppose $(\mu_t)_{t\in[0,T]}\in AC([0,T];\P_2(\R^{2d}))$ satisfies, over a suitable subsequence in $N$ (which we do not relabel),
\begin{equation}\label{eq:mutN-ptwisenarrowconv}
\mu_t^N\rightharpoonup \mu_t \text{ narrowly as } N\rightarrow\infty \text{ for } t=0 \text{ and a.e. } t\in(0,T]
\end{equation}
Then $(\mu_t)_{t\in[0,T]}$ is a weak solution of the Vlasov-Fokker-Planck equation -- i.e. 
\begin{equation}\label{def:VFP_distributional_sol}
\begin{split}
    \int_{0}^T\int_{\R^{2d}}\left[\partial_t\varphi+\begin{pmatrix}
            \id_v \\ -\nabla_x V-(\nabla_x W\ast\Pi^x\mu_{t}) - \alpha \id_v
        \end{pmatrix}\cdot \begin{pmatrix}
            \nabla_x \varphi \\ \nabla_v \varphi
        \end{pmatrix}
    +\alpha\Delta_v \varphi\right]\,d\mu_{t}\,dt\\
    =\int_{\R^{2d}}\varphi(T,z)\,d\mu_T(z)-\int_{\R^{2d}}\varphi(0,z)\,d\mu_0(z) \quad \text{ for any } \varphi\in C_c^\infty([0,T]\times\R^{2d}).
\end{split}
\end{equation}

\end{proposition}

\begin{proof}

\noindent\emph{Step 1$^o$.}
Fix $\varphi\in C_c^\infty([0,T]\times\R^{2d})$ and let $K_\varphi\subset\subset\R^{2d}$ such that $\supp\varphi\subset [0,T]\times K_\varphi$. Writing $\varphi(t,\cdot)=\varphi_t$ for simplicity, we may apply \eqref{eq:phi_dmu;discrete_time} for each $i=0,\cdots,N-1$ respectively to test function $\varphi_{(i+1)h}$ and sum to obtain
\begin{align*}
    &\sum_{i=1}^{N-1}\int_{\R^{2d}}\varphi_{ih}(z)-\varphi_{(i+1)h}(z)\,d\mu_i^N(z)+\int_{\R^{2d}}\varphi_{Nh}(z)\,d\mu_{Nh}^N - \int_{\R^{2d}} \varphi_h(z)\,d\mu_{0}^N  \\
    &=\sum_{i=0}^{N-1} \int_{\R^{2d}} \varphi_{(i+1)h}(z)\,d(\mu_{i+1}^N-\mu_i^N)(z) \\
    &= \sum_{i=0}^{N-1} h\int_{\R^{2d}} \begin{pmatrix}
            \id_v \\ -\nabla_x V-(\nabla_x W\ast\Pi^x\bar\mu_{(i+1)h}^N) - \alpha\id_v
        \end{pmatrix}\cdot \begin{pmatrix}
            \nabla_x \varphi_{(i+1)h} \\ \nabla_v \varphi_{(i+1)h}
        \end{pmatrix}
    +\alpha\Delta_v \varphi_{(i+1)h}\,d\bar\mu_{(i+1)h}^N \\
    &+  \O(\|\nabla_v^2 \varphi\|_\infty)\sum_{i=0}^{N-1} W_{2,v}^2(\mu_{ih}^N,\bar\mu_{(i+1)h}^N) + \O(\|\nabla^2\varphi\|_\infty)\sum_{i=0}^{N-1} h^2 \|\id_v\|_{L^2(\bar\mu_{(i+1)h}^N)}^2.
\end{align*}
Note that the last two terms on the right-hand side are of order $O(h)$. Indeed, by Jensen's inequality
\[\|\id_v\|_{L^2(\bar\mu_{(i+1)h}^N)}^2\leq \left(\|\id_v\|_{L^2(\bar\mu_0^N)}+\sum_{i=0}^{N-1}W_{2,v}(\mu_{ih}^N,\bar\mu_{(i+1)h}^N)\right)^2 \leq 2 \|\id_v\|_{L^2(\bar\mu_0^N)}^2 + N\sum_{i=0}^N W_{2,v}^2(\mu_{ih},\bar\mu_{(i+1)h}^N),\]
which is bounded uniformly in $N\in\ N$ and $i=0,1,\cdots,N-1$ due to the assumptions \eqref{cond:2nd_moment;unifbd} and \eqref{cond:AC2;bound}.
On the other hand, as $t\mapsto(\mu_t^N)_{t\in[0,T]}$ is piecewise constant,
\begin{align*}
    \sum_{i=1}^{N-1}\int_{\R^{2d}}\varphi_{ih}(z)-\varphi_{(i+1)h}(z)\,d\mu_i^N(z)=-\sum_{i=1}^{N-1}\int_{ih}^{(i+1)h}\int_{\R^{2d}} \partial_t\varphi(t,z)\,d\mu_t^N(z)\,dt    \\
    =-\int_{h}^{Nh} \int_{\R^{2d}} \partial_t\varphi(t,z)\,d\mu_t^N(z)\,dt
    = \int_0^T \int_{\R^{2d}} \partial_t\varphi(t,z)\,d\mu_t^N(z)\,dt + O(h).
\end{align*}
Thus, defining $\varphi^N(t,z):=\varphi(ih,z)$ and $\bar\mu_t^N=\bar\mu_{(i+1)h}^N$ for $t\in [ih,(i+1)h)$, we have
\begin{equation}\label{eq:phi_dmuN;PDE}
\begin{split}
    &-\int_0^T \partial_t\varphi(t,z)\,d\mu_t^N(z)+\int_{\R^{2d}}\varphi(T,z)\,d\mu_T^N-\int_{\R^{2d}}\varphi(0,z)\,d\mu_0^N(z)   \\
    &= \int_0^T \int_{\R^{2d}} \begin{pmatrix}
            \id_v \\ -\nabla_x V-(\nabla_x W\ast\Pi^x\bar\mu_{t}^N) + \alpha\id_v
        \end{pmatrix} \cdot \nabla\varphi^N + \alpha\Delta_v \varphi^N\,d\bar\mu_t^N\,dt + O(h).
\end{split}
\end{equation}

\vspace{3mm}
\noindent\emph{Step 2$^o$.} We now study the limit of the linear terms \eqref{eq:phi_dmuN;PDE} as $N\rightarrow\infty$. As $\varphi\in C_c^\infty([0,T]\times\R^{2d})$, we can directly pass the left-hand side of the equation to the limit. We claim that
\begin{equation}\label{eq:VFPlinterms_limit}
\begin{split}
    \lim_{N\rightarrow\infty}\int_0^T \int_{\R^{2d}} \begin{pmatrix}
            \id_v \\ -\nabla_x V - \alpha\id_v
        \end{pmatrix} \cdot \nabla\varphi^N + \alpha\Delta_v\varphi^N\,d\bar\mu_t^N\,dt
        \\
        =\int_0^T \int_{\R^{2d}} \begin{pmatrix}
            \id_v \\ -\nabla_x V - \alpha\id_v
        \end{pmatrix} \cdot \nabla\varphi + \alpha\Delta_v\varphi\,d\mu_t\,dt.
\end{split}
\end{equation}
Indeed, as $\id_v,\nabla_x V$ are continuous thus bounded on $K_\varphi\subset\subset\R^{2d}$,
\begin{equation}\label{eq:phiN-phi;linear}
\left|\int_0^T \int_{\R^{2d}} \begin{pmatrix}
            \id_v \\ -\nabla_x V - \alpha\id_v
        \end{pmatrix} \cdot \nabla(\varphi^N-\varphi) + \alpha\Delta_v (\varphi^N-\varphi)\,d\bar\mu_t^N\,dt\right|
\leq C h\|\varphi\|_{C^3} T
\end{equation}
for some $C=C(V,K_\varphi)>0$. Furthermore, by the dominated convergence theorem
\begin{align*}
    \int_0^T \int_{\R^{2d}} \begin{pmatrix}
            \id_v \\ -\nabla_x V - \alpha\id_v
        \end{pmatrix} \cdot \nabla\varphi + \alpha\Delta_v \varphi\,d\bar\mu_t^N\,dt\rightarrow\int_0^T \int_{\R^{2d}} \begin{pmatrix}
            \id_v \\ -\nabla_x V - \alpha\id_v
        \end{pmatrix} \cdot \nabla\varphi + \alpha\Delta_v \varphi\,d\mu_t\,dt
\end{align*}
as $N\rightarrow\infty$. From this and \eqref{eq:phiN-phi;linear} we deduce \eqref{eq:VFPlinterms_limit}.

\vspace{3mm}
\noindent\emph{Step 3$^o$.} 
It remains to pass the nonlinear term in \eqref{eq:phi_dmuN;PDE} to the limit. Let $\Lip\nabla_x V,\Lip\nabla_x W\leq M$. first notice that the uniform bound on the second moment \eqref{cond:2nd_moment;unifbd} implies, by Lemma~\ref{lem:dWVip_growthbd}, that there exists a constant $C=C(M,M_2)$ independent of $N\in\N$ and $i=1,\cdots,N$ such that
\begin{equation}\label{eq:Emu_ih;upperbd}
    E(\mu_{ih}^N)=\|\nabla_x V\|_{L^2(\mu_{ih}^N)}+\|\nabla_x W\ast\Pi^x\bar\mu_{ih}^N\|_{L^2(\bar\mu_{ih}^N)}+\|\id_v\|_{L^2(\mu_{ih}^N)} \leq C
\end{equation}
for all $N\in\N$ and $i=0,1,\cdots,N$. Indeed, it is easy to check that $E(\mu_0^N)\leq C'$ for some $C'=C'(M,M_2)>0$ follows from \eqref{cond:2nd_moment;unifbd} and the Lipschitz continuity of $\nabla_x V,\nabla_x W$. Then applying \eqref{eq:WVpmoment_growthbd} with $h=h_N=T/N$ and a discrete Gr\"onwall's inequality provides the uniform upper bound \eqref{eq:Emu_ih;upperbd}.

Moreover, the uniform upper bound and along with \eqref{cond:AC2;bound} allows us to control $W_2(\mu_0^N,\mu_{ih}^N)$ uniformly for all $N\in\N$ and $i\leq N$ as
\begin{align*}
    W_2^2(\mu_0^N,\mu_{ih}^N)\leq \left(\sum_{j=0}^{i-1} W_2(\mu_{jh}^N,\mu_{(j+1)h}^N)\right)^2 &\leq  2i\sum_{j=0}^{i-1} W_{2,v}^2(\mu_{jh}^N,\bar\mu_{(j+1)h}^N)+W_{2,x}^2(\bar\mu_{(j+1)h}^N,\mu_{(j+1)h}^N) \\
    &= 2i\sum_{j=0}^{i-1} W_{2,v}^2(\mu_{jh}^N,\bar\mu_{(j+1)h}^N) + h^2\|\id_v\|_{L^2(\bar\mu_{(j+1)h}^N)}^2
\end{align*}
In particular, the second moments of $\mu_{ih}^N$ are uniformly bounded.
Letting $\gamma_t\in\Gamma_o(\mu_t,\bar\mu_t^N)$, we have
\begin{align*}
    &\int_{\R^{2d}} (\nabla_x W\ast\Pi^x\bar\mu_t^N)(x)\cdot\nabla_v \varphi(t,x,v) d\bar\mu_t^N-\int_{\R^{2d}} (\nabla_x W\ast\Pi^x\mu_{t})(x)\cdot\nabla_v \varphi(t,x,v) d\mu_{t}(x,v)   \\
    &= \int_{\R^{2d}} (\nabla_x W\ast\Pi^x\bar\mu_t^N)(y)\cdot\nabla_v\varphi(t,y,w)-(\nabla_x W\ast\Pi^x\mu_{t})(x)\cdot\nabla_v \varphi(x,v)\,d\gamma_t((x,v),(y,w))\\
    &= \int_{\R^{2d}} \nabla_x W\ast \Pi^x\bar\mu_t^N(y)\cdot(\nabla_v\varphi(t,y,w)-\nabla_v\varphi(t,x,v))\,d\gamma_t((x,v),(y,w))  \\
    &+ \int_{\R^{2d}} (\nabla_x W\ast \Pi^x\bar\mu_t^N(y)-\nabla_x W\ast\Pi^x\mu_{t}(x))\cdot\nabla_v\varphi(t,x,v))\,d\gamma_t((x,v),(y,w))  \\
    &\leq \|\nabla^2\varphi\|_\infty W_2(\bar\mu_t^N,\mu_t) \|\nabla_x W\ast\Pi^x\bar\mu_t^N\|_{L^2(\bar\mu_t^N)}
    + M \|\nabla_v\varphi\|_\infty W_2(\bar\mu_t^N,\mu_t).
\end{align*}
The estimate \eqref{eq:WVpmoment_growthbd} implies a uniform bound on $|\partial_x\W|(\mu_{ih}^N)=|\partial_x\W|(\bar\mu_{ih}^N)=\|\nabla_x W\ast\Pi^x\bar\mu_{ih}^N\|_{L^2(\bar\mu_{ih}^N)}$ only dependent on $\mu_0^N$ and $M,T>0$. Moreover,
by the lower semicontinuity of $W_2$ with respect to narrow convergence, for each $t\in[0,T]$, writing $i_-=i_-(N,t)$ defined \eqref{def:i_vm}
\begin{equation}\label{eq:W2barmut_conv}
\begin{split}
    W_2(\bar\mu_t^N,\mu_t)=W_2(\bar\mu_{(i_-+1)T/N}^N,\mu_t)&\leq W_2(\bar\mu_{(i_-+1)T/N}^N,\mu_{t}^N)+W_2(\mu_t^N,\mu_t)\\
    &=W_2(\bar\mu_{(i_-+1)T/N}^N,\mu_{i_-T/N}^N)+W_2(\mu_t^N,\mu_t)\xrightarrow[]{N\uparrow\infty} 0.
\end{split}
\end{equation}
Hence we can find $C=C(M,\mu_0^N)$ such that
\begin{align*}
    \left|\int_0^T \int_{\R^{2d}} \nabla_x W\ast\Pi^x\bar\mu_t^N \cdot \nabla_v\varphi_t\,d\bar\mu_t^N\,dt-\int_0^T \int_{\R^{2d}} \nabla_x W\ast\Pi^x\mu_t\cdot \nabla_v\varphi_t\,d\mu_t^N\,dt\right| \\
    \leq C\|\varphi\|_{C^2}\int_0^T W_2(\bar\mu_t^N,\mu_t)\,dt
    \xrightarrow[]{N\rightarrow\infty} 0
\end{align*}
where we have applied the dominated convergence theorem in the last line, as $t\mapsto W_2(\bar\mu_t^N,\mu_t)$ is bounded due to uniform second moment bound and converges a.e. to $0$ as $N\rightarrow\infty$ by \eqref{eq:W2barmut_conv}. 
\end{proof}

\begin{remark}\label{rmk:conv_PDE}
    If $W\equiv 0$ then the third step of the proof of Proposition~\ref{prop:conv_PDE} is unnecessary, hence we can replace the Lipschitz-continuity of $\nabla_x V$ with continuous differentiability of $V$.
\end{remark}

The main result of this paper then follows directly from Proposition~\ref{prop:conv_PDE} and the existence of minimizing movements (Theorems~\ref{thm:disc_sol;subseq_conv} and~\ref{thm:disc_sol;cpct2}).
\begin{theorem}\label{thm:main}
    Suppose $\nabla_x V,\nabla_x W$ are Lipschitz continuous, and let $\mu_0$ satisfy either
    \begin{listi}
        \item $\H(\mu_0)<+\infty$ and $\inf_{\nu\P_2(\R^{2d})} \H(\nu)>-\infty$, or
        \item $|\partial_v\H|(\mu_0)<+\infty$.
    \end{listi}
    Fix $T>0$, and for each $N\in\N$ let $h_N=T/N$. Let $(\mu_{ih_N}^N)_{i=0}^N$ be as defined in \eqref{eq:SIE} with the initial data $\mu_0^N$ narrowly converging as $N\nearrow\infty$ to some $\mu_0\in\P_2(\R^{2d})$ satisfying $\H(\mu_0)<+\infty$. Denote by $(\mu_t^N)_{t\in [0,T]}$ the corresponding piecewise-constant interpolations defined as in \eqref{def:const_interpol}.
    
    Then $\mu_t^N\xrightharpoonup[]{N\rightarrow\infty} \mu_t$ narrowly for all $t\in[0,T]$
    where $(\mu_t)_{t\in[0,T]}\in AC([0,T];\P_2(\R^{2d}))$ is the unique weak solution of the Vlasov-Fokker-Planck equation in the sense of \eqref{def:VFP_distributional_sol}.
\end{theorem}

\begin{proof}
As $\nabla_x V,\nabla_x W$ are Lipschitz continuous, all coefficients appearing in the PDE are Lipschitz continuous. Thus uniqueness of weak solution of the Vlasov-Fokker-Planck equation in $\P_2(\R^{2d})$ is ensured for instance by \cite[Theorem 2.2]{Meleard96}. In each case (i) and (ii), existence of minimizing movements follows respectively by Theorem~\ref{thm:disc_sol;cpct2} and Theorem~\ref{thm:disc_sol;subseq_conv}. 

Thus, once we verify that the discrete solutions satisfy the conditions of Proposition~\ref{prop:conv_PDE} the claim follows. As the uniform bound on the second moments of the initial data \eqref{cond:2nd_moment;unifbd} follows from the convergence of $\mu_0^N$ to $\mu_0$ in $\P_2(\R^{2d})$, it remains to verify \eqref{cond:AC2;bound}. In case (i), we have already verified this in the proof of Theorem~\ref{thm:disc_sol;cpct2}. In case (ii), recall \eqref{eq:W2v_leq_slope} which states
\[ W_{2,v}(\mu_{ih}^N,\bar\mu_{(i+1)h}^N)\leq \frac{h}{1+\alpha h} |\partial_v(\Lfun_v+\alpha\H)|(\mu_{ih}^N),\]
whereas the slope $|\partial_v(\Lfun_v+\alpha\H)|$ is uniformly bounded by Proposition~\ref{prop:equicts;discrete_sol}. From this the condition \eqref{cond:AC2;bound} follows easily. 
\end{proof}

\begin{remark}[General time partition]\label{rmk:timestep}
    For simplicity we have restricted our attention to uniform time step $h_N=T/N$ in Theorem~\ref{thm:main}. However, from the proofs it is clear that ths convergence of the discrete solution holds for general partitions $\bm{h}_N=\{h_i:\;1\leq i\leq N, 0=h_0<h_1<\cdots<h_N=T\}$ of $[0,T]$ such that the modulus $|\bm{h}_N|$ tends to zero. Indeed, the errors terms in the key discrete estimates in Section~\ref{sec:SIE;cpct} (Proposition~\ref{prop:equicts;discrete_sol}, Lemma~\ref{lem:energy_decay;disc}) and Section~\ref{sec:MMS=VFP} (Lemma~\ref{eq:phi_dmu;discrete_time}) vanishes as $|\bm{h}_N|\searrow 0$.
\end{remark}

\section{Decay of the Hamiltonian along discrete solutions}\label{sec:Hdecay;disc}
We conclude this paper by establishing a two-sided Hamiltonian decay bound over discrete solutions that follows easily from energy estimates of Section~\ref{sec:SIE;cpct}. 

\begin{proposition}[Dissipation of the Hamiltonian at the discrete level]\label{prop:Hdecay;disc}
 Suppose $\nabla_x V,\nabla_x W$ are $M$-Lipschitz. Let $(\mu_{ih}^N)_{i=0}^N$ be defined by \eqref{eq:SIE} with time step $h>0$, number of iterations $N\in\N$, and initial datum $\mu_0^N\in D(\H)$. If $Nh\leq T$, then 
\begin{equation}\label{eq:Hdecay;ubd}
\begin{split}
    \frac{\H(\mu_{(i+1)h}^N)-\H(\mu_{ih}^N)}{h}
    &\leq -\alpha \|\id_v+\nabla_v\rho_{(i+1)h}^N/\rho_{(i+1)h}^N\|_{L^2(\mu_{(i+1)h}^N)}^2\\
    &-\frac{W_{2,v}^2(\mu_{ih}^N,\bar\mu_{(i+1)h}^N)}{2h} + M h\|\id_v\|_{L^2(\mu_{(i+1)h}^N)}^2,
\end{split}
\end{equation}
and
\begin{equation}\label{eq:Hdecay;lbd}
\begin{split}
     \frac{\H(\mu_{(i+1)h}^N)-\H(\mu_{ih}^N)}{h}&\geq - \frac{\alpha}{1+\alpha h}\|\id_v+\nabla_v\rho_{ih}^N/\rho_{ih}^N\|_{L^2(\mu_{ih}^N)}^2+\frac{W_{2,v}^2(\mu_{ih}^N,\bar\mu_{(i+1)h}^N)}{2h}   \\
    &-\frac{h}{1+\alpha h}\|\nabla_x V+\nabla_x W\ast\Pi^x\mu_{ih}^N\|_{L^2(\mu_{ih}^N)}^2
    - Mh\|\id_v\|_{L^2(\mu_{(i+1)h}^N)}^2.
\end{split}
\end{equation}

In particular, there exists $C=C(T,M,\mu_0^N)>0$ such that
\begin{equation}\label{eq:Hdecay;disc}
\begin{split}
    - \frac{\alpha}{1+\alpha h}|\partial_v\H|^2(\mu_{ih}^N)-Ch\leq\frac{\H(\mu_{(i+1)h}^N)-\H(\mu_{ih}^N)}{h}\leq -\alpha |\partial_v\H|^2(\mu_{(i+1)h}^N)+ Ch.
\end{split}
\end{equation}
\end{proposition}

\begin{remark}\label{rmk:Hdecay;disc}
    Observe that \eqref{eq:Hdecay;disc} resembles the decay of the Hamiltonian along the solution of the Vlasov-Fokker-Planck equation
    \[\frac{d}{dt}\H(\mu_t)=-\alpha|\partial_v\H|^2(\mu_t)=-\alpha\norm{\id_v+\frac{\nabla_v\rho_t}{\rho_t}}_{L^2(\mu_t)}^2.\]
    Moreover, as we can see from the error terms in\eqref{eq:Hdecay;lbd}-\eqref{eq:Hdecay;ubd}, the constant $C$ in \eqref{eq:Hdecay;disc} are controlled by the Lipschitz constants of $V,W$ and the second moments of the discrete solutions $(\mu_{ih}^N)_{i=0}^N$.
\end{remark}

\begin{proof}

\noindent\emph{Step 1$^o$.} (Upper bound)
Recall that in Lemma~\ref{lem:energy_decay;disc} we had
\begin{equation}\label{eq:Hdiff;ubd}
\begin{split}
    &\alpha(\H(\mu_{(i+1)h}^N)-\H(\mu_{ih}^N)) 
    \leq -(1+\frac{\alpha h}{2})\frac{W_{2,v}^2(\mu_{ih}^N,\bar\mu_{(i+1)h}^N)}{h} \\
    &+h\|\nabla_x V+\nabla_x W\ast\Pi^x\bar\mu_{(i+1)h}^N\|_{L^2(\bar\mu_{(i+1)h}^N)}^2
    +2\alpha h\Lfun_v(\bar\mu_{(i+1)h}^N)+M\alpha h^2\|\id_v\|_{L^2(\bar\mu_{(i+1)h}^N)}^2.
\end{split}
\end{equation}
By the slope estimates \eqref{eq:partial_slope_est}, each $\bar\mu_{(i+1)h}^N\in D(|\partial_v\U|)$ and thus we can apply \eqref{eq:xidvlogrho=0} to deduce
\[\langle \nabla_x V+\nabla_x W\ast\Pi^x\bar\mu_{(i+1)h}^N,\nabla_v\bar\rho_{(i+1)h}/\bar\rho_{(i+1)h}^N\rangle_{L^2(\bar\mu_{(i+1)h}^N)}= 0.
\]
Expanding the squares and using the slope estimate \eqref{eq:partial_slope_est}
\begin{align*}
    &h\|\nabla_x V+\nabla_x W\ast\Pi^x\bar\mu_{(i+1)h}^N\|_{L^2(\bar\mu_{(i+1)h}^N)}^2+2\alpha h\Lfun_v(\bar\mu_{(i+1)h}^N)+\alpha^2h\|\id_v+\nabla_v\bar\rho_{(i+1)h}^N/\bar\rho_{(i+1)h}^N\|_{L^2(\bar\mu_{(i+1)h}^N)}^2\\
    &\qquad=h\|\nabla_x V+W\ast\Pi^x\bar\mu_{(i+1)h}^N+\alpha(\id_v+\nabla_v\bar\rho_{(i+1)h}^N/\bar\rho_{(i+1)h}^N)\|_{L^2(\bar\mu_{(i+1)h}^N)}^2\\
    &\qquad=h|\partial_v(\Lfun_v+\alpha\H)|^2(\bar\mu_{(i+1)h}^N)\leq \frac{W_{2,v}^2(\mu_{ih}^N,\bar\mu_{(i+1)h}^N)}{h}.
\end{align*}
Reorganizing,
\begin{align*}
    -\frac{W_{2,v}^2(\mu_{ih}^N,\bar\mu_{(i+1)h}^N)}{h} &+h\|\nabla_x V+W\ast\Pi^x\bar\mu_{(i+1)h}^N\|_{L^2(\bar\mu_{(i+1)h}^N)}^2+2\alpha h\Lfun_v(\bar\mu_{(i+1)h}^N)\\
    &\leq -\alpha^2h\|\id_v+\nabla_v\bar\rho_{(i+1)h}^N/\bar\rho_{(i+1)h}^N\|_{L^2(\bar\mu_{(i+1)h}^N)}^2
\end{align*}
As $\mu\mapsto \|\id_v+\nabla_v\rho/\rho\|_{L^2(\mu)}$ is invariant under pushforward by the map $(x,v)\mapsto (x+hv,v)$ (see Corollary~\ref{cor:psubdiff;UPhih}),
\[\|\id_v+\nabla_v\bar\rho_{(i+1)h}^N/\bar\rho_{(i+1)h}^N\|_{L^2(\bar\mu_{(i+1)h}^N)}^2=\|\id_v+\nabla_v\rho_{(i+1)h}^N/\rho_{(i+1)h}^N\|_{L^2(\mu_{(i+1)h}^N)}^2.\]
Thus plugging into \eqref{eq:Hdiff;ubd} we obtain
\begin{equation*}
\begin{split}
    \alpha(\H(\mu_{(i+1)h}^N)-\H(\mu_{ih}^N))
    &\leq -\alpha^2 h\|\id_v+\nabla_v\rho_{(i+1)h}^N/\rho_{(i+1)h}^N\|_{L^2(\mu_{(i+1)h}^N)}^2\\
    &-\frac{\alpha W_{2,v}^2(\mu_{ih}^N,\bar\mu_{(i+1)h}^N)}{2} + M\alpha h^2\|\id_v\|_{L^2(\bar\mu_{(i+1)h}^N)}^2.
\end{split}
\end{equation*}
Dividing both sides by $\alpha h$ we have the upper bound \eqref{eq:Hdecay;ubd}.

\vspace{3mm}
\noindent\emph{Step 2$^o$.} (Lower bound)
In order to obtain the lower bound, note that by the mean value theorem we have
\[\H(\mu_{(i+1)h}^N)-\H(\bar\mu_{(i+1)h}^N)=(\W+\V)(\mu_{(i+1)h}^N)-(\W+\V)(\bar\mu_{(i+1)h}^N)\geq h\Lfun_v(\bar\mu_{(i+1)h}^N)-M h^2\|\id_v\|_{L^2(\mu_{(i+1)h}^N)}^2.\]
Moreover, by the partial slope estimates \eqref{eq:partial_slope_est}
\begin{align*}
    \alpha&(\H(\mu_{ih}^N)-\H(\bar\mu_{(i+1)h}^N))+\Lfun_v(\mu_{ih}^N)-\Lfun_v(\bar\mu_{(i+1)h}^N)  \\
    &\leq \frac{W_{2,v}^2(\mu_{ih}^N,\bar\mu_{(i+1)h}^N)}{2h} + \frac{h}{2(1+\alpha h)}|\partial_v(\Lfun_v+\alpha\H)|^2(\mu_{ih}^N) \\
    &\leq \frac{h}{1+\alpha h}|\partial_v(\Lfun_v+\alpha\H)|^2(\mu_{ih}^N)-\frac{\alpha}{2}W_{2,v}^2(\mu_{ih}^N,\bar\mu_{(i+1)h}^N).
\end{align*}
where the last line uses \eqref{eq:partial_slope_est} in the form $W_{2,v}^2(\mu_{ih}^N,\bar\mu_{(i+1)h}^N)/h\leq \frac{h}{(1+\alpha h)^2}|\partial_v(\Lfun_v+\alpha\H)|^2(\mu_{ih}^N)$. 
Combining the two, we have
\begin{align*}
    &\alpha(\H(\mu_{(i+1)h}^N)-\H(\mu_{ih}^N))=\alpha(\H(\mu_{(i+1)h}^N)-\H(\bar\mu_{(i+1)h}^N))+\alpha(\H(\bar\mu_{(i+1)h}^N)-\H(\mu_{ih}^N))\\
    &\geq -(1-\alpha h)\Lfun_v(\bar\mu_{(i+1)h}^N)+\Lfun_v(\mu_{ih}^N)-\frac{h}{1+\alpha h}|\partial_v(\Lfun_v+\alpha\H)|^2(\mu_{ih}^N) \\
    &+\frac{\alpha}{2}W_{2,v}^2(\mu_{ih}^N,\bar\mu_{(i+1)h}^N)-\alpha M h^2\|\id_v\|_{L^2(\mu_{(i+1)h}^N)}^2.
\end{align*}
As we know from Step 1$^o$
\begin{align*}
    |\partial_v(\Lfun_v+\alpha\H)|^2(\mu_{ih}^N)
    =\|\nabla_x V+\nabla_x W\ast\Pi^x\mu_{ih}^N\|_{L^2(\mu_{ih}^N)}^2+ \alpha^2 \|\id_v+\nabla_v\rho_{ih}^N/\rho_{ih}^N\|_{L^2(\mu_{ih}^N)}^2 + 2\alpha\Lfun_v(\mu_{ih}^N),
\end{align*}
observe
\begin{align*}
    &-\frac{h}{1+\alpha h}|\partial_v(\Lfun_v+\alpha\H)|^2(\mu_{ih}^N)-(1-\alpha h)\Lfun_v(\bar\mu_{(i+1)h}^N)+\Lfun_v(\mu_{ih}^N)\\
    &=-\frac{h}{1+\alpha h}(\alpha^2\|\id_v+\nabla_v\rho_{ih}^N/\rho_{ih}^N\|_{L^2(\mu_{ih}^N)}^2+\|\nabla_x V+\nabla_x W\ast\Pi^x\mu_{ih}^N\|_{L^2(\mu_{ih}^N)}^2)\\
    &-(1-\alpha h)\Lfun_v(\bar\mu_{(i+1)h}^N)+\frac{1-\alpha h}{1+\alpha h}\Lfun_v(\mu_{ih}^N).
\end{align*}
Applying \eqref{eq:Lp_diff;eq} in the form $\Lfun_v(\mu_{ih}^N)=(1+\alpha h)\Lfun_v(\bar\mu_{(i+1)h}^N)+h\|\nabla_x V+\nabla_x W\ast\Pi^x\mu_{ih}^N\|_{L^2(\mu_{ih}^N)}^2$,
\begin{align*}
    -(1-\alpha h)\Lfun_v(\bar\mu_{(i+1)h}^N)+\frac{1-\alpha h}{1+\alpha h}\Lfun_v(\mu_{ih}^N)
    =\frac{h(1-\alpha h)}{1+\alpha h}\|\nabla_x V+\nabla_x W\ast\Pi^x\mu_{ih}^N\|_{L^2(\mu_{ih}^N)}^2.
\end{align*}
Thus we can simplify the expressions and obtain the lower bound \eqref{eq:Hdecay;lbd}
\begin{align*}
    &\alpha(\H(\mu_{(i+1)h}^N)-\H(\mu_{ih}^N))\geq - \frac{\alpha^2 h}{1+\alpha h}\|\id_v+\nabla_v\rho_{ih}^N/\rho_{ih}^N\|_{L^2(\mu_{ih}^N)}^2\\
    &-\frac{\alpha h^2}{1+\alpha h}\|\nabla_x V+\nabla_x W\ast\Pi^x\mu_{ih}^N\|_{L^2(\mu_{ih}^N)}^2-\alpha Mh^2\|\id_v\|_{L^2(\mu_{ih}^N)}^2+\frac{\alpha}{2}W_{2,v}^2(\mu_{(i+1)h}^N,\bar\mu_{(i+1)h}^N).
\end{align*}

To conclude \eqref{eq:Hdecay;disc} first note that we can ignore the terms $W_{2,v}^2(\mu_{(i+1)h}^N,\bar\mu_{(i+1)h}^N)$ in \eqref{eq:Hdecay;lbd}-\eqref{eq:Hdecay;ubd} due to their signs. Remaining error terms are bounded by $Ch$ for some constant $C$ only depending on $T,M,\mu_0^N$ by Lemma~\ref{lem:dWVip_growthbd}.
\end{proof}

\textbf{Acknowledgements.}
The author is grateful to Dejan Slep\v{c}ev and Lihan Wang for stimulating discussions and valuable comments. The author would also like to thank Jeremy Sheung-Him Wu for pointing us to various related works. This work was partially supported by National Science Foundation via the grant DMS-2206069 (PI: Dejan Slep\v{c}ev) and DMS-2106534 (PI: Robert Pego).
\smallskip

\textbf{Data availability.}
We do not analyze or generate any datasets, because our work proceeds within a theoretical and mathematical approach.

\bibliographystyle{siam}
\bibliography{momentum_bib.bib}

\begin{appendix}

\section{Proofs of results in Section~\ref{ssec:partial_subdiff;U}}\label{app:proof;U}
This section provides various proofs of results regarding the convexity subdifferential of the internal energy functional $\U$ in Section~\ref{ssec:partial_subdiff;U}.
\begin{proof}[Proof of Lemma~\ref{lem:pdirder;U}]
    Let $\mu=\int_{\R^d} \mu^x\,d\Pi^x\mu$, where the disintegration $\mu^x\in\P^r(\R^d)$ is uniquely defined up to $\Pi^x\mu$-null set. As $\mu^x\in\P^r(\R^d)$, by \cite[Theorem 6.1.3]{AGS} the optimal transport map $\br^x=T_{\mu^x}^{\nu^x}$ is approximately differentiable at $v$ and $\tilde\nabla_v T_{\mu^x}^{\nu^x}(v)$ is diagonalizable with nonnegative eigenvalues. Let
\[\br(x,v)=(x,\br^x(v)),\qquad \br_t(x,v)=(1-t)(x,v)+t\br(x,v)=(x,(1-t)v+t\br^x(v)),\]
where $\br$ is rigorously defined as in Proposition~\ref{prop:OTmap_partial}.
By the push-forward formula \cite[Lemma 5.5.3]{AGS}, the density $\rho_t=\frac{d(\br_t)_{\#}\mu)}{d\Leb^{2d}}$ satisfies, by considering $v\mapsto \rho_t(x,v)$ for each fixed $x\in\R^d$,
\[\rho_t(y,w)=\frac{\rho(y,(\br_t^x)^{-1}(w))}{\det\tilde\nabla_v\br_t^x((\br_t^x)^{-1}(w))} \text{ for } \mu\text{-a.e. }(y,w)\in\R^d\times\R^d.\]
Thus, by the change of variables $v=\br_t^x(v)$ and that $|\det\tilde\nabla_v\br_t(v)|>0$ for all $v\in\R^d$ whenever $t>0$, we have,
\begin{align*}
    \U((\br_t)_{\#}\mu)-\U(\mu) &= \int_{\R^{2d}} U\left(\frac{\rho(x,(\br_t^x)^{-1}(w))}{\det\tilde\nabla_v\br_t^x((\br_t^x)^{-1}(w))}\right)\,dq\,dx - \int_{\R^{2d}} U(\rho(x,v))\,dv\,dx \\
    &=\int_{\R^{2d}} U\left(\frac{\rho(x,v)}{\det\tilde\nabla_v\br_t^x(v))}\right)\,\det\tilde\nabla_v\br_t(v)\,dv\,dx - \int_{\R^{2d}} U(\rho(x,v))\,dv\,dx\\
    &= \int_{\R^{2d}} G(\rho(x,v),\det\tilde\nabla_v\br_t^x(v))-U(\rho(x,v))\,dv\,dx
\end{align*}
where $G(z,s):=sU(z/s)$. Note that $G(\rho(x,v),\det\tilde\nabla_v\br_t^x(v))$ is the composition of the convex nonincreasing map $s\mapsto s^d U(\rho(x,v)/s^d)$ and the map
\[t\mapsto\det(\tilde\nabla_v\br_t^x)^{1/d}=\det((1-t)\id_v+t\tilde\nabla_v\br^x)^{1/d}\]
which is concave in $t\in[0,1]$ by condition (ii) that $\tilde\nabla_v\br^x$ is diagonalizable with nonnegative eigenvalues for a.e. $\Pi^x\mu$-a.e. $x\in\R^d$. Thus $t\mapsto G(\rho(x,v),\det\tilde\nabla_v\br_t^x(v))$ is convex for $\mu$-a.e. $(x,v)$, which directly implies
\[G(\rho(x,v),\det\tilde\nabla_v\br_t^x(v))\leq (1-t)U(\rho(x,v))+tG(\rho(x,v),\det\tilde\nabla_v\br^x(v)) \text{ for } t\in(0,1].\]
Upon integrating in $x,v$ we directly obtain
\[\U((\br_t)_\#\mu)\leq (1-t)\U(\mu)+tU((\br)_\#\mu).\]
As for any $\nu\in \P_2(\R^d)$ with $\Pi^x\nu=\Pi^x\mu$ by Proposition~\ref{prop:OTmap_partial} we can find such $\br^x:=T_{\mu^x}^{\nu^x}$ this implies $\U$ is $0$-convex along $W_{2,v}$-geodesics. 

Furthermore, by taking instead $\br_{\bar t}$ for some $\bar t>0$ of $\br=\br_1$, we see that 
\[(\br_{\bar t})_t=(1-t)\id_v+t\br_{\bar t}=(1-t)\id_v+t(1-\bar t)\id_v+t \bar t\br=\br_{t \bar t}\]
and 
\begin{align*}
    G(\rho(x,v),\det\tilde\nabla_v\br_{t\bar t}^x(v))-U(\rho(x,v))\leq (1-t)U(\rho(x,v))+tG(\rho(x,v),\det\tilde\nabla_v\br_{\bar t}^x(v)),
\end{align*}
which, after rearranging,
\begin{align*}
    \frac{G(\rho(x,v),\det\tilde\nabla_v\br_{t\bar t}^x(v))-U(\rho(x,v))}{t\bar t}\leq \frac{G(\rho(x,v),\det\tilde\nabla_v\br_{\bar t}^x(v))-U(\rho(x,v))}{\bar t} \text{ for } t\in[0,1]
\end{align*}
-- i.e. that 
\[\frac{G(\rho(x,v),\det\tilde\nabla_v\br_t^x(v))-U(\rho(x,v))}{t}\quad t\in (0,\bar t]\]
is nondecreasing w.r.t. $t$ and bounded above by an integrable function, as seen in the case $t=\bar t$. Thus applying the monotone convergence theorem and noting
\[\frac{\partial}{\partial s}G(z,s)=\frac{\partial}{\partial s} (sU(z/s))=U(z/s)-\frac{z}{s}U'(z/s)
=-\frac{z}{s}\]
we have
\begin{equation}\label{eq:dirder;MCT}
\begin{split}
    +\infty&>\lim_{t\downarrow 0}\frac{\U((\br_t)_{\#}\mu)-\U(\mu)}{t}=\int_{\R^{2d}}\frac{d}{dt}\left.G(\rho(x,v),\det\tilde\nabla_v\br_t^x(v))\right|_{t=0}\,dv\,dx   \\
    &=-\int_{\R^{2d}} \left.\frac{\rho(x,v)}{\det\tilde\nabla_v\br_t^x(v)}\frac{d}{dt}\det\tilde\nabla_v\br_t^x(v)\right|_{t=0}\,dv\,dx
    =-\int_{\R^{2d}} \tr\tilde\nabla(\br-\id)\rho(x,v)\,dxdv
\end{split}
\end{equation}
where we have used $\det\tilde\nabla_v\br_t=1+t\tr\tilde\nabla_v(\br^x-\id)+o(t)$.

Finally consider the case when \eqref{cond:xi_v;Linftybd} holds instead of nonnegative eigenvalues of $\tilde\nabla_v(\br^x-\id_v)$. In this case, we replace the monotone convergence theorem in \eqref{eq:dirder;MCT} by the dominated convergence theorem. Indeed, the assumption \eqref{cond:xi_v;Linftybd} allows us to uniformly control $|\det\tilde\nabla_v(\br^x_t-\id_v)|$ for sufficiently small $t>0$, whereas
\[\left|\frac{\partial}{\partial s} G(z,s)\right|=|z/s|\leq 2|z| \text{ for } |s-1|\leq 1/2.\]
Thus we can pass to the limit $t\downarrow 0$ in \eqref{eq:dirder;MCT}.
\end{proof}

\begin{proof}[Proof of Theorem~\ref{thm:psubdiff;U}]
    
\noindent\emph{Step 1$^o$.} Suppose $|\partial_v\U|(\mu)<+\infty$. To see $\int_{\R^d} |\nabla_v\rho|\,dz<\infty$, fix any $u\in C_c^\infty(\R^d;\R^d)$ and set $\br^x(v):=u(v)$ for $\Pi^x\mu$-a.e. $x\in\R^{d}$. Defining $\br(x,v):=(x,\br^x(v))$ as in Proposition~\ref{prop:OTmap_partial}, note
\begin{align*}
    W_{2,v}(\mu,((1-t)\id+t\br)_\#\mu) \leq t\|\br-\id\|_{L^2(\mu)}= t\|u\|_{L^2(\mu)}.
\end{align*}
By \eqref{eq:pdirder;U}
\begin{align*}
    \int_{\R^{2d}}\tr(\nabla_v u)\rho\,d\Leb^{2d} 
    &= \lim_{t\downarrow 0} \frac{\U(\mu)-\U((\br_t)_{\#}\mu)}{t}  \\
    &\leq \limsup_{t\downarrow\infty}\frac{\U(\mu)-\U((\br_t)_{\#}\mu)}{W_{2,v}(\mu,(\br_t)_{\#}\mu)}\frac{t\|u\|_{L^2(\mu)}}{t} \leq |\partial_v\U|(\mu)\|u\|_{L^2(\mu)}.
\end{align*}
As $\|u\|_{L^2(\mu)}\leq \sup|u|$, taking supremum over all $u\in C_c^\infty(\R^d;\R^d)$ we deduce that $\rho$ is a function of bounded variation. Denoting by $D_v\rho$ the distributional derivative of $\rho$ in the $v$-variable, we have,
\[|\langle D_v\rho, u\rangle|\leq |\partial_v\U|(\mu) \|u\|_{L^2(\mu)}\]
hence by $L^2$ duality we can find $w_v\in L^2(\mu;\R^{d})$ such that for all $u\in L^2(\mu;\R^d)$
\[\langle D_v\rho,u\rangle=\langle w_v,u\rangle_{L^2(\mu;\R^d)}=\langle (0,w_v),(0,u)\rangle_{L^2(\mu;\R^{2d})}\]
--i.e. $w_v\rho=\nabla_v\rho$. Furthermore, as
\[|\langle w_v,u\rangle_{L^2(\mu;\R^d)}|=|\langle D_v \rho,u\rangle_{L^2(\R^{d})}|\leq |\partial_v\U|(\mu)\|u\|_{L^2(\mu;\R^d)}\]
we deduce $\|w_v\|_{L^2(\mu;\R^d)}\leq |\partial_v\U|(\mu)$, hence $(0,w_v)\in\partial_v^\circ\U(\mu)$.

\vspace{3mm}
\noindent\emph{Step 2$^o$.} To show the converse, it suffices to show that if $\nabla_v\rho\in L^1(\R^{2d})$ then
\begin{equation}\label{eq:psubdiff;U}
    \U(\nu)-\U(\mu)\geq \int_{\R^{2d}} w_v\cdot (T_{\mu^x}^{\nu^x}-\id_v)\,d\mu \text{ for any } \nu\in\P_2(\R^{2d}).
\end{equation}
Indeed, whenever $\nabla_v\rho\in L^1(\R^{2d})$ and $w_v=\nabla_v\rho/\rho\in L^2(\mu;\R^d)$, \eqref{eq:psubdiff;U} implies by \eqref{eq:psubdiff} that $(0,\nabla_v\rho/\rho)\in\partial_v\U(\mu)$. Thus by Proposition~\ref{prop:slope_subdiff}
\[|\partial_v\U|(\mu)\leq \|\nabla_v\rho/\rho\|_{L^2(\mu)}<+\infty.\]
Note that it suffices to show \eqref{eq:psubdiff;U} for all $\nu\in\P_2^r(\R^{2d})$, as if $\nu\not\ll\Leb^{2d}$ then the inequality is trivially true. Fix $\nu\in\P_2^r(\R^{2d})$ and set $\br^x:=T_{\mu^x}^{\nu^x}$. By \eqref{eq:pdirder;U} and the monotonicity we have
\[\U(\nu)-\U(\mu)\geq -\int_{\R^{2d}}\tr\tilde\nabla_v(\br^x-\id_v)\,d\mu.\]
We will show
\[-\int_{\R^{2d}} \tr\tilde\nabla(\br-\id)\,\mu(x,v) \geq \int_{\R^{2d}} (\br^x-v)\cdot \nabla_v \rho(x,v)\,dx.\]

Let us first assume $\nu\in\P_2^r(\R^{2d})$ is compactly supported in the $v$-variable -- i.e. there exists $K\subset\subset\R^d$ such that $\supp\nu\subset\R^d\times K$. By \cite[Theorem 6.2.9]{AGS} $\tilde\nabla_v\br^x \in BV_{loc}(\R^d)$ and the distributional divergence in $v$ satisfies $D_v\cdot\br^x\geq 0$, whereas $\tr\tilde\nabla_v\rho^x$ is the absolutely continuous part $D_v\cdot\br^x$. As $\supp\nu\subset\R^d\times K$, $\br^x$ is bounded hence by approximation we have
\[
\int_{\R^d} \varphi\tr\tilde\nabla_v\br^x(v)\,dv \leq - \int_{\R^d} \nabla_v \varphi\cdot \br^x\,dv \text{ for any nonnegative } \varphi\in W^{1,1}(\R^d).
\]
In particular, as $\int_{\R^{2d}} |\nabla_v\rho|\,dv\,dx<+\infty$, the map $v\mapsto \rho(x,v)$ is in $W^{1,1}(\R^d)$ for a.e. $x\in\R^d$. Thus
\begin{align*}
\U(\nu)-\U(\mu)&\geq -\int_{\R^{2d}}\tr\tilde\nabla_v(\br^x-\id_v)\,d\mu= -\int_{\R^{2d}}\rho(x,v)\tr\tilde\nabla_v\br^x(v) \,dxdv + d \\
&\geq \int_{\R^{2d}} \nabla_v\rho(x,v)\cdot\br^x(v)\,dv\,dx + d  
= \int_{\R^{2d}} \nabla_v\rho(x,v)\cdot(\br^x(v)-\id_v)\,dv\,dx \\
&= \int_{\R^{2d}} w_v(x,v)\cdot(\br^x(v)-\id_v)\,d\mu(x,v)
\end{align*}
for all $\nu$ compactly supported in the $v$-variable. 

For general $\nu\in\P_2^r(\R^{2d})$, consider a sequence of absolutely continuous $\nu_n\in\P_2^v(\R^{2d};\Pi^x\nu)$ compactly supported in the $v$-variable such that $W_{2,v}(\nu_n,\nu)\rightarrow 0$ and $\U(\nu)=\lim_{n\nearrow \infty}\U(\nu_n)$. Let $\br_n^x=T_{\mu^x}^{\nu_n^x}$ and $\br_n(x,v)=(x,\br_n^x(x,v))$.
Note
\begin{align*}
    \int_{\R^{2d}} w_v\cdot(\br^x-\id_v)\,d\mu
    &=\int_{\R^{2d}} w_v\cdot(\br^x_n-\id_v)\,d\mu+\int_{\R^{2d}} w_v\cdot(\br^x-\br^x_n)\,d\mu \\
    &\leq \U(\nu_n)-\U(\mu) + \|w_v\|_{L^2(\mu)}\|\br-\br_n\|_{L^2(\mu)},
\end{align*}
whereas by Proposition~\ref{prop:stability}, $\|\br-\br_n\|_{L^2(\mu)}\rightarrow 0$ as $n\rightarrow\infty$. Thus we may pass the limit $n\rightarrow\infty$ to conclude \eqref{eq:psubdiff;U}.
\end{proof}

\section{An alternative proof of compactness of the discrete solutions}\label{app:cpct_alt}
This section provides the proof of Theorem~\ref{thm:disc_sol;subseq_conv}. We first show the uniform equicontinuity of discrete solutions. 
\begin{proposition}[Uniform equicontinuity of discrete solutions]\label{prop:equicts;discrete_sol}
Let $\nabla_x V,\nabla_x W$ be $M$-Lipschitz, and let $(\mu_{ih}^N)_{i=1}^N$ be the solution to the variational scheme \eqref{eq:SIE} with time step $h>0$ and initial datum $\mu_0^N=\mu_0$. Then
\begin{equation}\label{eq:partial_grad;ubd}
\begin{split}
\frac{1}{1+\alpha h}|\partial_v(\Lfun_v+\alpha\H)|(\mu_{ih}^N) + \alpha\|\id_v\|_{L^2(\mu_{(i+1)h}^N)} &\leq \frac{1}{1+\alpha h}|\partial_v(\Lfun_v+\alpha\H)|(\mu_{(i-1)h}^N)\\
&+(\alpha+3Mh)\|\id_v\|_{L^2(\mu_{ih}^N)}.
\end{split}
\end{equation}
In particular,
\begin{equation}\label{eq:W2discrete;ubd}
\begin{split}
   \frac1h W_2(\mu_{(i+1)h}^N,\mu_{ih}^N)&\leq \frac{1}{1+\alpha h}|\partial_v(\Lfun_v+\alpha\H)|(\mu_{ih}^N) + \|\id_v\|_{L^2(\bar\mu_{(i+1)h}^N)}  \\
   & \leq \max\{\alpha^{-1},1\}(1+3Mh)^{i}(|\partial_v(\Lfun_v+\alpha\H)|(\mu_{0})+\alpha\|\id_v\|_{L^2(\mu_0)}).
\end{split}
\end{equation}
\end{proposition}
\begin{proof}
Let $\gamma_i\in \Gamma_o^{v}(\mu_{ih}^N,\bar\mu_{(i+1)h}^N)$. Then by the $\alpha$-convexity of $\Lfun_v+\alpha\H$ along $W_{2,v}$ geodesics, the slope estimates \eqref{eq:partial_slope_est} yield
\begin{align*}
    &\|\id_v\|_{L^2(\mu_{(i+1)h}^N)}-\|\id_v\|_{L^2(\mu_{ih}^N)}=\|\id_v\|_{L^2(\bar\mu_{(i+1)h}^N)}-\|\id_v\|_{L^2(\mu_{ih}^N)}\\
    &\leq \left(\int_{\R^{2d}} |w-v|^2\,d\gamma_i(w,v,x)\right)^{1/2}
    = W_{2,v}(\bar\mu_{(i+1)h}^N,\mu_{ih}^N)
    \leq  \frac{h}{1+\alpha h}|\partial_v(\Lfun_v+\alpha\H)|(\mu_{ih}^N).
\end{align*}
Moreover, by \eqref{eq:dvE_xstep;ubd} and that $\Pi^v\mu_{ih}^N=\Pi^v\mu_{ih}^N$),
\begin{align*}
    |\partial_v(\Lfun_v+\alpha\H)|(\mu_{ih}^N)\leq |\partial_v(\Lfun_v+\alpha\H)|(\bar\mu_{ih}^N)+3Mh\|\id_v\|_{L^2(\mu_{ih}^N)}.
\end{align*}
As $\Lfun_v+\alpha\H$ is $\alpha$-partial geodesically convex in $v$, again Proposition~\ref{prop:pslope_est} yields
\begin{align*}
    |\partial_v (\Lfun_v+\alpha\H)|(\bar\mu_{(i+1)h}^N) \leq \frac1h W_{2,v}(\mu_{ih}^N,\bar\mu_{(i+1)h}^N)\leq \frac{1}{1+\alpha h}|\partial_v(\Lfun_v+\alpha\H)|(\mu_{ih}^N).
\end{align*}
Then \eqref{eq:partial_grad;ubd} follows directly from above three estimates, as
\begin{align*}
    \frac{1}{1+\alpha h}|\partial_v(\Lfun_v+\alpha\H)|(\mu_{ih}^N) &+ \alpha\|\id_v\|_{L^2(\mu_{(i+1)h}^N)}  \\
    &\leq |\partial_v(\Lfun_v+\alpha\H)|(\mu_{ih}^N) + \alpha\|\id_v\|_{L^2(\mu_{ih}^N)}    \\
    &\leq |\partial_v(\Lfun_v+\alpha\H)|(\bar\mu_{ih}^N)+\left(\alpha+3Mh\right)\|\id_v\|_{L^2(\mu_{ih}^N)} \\
    &\leq \frac{1}{1+\alpha h}|\partial_v(\Lfun_v+\alpha\H)|(\mu_{(i-1)h}^N)+(\alpha+3Mh)\|\id_v\|_{L^2(\mu_{ih}^N)}.
\end{align*}

Now we turn to \eqref{eq:W2discrete;ubd}. Note
\[W_{2,x}(\bar\mu_{(i+1)h}^N,\mu_{(i+1)h}^N)= h\|\id_v\|_{L^2(\bar\mu_{(i+1)h}^N)}=h\|\id_v\|_{L^2(\mu_{(i+1)h}^N)}\]
whereas the slope estimates (Proposition~\ref{prop:pslope_est}) with $\alpha$-geodesic convexity of $\Lfun_v+\alpha\H$ yield
\begin{equation}\label{eq:W2v_leq_slope}
    W_{2,v}(\mu_{ih}^N,\bar\mu_{(i+1)h}^N)\leq \frac{h}{1+\alpha h} |\partial_v(\Lfun_v+\alpha\H)|(\mu_{ih}^N).
\end{equation}
Thus
\begin{align*}
    W_2(\mu_{ih}^N,\mu_{(i+1)h}^N)&\leq W_2(\mu_{ih}^N,\bar\mu_{(i+1)h}^N)+W_2(\bar\mu_{(i+1)h}^N,\mu_{(i+1)h}^N) \\
    &\leq W_{2,v}(\mu_{ih}^N,\bar\mu_{(i+1)h}^N) + W_{2,x}(\bar\mu_{(i+1)h}^N,\mu_{(i+1)h}^N)  \\
    &\leq \frac{h}{1+\alpha h} |\partial_v(\Lfun_v+\alpha\H)|(\mu_{ih}^N) + h\|\id_v\|_{L^2(\mu_{(i+1)h}^N)} \\
    &\leq \max\{\alpha^{-1},1\} h \left(\frac{1}{1+\alpha h}|\partial_v(\Lfun_v+\alpha\H)|(\mu_{ih}^N)+\alpha\|\id_v\|_{L^2(\mu_{(i+1)h}^N)}\right)
\end{align*}

Replacing \eqref{eq:partial_grad;ubd} by the crude bound
\begin{align*}
\frac{1}{1+\alpha h}|\partial_v(\Lfun_v+\alpha\H)|(\mu_{ih}^N) &+ \alpha\|\id_v\|_{L^2(\mu_{ih}^N)}\\
&\leq (1+3Mh)\left(\frac{1}{1+\alpha h}|\partial_v(\Lfun_v+\alpha\H)|(\mu_{(i-1)h}^N)+\alpha\|\id_v\|_{L^2(\mu_{ih}^N)}\right)
\end{align*}
we deduce
\begin{align*}
\frac1h W_2(\mu_{ih}^N,\mu_{(i+1)h}^N)&\leq \max\{\alpha^{-1},1\} \left(\frac{1}{1+\alpha h}|\partial_v(\Lfun_v+\alpha\H)|(\mu_{ih}^N)+\alpha\|\id_v\|_{L^2(\mu_{(i+1)h}^N)}\right)\\
& \leq \max\{\alpha^{-1},1\} (1+3Mh)^{i} \left(\frac{1}{1+ \alpha h}|\partial_v(\Lfun_v+\alpha\H)|(\mu_{0})+\alpha\|\id_v\|_{L^2(\mu_{h}^N)}\right)   \\
&\leq \max\{\alpha^{-1},1\} (1+3Mh)^{i}(|\partial_v(\Lfun_v+\alpha\H)|(\mu_{0})+\alpha\|\id_v\|_{L^2(\mu_0)}),
\end{align*}
where in the last line we have used $\|\id_v\|_{L^2(\mu_{h}^N)}\leq\|\id_v\|_{L^2(\mu_{0})}+\frac{h}{1+\alpha h}|\partial_v(\Lfun_v+\alpha\H)|(\mu_0)$.
        
\end{proof}

Now we deduce compactness of discrete solutions combining the equicontinuity estimates and the weak Arzel\`a-Ascoli Theorem.

\begin{proof}[Proof of Theorem~\ref{thm:disc_sol;subseq_conv}]
    By the Arzel\`a-Ascoli theorem (Proposition~\ref{prop:ArzelaAscoli-weak}), it suffices to show that there exists a narrowly compact set $K\subset\P_2(\R^{2d})$ such that $\mu_t^N\in K$ for all $t\in[0,T]$ and $N\in\N$, and the equicontinuity 
    \begin{equation}\label{eq:eqcts;discsol}
        \limsup_{N\rightarrow\infty}W_2(\mu_s^N,\mu_t^N)\leq C|s-t| \text{ for all } s,t\in[0,T] 
    \end{equation}
    for some constant $C>0$ independent of $N$.

    Letting $i(N,t)=\lceil Nt/T-1\rceil$, \eqref{eq:W2discrete;ubd} allows us to obtain
    \begin{align*}
    W_2(\mu_0^N,\mu_t^N)&\leq \sum_{j=1}^{i(N,t)}W_2(\mu_{(j-1)h_N}^N,\mu_{jh_N}^N)\leq \sum_{j=1}^N W_2(\mu_{(j-1)h_N}^N,\mu_{jh_N}^N) \\
    &\leq Nh_N (1+3Mh_N)^N\max\{\alpha^{-1},1\}(|\partial_v(\Lfun_v+\alpha\H)|(\mu_0)+\alpha \|\id_v\|_{L^2(\mu_0)}) \\
    &= T e^{2MT}\max\{\alpha^{-1},1\}(|\partial_v(\Lfun_v+\alpha\H)|(\mu_0^N)+\alpha\|\id_v\|_{L^2(\mu_0^N)}).
    \end{align*}
    By assumption \eqref{ass:mu0N} we can find $\mu_0^\ast\in\P_2(\R^{2d})$ and sufficiently large $R>0$ such that 
    \[W_2(\mu_0^\ast,\mu_0^N)\leq R \text{ for all } N\in\N.\]
    The assumption on the initial data further implies
    \[C_0:=\max\{\alpha^{-1},1\}\sup_{N\in\N}(|\partial_v(\Lfun_v+\alpha\H)|(\mu_0^N)+\alpha\|\id_v\|_{L^2(\mu_0^N)})<+\infty.\]
    Thus
    \[K:=\{\nu\in\P_2(\R^{2d}): W_2(\mu_0^*,\nu)\leq R+T e^{2MT}C_0\}\]
    which is narrowly compact; see for instance \cite[Proposition 7.1.5]{AGS}.

    On the other hand, using \eqref{eq:W2discrete;ubd} again observe that for any $s<t$
    \begin{align*}
        W_2(\mu_s^N,\mu_t^N)= W_2(\mu_{i(N,s)}^N,\mu_{i(N,t)}^N)
        \leq \sum_{j=i(N,s)+1}^{i(N,t)} W_2(\mu_{(j-1)h}^N,\mu_{jh}^N)
        \leq h_N(i(N,t)-i(N,s)) e^{2MT}C_0 \\
        \leq \frac{T}{N}\left(\frac{N(t-s)}{T}+1\right)e^{2MT}C_0
        \leq \left((t-s)+\frac{T}{N}\right)e^{3MTC_0}.
    \end{align*}
    Thus by taking $\limsup_{N\rightarrow\infty}$ on both sides we have \eqref{eq:eqcts;discsol} with $C=e^{3MTC_0}$. Moreover, lower semicontinuity of $W_2$ with respect to the narrow convergence ensures that the limiting curve is $e^{3MTC_0}$-Lipschitz.
\end{proof}

\end{appendix}

\end{document}